\documentclass[11pt]{amsart}
\usepackage{amsmath}
\usepackage{amssymb, verbatim}
\usepackage{pstricks,pst-node,pst-plot}
\usepackage{comment}
\usepackage{color}
\usepackage{extarrows}

 \title[]{Stability of the tangent bundle through conifold transitions}
\author[T. C. Collins]{Tristan C. Collins}
  \email{tristanc@mit.edu}
  \address{Department of Mathematics, Massachusetts Institute of Technology, 77 Massachusetts Avenue, Cambridge, MA 02139}
\thanks{T.C.C is supported in part by NSF grant DMS-1810924, NSF CAREER grant DMS-1944952 and an Alfred P. Sloan Fellowship. } 

 \author[S. Picard]{Sebastien Picard}
  \email{spicard@math.ubc.ca}
  \address{Department of Mathematics, The University of British
    Columbia, 1984 Mathematics Road, Vancouver BC Canada V6T 1Z2}
  \thanks{}
  
 \author[S.-T. Yau]{Shing-Tung Yau}
  \email{yau@math.harvard.edu}
  \address{Department of Mathematics, Harvard University, 1 Oxford St., Cambridge, MA, 02138}
  \thanks{} 

\theoremstyle{plain}
\newtheorem{thm}{Theorem}[section]
\newtheorem{prop}[thm]{Proposition}
\newtheorem{defn}[thm]{Definition}
\newtheorem{lem}[thm]{Lemma}
\newtheorem{cor}[thm]{Corollary}

\theoremstyle{definition}

\newtheorem{rk}[thm]{Remark}

\numberwithin{equation}{section}

\newcommand{\Tr}{\textrm{Tr}\,}
\newcommand{\vol}{\textrm{vol}}

\newcommand{\del}{\partial}
\newcommand{\de}{\partial}

\newcommand{\dvol}{d{\rm vol}}

\newcommand{\dbar}{\overline{\del}}
\newcommand{\ddb}{\sqrt{-1}\del\dbar}

\newcommand{\Ric}{\mathrm{Ric}}

\newcommand{\F}{\mathcal{F}}
\newcommand{\LL}{\mathcal{L}}

\newcommand{\HH}{\mathcal{H}}

\newcommand{\E}{\mathcal{E}}
\newcommand{\Q}{\mathcal{Q}}

\newcommand{\be}{\begin{equation}}
\newcommand{\bea}{\begin{eqnarray}}
\newcommand{\eea}{\end{eqnarray}} 
\newcommand{\ee}{\end{equation}}

\renewcommand{\leq}{\leqslant}
\renewcommand{\geq}{\geqslant}

\renewcommand{\epsilon}{\varepsilon}
\renewcommand{\phi}{\varphi}

\begin{document}
\maketitle

\begin{abstract}
Let $X$ be a compact, K\"ahler, Calabi-Yau threefold and suppose $X\mapsto \underline{X}\leadsto X_t$ , for $t\in \Delta$, is a conifold transition obtained by contracting finitely many disjoint $(-1,-1)$ curves in $X$ and then smoothing the resulting ordinary double point singularities.  We show that, for $|t|\ll1$ sufficiently small, the tangent bundle $T^{1,0}X_{t}$ admits a Hermitian-Yang-Mills metric $H_t$ with respect to the conformally balanced metrics constructed by Fu-Li-Yau.  Furthermore, we describe the behavior of $H_t$ near the vanishing cycles of $X_t$ as $t\rightarrow 0$.
\end{abstract}

\section{Introduction}

Let $X$ be a compact, K\"ahler, Calabi-Yau threefold with trivial canonical bundle. Around 1985 Clemens described a general procedure for constructing new, possibly non-K\"ahler complex manifolds with trivial canonical bundle by contracting a collection of disjoint $(-1,-1)$ curves and then smoothing the resulting ordinary double point (ODP) singularities.  Such a geometric transition is now called a conifold transition and we denote it by $X\mapsto \underline{X} \leadsto X_{t}$, where $\underline{X}$ is a singular variety with ODP singularities.  Reid's fantasy conjectures \cite{Reid} that all complex threefolds with trivial canonical bundle can be connected by a sequence of conifold transitions.  The goal of this paper, motivated in part by equations from heterotic string theory, is to show that the tangent bundle $T^{1,0}X_{t}$ admits a Hermitian-Yang-Mills metrics $H_t$ with respect to a class of balanced metrics constructed by Fu-Li-Yau \cite{FLY}, and to study the geometry of these metrics as $|t|\rightarrow 0$. In order to put our work in context, we recall the origins of Calabi-Yau geometry in theoretical physics. 

If $X$ is a compact K\"ahler manifold with $c_1(X)=0$ then the third author's solution of the Calabi conjecture \cite{Yau78} implies the existence of a unique Ricci-flat K\"ahler metric in any K\"ahler class on $X$.  This result, which can be viewed as a higher dimensional analog of the Uniformization Theorem, yields a plethora of examples of compact Riemannian manifolds with zero Ricci curvature and holonomy contained in $SU(n)$.

Following the solution of the Calabi conjecture, Candelas-Horowitz-Strominger-Witten \cite{CHSW} showed that compact K\"ahler manifolds with holonomy $SU(3)$, in particular, Calabi-Yau manifolds of complex dimension $3$, are fundamental building blocks in torsion-free superstring compactifications.  Precisely, \cite{CHSW} constructed superstring compactifications with the Standard Model gauge group from a Calabi-Yau threefold together with a holomorphic vector bundle $E\rightarrow X$ admitting a Hermitian-Yang-Mills connection and satisfying the topological constraints $c_1(E)=0$, and $c_2(E)=c_2(X)$.  Of course, the natural choice to make is $E=T^{1,0}X$, but more generally the Donaldson-Uhlenbeck-Yau theorem \cite{Donaldson, UY} implies that any slope stable vector bundle satisfying the topological constraints is admissible.  In total, these works lead to an abundance of a priori distinct superstring compactifications.

Shortly thereafter, the possibility of superstring compactifications with torsion was investigated. In this case, the compactifying manifold is a complex threefold $X$ with a non-vanishing holomorphic $(3,0)$-form $\Omega$ (so that canonical bundle is trivial) equipped with a holomorphic vector bundle $E\rightarrow X$ satisfying the topological constraints $c_1(E)=0=c_1(X)$ and $c_2(E)=c_2(X)$.  In order for this data to give rise to a supersymmetric compactification, $X$ must admit a hermitian metric $g$, with associated $(1,1)$-form $\omega$, and $E$ must admit a hermitian metric $H$ solving the following system of equations, called the Strominger system \cite{Strom}. The first equation is is formulated as in \cite{LY05}.
\be \label{eq: introHSbal}
d(\|\Omega\|_{\omega} \, \omega^2) =0,
\ee
\be\label{eq: introHSYM}
\omega^2 \wedge  F_{H}=0,
\ee
\be\label{eq: introAnom}
\ddb \omega - \frac{\alpha'}{4}\left(\Tr Rm_g\wedge Rm_g  - \Tr F_{H}\wedge F_{H}\right)=0.
\ee
Here $\alpha'>0$ is the inverse string tension, $F_{H}$ denotes the curvature of the Chern connection of $(E,H)$, and $\|\Omega\|^2_{\omega}$ is the norm of $\Omega$ with respect to $g$.  It is natural to view $Rm$ as the curvature of the Chern connection on $T^{1,0}X$, though other choices of connection are admissible as proposed by Hull \cite{Hull} and further discussed in e.g. \cite{dlOSvanes,FIUV,MGF}. Note that if $g$ is K\"ahler, so that $d\omega=0$, then~\eqref{eq: introHSbal} implies that $\omega$ solves the Monge-Amp\`ere equation and is hence Ricci-flat, while ~\eqref{eq: introHSYM} is the Hermitian-Yang-Mills equation for $H$, with respect to $\omega$.  Finally,  in the K\"ahler case, the anomaly cancellation equation~\eqref{eq: introAnom} dictates that
\[
\Tr Rm\wedge Rm  - \Tr F_{H}\wedge F_{H}=0.
\]
This coupling between the Hermitian-Yang-Mills equation and the Calabi-Yau equation is highly non-trivial, but it is automatically satisfied provided $E=T^{1,0}X$.  Thus, we can view the system~\eqref{eq: introHSbal},~\eqref{eq: introHSYM},~\eqref{eq: introAnom} as a generalization of the Calabi-Yau equation to the setting of non-K\"ahler complex manifolds.  In particular, the Strominger system provides a set of equations for uniformizing non-K\"ahler complex threefolds with trivial canonical bundle which can be viewed as natural generalizations of the Calabi-Yau equation.  

This system from heterotic string theory has recently generated a great deal of interest in mathematics, both for its applications to the study of non-K\"ahler Calabi-Yau manifolds and its connections to theoretical physics.  Li-Yau \cite{LY05} constructed solutions on K\"ahler Calabi-Yau threefolds by deforming the complex structure of $T^{1,0}X\oplus \mathbb{C}^r$.  Deformations of K\"ahler solutions to more general bundles were considered by Andreas-Garcia-Fernandez \cite{AndGF, AndGF2}.  Fu-Yau \cite{FuYau} constructed solutions to the Strominger system on Calabi-Eckmann-Goldstein-Prokushkin fibrations by using a certain ansatz to reduce the system to a non-linear PDE of Monge-Amp\`ere type on a $K3$ surface.  Higher dimensional versions of the Fu-Yau construction have recently been considered in \cite{CHZ,FVG,Huang,PPZ1, PPZ3, PPZ4,PPZ8}. Further geometric constructions of solutions, in both the compact and non-compact cases, can be found in e.g. \cite{Fei1, Fei2, FHP, FeiYau, FIUV, FIUVil, FTY,MGF20,OUV}. Very recently, Phong, Zhang and the second author have introduced \cite{PPZ2} a parabolic approach via the Anomaly flow and obtained a new proof of the Fu-Yau result; see \cite{BV,FeiPhong,FPPZ, PPZ7, PPZ6, PPZ5} and the references therein.  We refer the reader to \cite{MGF,GRT,GRST,PhongSurv,TsengYau} and the references therein for more on this very active area of research.

While the plethora of solutions to (\ref{eq: introHSbal}), (\ref{eq: introHSYM}), (\ref{eq: introAnom}) is interesting from a mathematical point of view, the lack of a unique vacuum for the heterotic string is a fundamental problem for the predictive power of string theory.  A conjectural resolution of this problem was put forth by Reid \cite{Reid}, inspired by work of Clemens and Friedman.  {\em Reid's fantasy} proposes that all complex $3$-folds with trivial canonical bundle are connected by a sequence of contractions and smoothings.  Recall that a conifold transition 
\be\label{eq: introConTran}
X \rightarrow \underline{X}\leadsto X_t,
\ee
consists of a contraction followed by a smoothing, where the contraction map $X\rightarrow \underline{X}$ contracts a collection of disjoint rational curves $C_i\subset X$, with normal bundle $\mathcal{O}_{\mathbb{P}^1}(-1)^{\oplus 2}$ (called $(-1,-1)$ curves) to ordinary double point (ODP) singularities, given locally by equations
\be\label{eq: introConifold}
\left\{\sum_{i=1}^4 z_{i}^2=0\right\} \subset \mathbb{C}^4.
\ee
By work of Friedman \cite{Fried}, under appropriate assumptions there is a smoothing
\be\label{eq: introSmoothing}
\mu:\mathcal{X}\rightarrow \Delta, \qquad \Delta = \{ t\in \mathbb{C} : |t|<1\}
\ee
such that $\mu^{-1}(0)= \underline{X}$, and $\mu^{-1}(t)=X_{t}$ is a smooth complex threefold with trivial canonical bundle; see Section~\ref{sec: geoConifold} for a more thorough discussion of conifold transitions.  Locally near the ODP singularities, this smoothing is given by
\be\label{eq: introModelSmoothing}
\left\{\sum_{i=1}^4 z_i^2=t \right\} \subset \mathbb{C}^4\times\mathbb{C}.
\ee
Note that in general, $X_{t}$ will no longer be K\"ahler (even topologically), even if the initial manifold $X$ was projective.  These transitions were exploited by Clemens to construct many new examples of compact, complex threefolds with trivial canonical bundle; see \cite{Fried}. For an introduction to Calabi-Yau transitions, see \cite{Rossi}.

Green-H\"ubsch \cite{GreenHub, GreenHub1} and Candelas-Green-H\"ubsch \cite{CGH, CGH1} argued that conifold transitions could be used to connect any two Calabi-Yau manifolds realized as complete intersections in products of projective spaces. As we move along in the moduli of string vacua, string physics should smoothly interpolate through topological changes. For Type II strings, this problem was studied in \cite{GMS,Strom95}. Here we consider heterotic strings, and in order to resolve the vacuum degeneracy problem for these compactifications it is essential to understand the solvability of the system (\ref{eq: introHSbal}), (\ref{eq: introHSYM}), (\ref{eq: introAnom}) through conifold transitions.  In fact, the third author has advocated that the solvability of the Strominger system may provide a useful tool for studying Reid's fantasy as it provides a uniformization of non-K\"ahler Calabi-Yau threefolds.  

The study of the Strominger system through conifold transitions was initiated by Fu-Li-Yau \cite{FLY} who established the existence of metrics $\omega_{t}$ on $X_{t}$ solving~\eqref{eq: introHSbal} assuming the input manifold $X$ in~\eqref{eq: introConTran} is a compact, K\"ahler Calabi-Yau.  Chuan \cite{Chuan} showed that if $E\rightarrow X$ is a holomorphic vector bundle which is Mumford-Takemoto stable with respect to some K\"ahler class, holomorphically trivial in a neighborhood of the curves contracted by the map $X\rightarrow \underline{X}$, and $(\underline{X},E)$ can be smoothed to a family of holomorphic bundles $(X_t,E_t)$, then stability can be passed through the transition in the following sense: there is a hermitian metric $H_{t}$ on $E_t$ solving the Hermitian-Yang-Mills equation~\eqref{eq: introHSYM} with respect to the Fu-Li-Yau metric.  We remark that it is unclear whether such bundles $E\rightarrow X$ can be constructed so that, $c_2(E)=c_2(X)$ in $H^{2}(X,\mathbb{R})$ and in addition $c_2(E_t)= c_2(X_t)$ in $H^4(X_{t},\mathbb{R})$ after the conifold transition.  Such a situation would be necessary in order to pass solutions of the full Strominger system through conifold transitions.  In any event, the metric $H_t$ constructed by Chuan is approximately flat in a neighborhood of the vanishing cycles of the smooth $\underline{X}\leadsto X_{t}$, while the Fu-Li-Yau metric is modeled on a non-flat, K\"ahler Calabi-Yau cone metric.  

In this work we initiate the study of the Strominger system through conifold transitions in the case when the gauge bundle $E$ is taken to be $T^{1,0}X$.  Our main theorem is the following:

\begin{thm}\label{thm: introMain}
Let $X \rightarrow \underline{X}\leadsto X_t$ be a conifold transition, with $X_t$ as in~\eqref{eq: introSmoothing}. Equip $X_{t}$ with the Fu-Li-Yau balanced metric $g_{t}$ with associated $(1,1)$ form $\omega_{t}$.  Then, for all $|t|\ll 1$ sufficiently small there exists a hermitian metric $H_{t}$ on $T^{1,0}X_{t}$ solving 
\[
\omega_{t}^2\wedge F_{H_{t}} =0.
\]
Furthermore, there exists $\lambda>0$ such that if $p_i \in \underline{X}$ is a ODP singularity, then after identifying a neighborhood of $p_i \in \mathcal{X}$ with the model smoothing~\eqref{eq: introModelSmoothing}, there are constant $c_i,d_i>0$  such that on a neighborhood of the vanishing cycles given by
\[
\mathcal{R}_{\lambda} = \{ \|z\|^2 \leq |t|^{\frac{3}{3+\lambda}}\} \subset X_t,
\]
for each $k\in \mathbb{Z}_{\geq 0}$ there is a constant $C_k$ such that 
\[
| \nabla^k_{g_{co,t}} (g_{t}-c_i g_{co,t}) |_{g_{co,t}} \leq C_k |t|^\lambda \|z\|^{-{2 \over 3}k},
\]
\[
| \nabla^k_{g_{co,t}} (H_t-d_i g_{co,t}) |_{g_{co,t}} \leq C_k |t|^\lambda \|z\|^{-{2 \over 3} k}.
\]
Here $\|z\|^2= \sum_{i=1}^4 |z_i|^2$ and $g_{co,t}$ is the explicit K\"ahler Ricci-flat metric on the smoothing ~\eqref{eq: introModelSmoothing} constructed by Candelas-de la Ossa.
\end{thm}


\begin{rk}
Since both $g_t$ and $H_t$ are locally modeled on the Candelas-de la Ossa explicit K\"ahler Ricci-flat metric $g_{co,t}$ near the vanishing cycles, we see that the local metric description of conifold transitions given by Candelas-de la Ossa \cite{CO} accurately describes global non-K\"ahler conifold transitions of heterotic strings near the vanishing cycles. Our estimates give convergence, after a suitable local rescaling, of the pair $(g_t, H_t)$ to a solution of the anomaly cancellation equation (\ref{eq: introAnom}) near the ordinary double points of $\underline{X}$ as $t\rightarrow 0$. For related work on Calabi-Yau metrics ($g_t=H_t$) in the case when both sides are K\"ahler, see e.g. \cite{HeinSun,RoZh,Song,Tosatti09} and references therein. 
\end{rk}

  \begin{rk} A similar result holds if we replace the balanced Fu-Li-Yau metric $g_{t}$ with the conformally balanced metric $\check{g}_t$ obtained by conformally rescaling $g_{t}$.  In particular, the pair $(\check{g}_t, H_{t})$ simultaneously solves ~\eqref{eq: introHSbal} and~\eqref{eq: introHSYM}, and satisfies an estimate similar to Theorem \ref{thm: introMain}; see Remark~\ref{rk: confRescale}.  This implies that near the ODP singularities, at a suitable scale, the pair $(\check{g}_t, H_{t})$ converges to the Calabi-Yau solution of the Strominger system on the conifold as $|t|\rightarrow 0$. 
\end{rk}

\begin{rk}
The existence of a metric $H_t$ on $T^{1,0}X_t$ solving $\omega_t^2 \wedge F_{H_t} = 0$ implies that the tangent bundle $T^{1,0}X_t$ is stable with respect to the Fu-Li-Yau balanced class $[\omega_t]^2$. In the case when $X_t$ is topologically $\#_{k}S^3 \times S^3$ \cite{Fried,LuTian}, the stability of the tangent bundle was noted in \cite{Boz}.
\end{rk}

The third author has conjectured \cite{Yau10} that if $X$ is any complex threefold with trivial canonical bundle admitting a pair of metrics $(\omega, H)$ solving the conformally balanced equation ~\eqref{eq: introHSbal} and the Hermitian-Yang-Mills equation~\eqref{eq: introHSYM}, then there is a solution of the full Strominger system.  We hope to return to this problem, in the setting of conifold transitions, in future work.
 
 The outline of this paper is as follows.  In Section~\ref{sec: geoConifold} we discuss some background material, including the basic geometric properties of conifold transitions that will be important for our work.  In Section~\ref{sec: HYM-central} we construct a metric $H_0$ on the tangent bundle of $T\underline{X} \rightarrow \underline{X}_{reg}$.  More precisely, the metric $H_0$ is Hermitian-Yang-Mills with respect to a smooth, balanced metric $g_0$ on $\underline{X}_{reg}$ and, near the singular points of $\underline{X}$, is uniformly equivalent to the Candelas-de la Ossa Ricci-flat K\"ahler metric on the conifold~\eqref{eq: introConifold}.  Furthermore, we show that $H_0$ also satisfies scale invariant higher-order estimates.  The metric $H_0$ serves as the model metric for the Hermitian-Yang-Mills metric on $X_t$, at least away from the vanishing cycles.  In Section~\ref{sec: quant-decay} we establish quantitative, polynomial decay of $H_0$ towards a multiple of the Candelas-de la Ossa metric on the conifold. In Section~\ref{sec: approx-hym} we use $H_0$ to construct an approximately Hermitian-Yang-Mills metric $H_t$ on $X_t$ with an explicit estimate for the decay rate, with respect to $|t|$, towards a Hermitian-Yang-Mills metric.  Finally, in Section~\ref{sec: perturbation} we show that, for $|t|\ll1$ sufficiently small, $H_t$ can be perturbed to a genuine Hermitian-Yang-Mills metric on $T^{1,0}X_{t}$.  For the reader's convenience we have provided an appendix detailing the aspects of the Fu-Li-Yau construction which are important for our work.
\bigskip

{\bf Acknowledgements:} We thank M. Garcia-Fernandez for helpful comments. T.C.C. is grateful to A. Jacob for helpful discussions concerning \cite{JacobWalpuski, JacobEarpWalpuski}
  
\section{The geometry of conifold transitions}\label{sec: geoConifold}

In this section we will discuss the basic geometry of conifold transitions.  We begin with the following definition, which fixes the notion of Calabi-Yau manifold to be considered in this paper.

\begin{defn}\label{defn: CY}
A smooth Calabi-Yau threefold is a smooth complex threefold with finite fundamental group and trivial canonical bundle
\end{defn}

The primary aim of this paper is to understand the solvability of the Hermitian-Yang-Mills equation on the tangent bundle to a Calabi-Yau threefold as it passes through a conifold transition.

\begin{defn}
 A $(-1,-1)$ curve $C\subset X$ is a smooth rational curve $C\simeq \mathbb{P}^1$ such that the normal bundle $N_{C/X} \simeq \mathcal{O}_{\mathbb{P}^1}(-1)^{\oplus 2}$.
 \end{defn}
 
From \cite[Satz 7]{Grauert} there is an open neighborhood $U$ of $C$ in $X$ such that $U$ is biholomorphic to a neighbourhood of the zero section in the total space of $\mathcal{O}_{\mathbb{P}^1}(-1)\oplus \mathcal{O}_{\mathbb{P}^1}(-1)$.  In particular, this implies the existence of a contraction map
 \[
 \pi_{C}:X\rightarrow \underline{X}
 \]
 to a singular complex space $\underline{X}$ with an ordinary double point singularity at $p$ such that $\pi_{C}: X\setminus C \rightarrow \underline{X}\setminus\{p\}$ and $\pi_{C}(C)=p$.  Concretely, if $[X_1:X_2]$ denote homogeneous coordinates on $\mathbb{P}^1$, then any non-zero point in the total space $\mathcal{O}_{\mathbb{P}^1}(-1)^{\oplus 2}$ can be written uniquely as $(w_1X_1,w_1X_2, w_2X_1,w_2X_2)$.  This defines a biholomorphism from the complement of the zero section in $\mathcal{O}_{\mathbb{P}^1}(-1)^{\oplus 2}$ to the complement of the origin in the conifold
 \be\label{eq: sec2ConifoldHat}
\hat{V}_0:= \{\hat{z}_1\hat{z}_2-\hat{z}_3\hat{z}_4=0\} \subset \mathbb{C}^4.
 \ee
 This map can clearly be extended holomorphically over the zero section of $\mathcal{O}_{\mathbb{P}^1}(-1)^{\oplus 2}$ by sending $\mathbb{P}^1$ to the origin in $\mathbb{C}^4$.  After a unitary change of coordinates we can rewrite~\eqref{eq: sec2ConifoldHat} as the standard conifold
  \be\label{eq: sec2Conifold}
V_0:= \{z_1^2+z_2^2+z_3^2+z_4^2=0\} \subset \mathbb{C}^4.
 \ee
 
 We now describe another realization of the affine variety~\eqref{eq: sec2Conifold}.  Consider the Fano surface $\mathbb{P}^1\times \mathbb{P}^1$.  Denote by
 \[
 p_{i}: \mathbb{P}^1\times \mathbb{P}^1 \rightarrow \mathbb{P}^1 \quad \text{ for } i=1,2
 \]
 the projection onto the $i$-th factor.  The conifold can be realized as the blow-down of the zero section in the total space of $p_1^*\mathcal{O}_{\mathbb{P}^1}(-1)\otimes p_2^*\mathcal{O}_{\mathbb{P}^1}(-1)$. Explicitly, the global holomorphic sections of $p_1^*\mathcal{O}_{\mathbb{P}^1}(1)\otimes p_2^*\mathcal{O}_{\mathbb{P}^1}(1)$ define an embedding 
 \[
 \iota: \mathbb{P}^1\times \mathbb{P}^1 \hookrightarrow \mathbb{P}^3
 \]
 which is precisely the Segre embedding
 \[
\iota( \mathbb{P}^1\times \mathbb{P}^1)= \{ X_1X_2-X_3X_4=0\,\big|\, [X_1:X_2:X_3:X_4 ]\in \mathbb{P}^3\}.
 \]
Taking the cone over this projective variety yields~\eqref{eq: sec2ConifoldHat}.  

The singular affine variety $V_0$ given in~\eqref{eq: sec2Conifold} admits an explicit smoothing given by
\be \label{eq: modelSmoothing}
\mathcal{V} = \{z_1^2+z_2^2+z_3^2+z_4^2=t\} \subset \mathbb{C}^4\times \mathbb{C}.
\ee
We denote by $\mu:\mathcal{V}\rightarrow \mathbb{C}$ the projection $\mu(z, t) = t$, and let $V_{t} = \mu^{-1}(t)$ be the fiber over $t$.

Now suppose that $X$ is a Calabi-Yau threefold and let $C_{1},\ldots, C_{k}\subset X$ be a collection of disjoint $(-1,-1)$ curves. Let $\pi: X\rightarrow \underline{X}$ be the map contracting the $C_i$, so that $\underline{X}$ is a compact complex space with ordinary double point singularities at $p_i= \pi(C_i)$.  We have the following well-known result of Friedman \cite{Fried1, Fried}

\begin{thm}[Friedman, \cite{Fried}]\label{thm: Friedman}
There is a first order deformation of $\underline{X}$ smoothing $p_i$ if and only if there is a relation
\be\label{eq: Friedman}
\sum_{i=1}^k \lambda_i [C_i] =0 \, \text{ in } H_{2}(X,\mathbb{R})
\ee
where each $\lambda_i \ne 0$.
\end{thm}

When $X$ is K\"ahler (or more generally, satisfies the $\ddb$-lemma) Kawamata \cite{Kaw}, building on work of Ran \cite{ZRan}, and independently Tian \cite{Tian} showed that the first order smoothings in Theorem~\ref{thm: Friedman} integrate to genuine smoothings.  Furthermore, by \cite[Lemma 8.2]{Fried1} if $X$ is a K\"ahler, Calabi-Yau threefold in the sense of Definition~\ref{defn: CY} then the fibers of smooth $\mu: \mathcal{X}\rightarrow \Delta$ are again Calabi-Yau.  In particular, assuming~\eqref{eq: Friedman} holds, there is a holomorphic family
\[
\mu: \mathcal{X} \rightarrow \Delta:=  \{ t\in \mathbb{C} : |t|<1\}
\]
such that $\mu^{-1}(t)= X_t$ is smooth for $t\ne 0$ and $\mu^{-1}(0)= \underline{X}$.  By a result of Kas-Schlessinger \cite{KSchles} the family $\mathcal{X}_{t}$ is locally biholomorphic to the model smoothing $\mathcal{V}$ near each ordinary double point.  We make the following definition.

\begin{defn}
Let $X$ be a smooth, compact, complex $3$-fold.  A conifold transition of $X$, denoted $X\rightarrow \underline{X} \leadsto X_{t}$ consists of a holomorphic map $\pi:X \rightarrow \underline{X}$ and family $\mu:\mathcal{X} \rightarrow \Delta$ with $\mu^{-1}(0)= \underline{X}$ such that
\begin{enumerate}
\item  $\pi: X \rightarrow \underline{X}$ contracts a collection of disjoint $(-1,-1)$ curves $C_1,\ldots, C_k$ to isolated, ordinary double point singularities $p_1, \ldots, p_k \in \underline{X}$, and $\pi$ is a biholomorphism  $X\setminus \cup_i C_i \rightarrow \underline{X} \setminus \{p_1,\ldots, p_k\}$.
\item $\mu:\mathcal{X} \rightarrow \Delta$ is a holomorphic smoothing of $\underline{X} = \mu^{-1}(0)$ and $X_t= \mu^{-1}(t)$.
\end{enumerate}
\end{defn}

Said informally, a conifold transition consists of contracting a collection of disjoint $(-1,-1)$ curves followed by smoothing the resulting double point singularities. At the level of topology, removing a neighborhood of the singular points of $\underline{X}$ leaves a boundary $S^2 \times S^3$, and either side of the transition corresponds topologically to gluing $S^2 \times B^4$ or $B^3 \times S^3$ along this boundary (see e.g. \cite{Rossi} for details). Thus conifold transitions allow travel between Calabi-Yau threefolds of different topology by degenerating 2-cycles and introducing 3-cycles. By the above discussion, Theorem~\ref{thm: Friedman} gives necessary and sufficient conditions for the existence of conifold transitions. 

\subsection{Metric geometry of the conifold}
Let us turn now to the discussion of some aspects of the metric geometry of conifold transitions.  Recall that the conifold~\eqref{eq: sec2Conifold} can be viewed as the cone over $\mathbb{P}^1\times\mathbb{P}^1$ in the negative line bundle $p_1^*\mathcal{O}_{\mathbb{P}^1}(-1)\otimes p_2^*\mathcal{O}_{\mathbb{P}^1}(-1)$.  Since $\mathbb{P}^1\times\mathbb{P}^1$ is K\"ahler-Einstein with positive Ricci curvature, it is well-known that the conifold admits a conical Calabi-Yau metric.  Explicitly, let 
\[
h_{FS} = \sum_{i=1}^4 |X_i|^2
\]
denote the Fubini-Study metric on $\mathcal{O}_{\mathbb{P}^3}(-1)$.  By direct computation we have
\be\label{eq: pullbackKE}
\iota^*h_{FS} = h_{KE}
\ee
where $h_{KE}=p_1^*h_{FS}\otimes p_2^*h_{FS}$ and $p_i^*h_{FS}$ is the pull-back of the Fubini-Study metric on $\mathbb{P}^1$ for $i=1,2$.  Define a function
\[
r: p_1^*\mathcal{O}_{\mathbb{P}^1}(-1)\otimes p_2^*\mathcal{O}_{\mathbb{P}^1}(-1) \rightarrow \mathbb{R}_{\geq 0}
\]
in the following way.  If $x \in \mathbb{P}^1\times\mathbb{P}^1$, and $\sigma$ is a local section of\\ ${p_1^*\mathcal{O}_{\mathbb{P}^1}(-1)\otimes p_2^*\mathcal{O}_{\mathbb{P}^1}(-1)}$, then define
\[
r(x,\sigma(x))^2 = \left( |\sigma|^2_{h_{KE}}\right)^{2/3}.
\]
 Clearly $r^{-1}(0)$ is precisely the zero section of $p_1^*\mathcal{O}_{\mathbb{P}^1}(-1)\otimes p_2^*\mathcal{O}_{\mathbb{P}^1}(-1)$, and hence $r$ defines a function on the conifold \eqref{eq: sec2Conifold}.  From the observation~\eqref{eq: pullbackKE}, we can write this function in terms of the coordinates on $\mathbb{C}^4$ as
\[
r^2= \|z\|^{4/3}, \qquad \|z\|^2:=\sum_{i=1}^{4}|z_i|^{2}.
\]
 An easy calculation shows that 
\[
\omega_{co,0} := \frac{3}{2}\ddb r^2
\]
defines a conical Calabi-Yau metric $g_{co,0}$ on $V_0$; note that the factor of $\frac{3}{2}$ is a harmless scaling, but we have included it to be consistent with \cite{FLY}.  More precisely, the Ricci-flat K\"ahler metric $g_{co,0}$ is a cone metric over the link $L:=\{r=1\} \subset V_0$ and can be written as
\[
g_{co,0}= \frac{3}{2}\left(dr^2 + r^2 g_{L}\right)
\]
where $g_{L}$ is (the pullback of) a Sasaki-Einstein metric on $L:=\{r=1\} \subset V_0$.  The cone $(V_0, g_{co,0})$ has a natural rescaling action generated by the vector field $r\frac{\del}{\del r}$.  The vector field
\[
\xi = J(r\frac{\del}{\del r})
\]
is tangent the level sets or $r$ and defines the Reeb vector field of the Sasaki structure on the link.  We will use frequently the holomorphic vector field
\[
\xi_{\mathbb{C}} := r\frac{\del}{\del r}-\sqrt{-1}\xi.
\]
Explicitly, the vector field $\xi_{\mathbb{C}}$ generates the $\mathbb{C}^*$ action on $V_0$ given by
\[
\lambda\cdot(z_1,z_2,z_3,z_4) = (\lambda^{3/2}z_1, \lambda^{3/2}z_2,\lambda^{3/2}z_3,\lambda^{3/2}z_4),
\]
 and one can easily check that the cone metric $g_{co,0}$ is homogeneous of degree $2$ under this action.  In particular, we have
 \begin{lem}
 The conical Calabi-Yau metric $g_{co,0}$ on the conifold has the following property:  for every $k \in \mathbb{Z}_{\geq 0}$ there is a constant $C_k$ so that
 \[
 |\nabla^k {\rm Rm}|_{g_{co,0}} \leq C_{k}r^{-2-k}.
\]
 \end{lem}

 \subsection{Metric geometry of the local smoothings}
 
 Candelas-de la Ossa \cite{CO} and independently Stenzel \cite{Stenz} constructed Calabi-Yau metrics on the smoothings of the conifold~\eqref{eq: modelSmoothing} using ODE techniques.  These metrics will play an important role for us.  
 
 \begin{prop}[Candelas-de la Ossa, \cite{CO}]
 Consider the smoothing of the conifold given by~\eqref{eq: modelSmoothing}, and set $V_t=\{z_1^2+z_2^2+z_3^2+z_4^2=t\} \subset \mathbb{C}^4$.  Let $\|z\|^2 =\sum_{i=1}^4 |z_i|^2$, and, for each $t \in \overline{\Delta}$, set
 \be\label{eq: COpotential}
 f_{t}(s) = 2^{-\frac{1}{3}}|t|^{2/3}\int_{0}^{\cosh^{-1}(\frac{s}{|t|})} (\sinh(2\tau)-2\tau)^{\frac{1}{3}} d\tau 
\ee
Then 
\[
\omega_{co,t} := \ddb f_t(\|z\|^2)
\]
 defines a smooth Calabi-Yau metric $g_{co,t}$ on $V_{t}$ whose tangent cone at infinity is the conifold $(V_0, g_{co,0})$.
 \end{prop}

 We note that $\| z \|^2 \geq t$ on $V_t$, and the set $\{ \| z \|^2 = t \}$, which is topologically $S^3$, will sometimes be called the vanishing cycle.
 \smallskip
\par Let $\mathcal{V}$ denote the model smoothing~\eqref{eq: modelSmoothing} and let $\mu: \mathcal{V}\rightarrow \mathbb{C}$ be the projection to the $t$ coordinate. There is a natural $\mathbb{C}^*$ action on the family $\mathcal{V}$, given by
 \be\label{eq: scalingMapDef}
S_{\lambda}(z_1,z_2,z_3,z_4, t) := (\lambda^{3/2}z_1, \lambda^{3/2}z_2, \lambda^{3/2}z_3, \lambda^{3/2}z_4, \lambda^3 t).
 \ee
 so that $S_{\lambda}:V_{t} \rightarrow V_{\lambda^3t}$. Under this $\mathbb{C}^*$ action we have
 \[
 S_\lambda^*f_{\lambda^3t_0}(\|z\|^2) = |\lambda|^2f_{t_0}(\|z\|^2)
 \]
 and so, in particular, we have
 \be\label{eq: gcotPullBack}
g_{co,t} = |t|^{2/3} (S_{t^{-1/3}})^*g_{co,1}.
 \ee
Note that, strictly speaking, we should fix a branch of $\log$ in the above expression, but since the CO metric is manifestly $S^1$ invariant such a distinction is irrelevant. It follows that the CO metrics $g_{co,t}$ are generated by the $\mathbb{C}^*$ action on $\mathcal{V}$, up to rescaling.  In particular, this shows (cf. \cite[Lemma 5.1]{FLY})
 
 \begin{lem}
For each $k \in \mathbb{Z}_{\geq 0}$, and $A>0$ there is a constant $C_{k,A}>0$, independent of $t$, so that the Calabi-Yau metrics $g_{co,t}$ satisfy
 \[
\sup_{\|z\|^2\leq A|t|} |\nabla^k {\rm Rm}(g_{co,t})|_{g_{co,t}} \leq C_{k,A}|t|^{-\frac{2}{3}(2+k)}.
 \]
 \end{lem}

 It will be important for us to understand the rate at which the CO
 metric converges to its tangent cone at infinity.   Consider the map
\be\label{eq: mapPhit}
\begin{aligned}
\Phi_{t}(z): V_{0}\setminus\left\{\|z\|^2 \leq \frac{|t|}{2}\right\}& \longrightarrow \, V_t \setminus \{\|z\|^2 = |t|\},\\
z&\longmapsto z + \frac{t\bar{z}}{2\|z\|^2}.
 \end{aligned}
\ee
Tracing the definitions one can check that $\| z \|^2 \leq \| \Phi_t(z) \|^2 \leq 2 \| z \|^2$ and
\be\label{eq:PhitRescale}
\Phi_{t} = S_{t^{1/3}}\circ \Phi_{1} \circ S_{t^{-1/3}}.
\ee
We have

\begin{lem}[Conlon-Hein \cite{CH}, Proposition 5.9]\label{lem: ConHeinDecay}
Under the identification $\Phi_{1}$ we have that, for all $k \in \mathbb{Z}_{\geq 0}$, there is a constant $C_k$ such that
\[
|\nabla^{k}_{g_{co,0}}( \Phi_{1}^*g_{co,1} - g_{co,0})|_{g_{co,0}} \leq C_k r^{-3-k}.
\]
\end{lem}

Combining Lemma~\ref{lem: ConHeinDecay} with~\eqref{eq: gcotPullBack} and ~\eqref{eq:PhitRescale} we obtain estimates for the decay rate of $g_{co,t}$ towards $g_{co,0}$.
 \begin{cor}\label{cor: decayGcotPull}
For all $k\in \mathbb{Z}_{\geq 0}$ there is a uniform constant $C_k$, independent of $t$, so that
\[
|\nabla^{k}_{g_{co,0}} \Phi_{t}^*g_{co,t} - g_{co,0}|_{g_{co,0}} \leq C_k|t| r^{-3-k}.
\]
\end{cor}
\begin{proof}
From~\eqref{eq: gcotPullBack} and ~\eqref{eq:PhitRescale} we have
\[
\Phi_{t}^*g_{co,t} =|t|^{2/3} (S_{t^{-1/3}})^*\Phi_{1}^*g_{co,1}.
\]
It follows that, if $x\in V_0$, then, from Lemma~\ref{lem: ConHeinDecay} we get
\[
\begin{aligned}
|\nabla^{k}_{g_{co,0}}( \Phi_{t}^*g_{co,t} - g_{co,0})|_{g_{co,0}}(x) &=|t|^{2/3}|\nabla^{k}_{g_{co,0}} \Phi_{t}^*g_{co,t} - g_{co,0}|_{|t|^{2/3}g_{co,0}} (S_{t^{-1/3}}(x))\\
&\leq C_k|t|^{2/3}|t|^{-\frac{2}{3}\cdot \frac{2+k}{2}}r(S_{t}^{-1/3}(x))^{-3-k}\\
&=C_k|t|r(x)^{-3-k}
\end{aligned}
\]
\end{proof}

One application of this result will be to transplanting estimates for tensors on $V_0$ to estimates on the smooth varieties $V_t$.  The following lemma follows from Corollary \ref{cor: decayGcotPull} and $\| z \|^2 \leq \| \Phi(z)\|^2 \leq 2 \| z\|^2$.

\begin{lem} \label{lem:transplant}
There is a constant $R>0$ depending only on the constant $C_0$ in
Corollary~\ref{cor: decayGcotPull}, such that, if $T$ is a tensor on
some subset of $V_{0} \cap \{ r^3 >R\frac{|t|}{2}\}$ satisfying the estimate
\[
|\nabla^k_{g_{co,0}} T|_{g_{co,0} } \leq M_k r^{\lambda-k}
\]
for some $k \in \mathbb{Z}_{\geq 0}$, constant $M_k>0$ and some $\lambda$,
then $(\Phi_{t}^{-1})^*T$ defines a tensor on $V_t \cap\{ r^3 \geq \frac{R^2+1}{2R}|t|\}$ satisfying the estimate
\[
|\nabla^k_{g_{co,t}} (\Phi_{t}^{-1})^*T|_{g_{co,t}}\leq M_k' r^{\lambda-k}
\]
for constants $M_k' >0$ depending only on $M_k$ and the the constants $C_k$ appearing in Corollary~\ref{cor: decayGcotPull}.
\end{lem}
We also note that the estimate from Corollary \ref{cor: decayGcotPull} implies the
following estimate on
$V_t$
\be \label{eq:decayGco0Pull}
 \big| \nabla^k_{g_{co,t}} [(\Phi^{-1}_t)^*g_{co,0} - g_{co,t}]\big|_{g_{co,t}} \leq C|t| r^{-3-k}.
\ee

It will be useful later in the paper to have a well adapted system of coordinates in which to carry out our analysis.  The following, which we refer to as ``holomorphic cylindrical coordinates" were used in \cite{Chuan}.  For completeness, we recall these coordinates and prove their existence in the following
\begin{lem}\label{lem: cylCoords}
There are uniform constants $\rho>0$ and $C_{k}>0, k\in \mathbb{Z}_{\geq 0}$ with the following effect: If $\hat{z}:= (\hat{z}_1,\ldots, \hat{z}_4) \in V_{t}$ and $\hat{z}\ne 0$, then there is an open neighborhood $U_{\hat{z}} \ni \hat{z}$ and a holomorphic embedding $\psi:B_{\rho}(0)\rightarrow U_{\hat{z}}$ such that, if $(w_1,\ldots, w_{3})$ denote the local holomorphic coordinates on $B_{\rho}(0)$, then, setting $\hat{r}=r(\hat{z})$ we have
\begin{itemize}
\item[$(i)$] We have $\frac{1}{4} \hat{r} \leq r(w) \leq 4\hat{r}$ on $B_{\rho}$, and
\[
|\del^kr|_{g_{euc}}(w) \leq C_k \hat{r}
\]
where $g_{euc}$ denotes the Euclidean metric on $B_{\rho}$.
\item[$(ii)$] In these coordinates we have
\[
 C_0^{-1} g_{euc} \leq \hat{r}^{-2}g_{co,t} \leq C_0 g_{euc}
\]
and
\[
|\del^k (\hat{r}^{-2}g_{co,t})|_{g_{euc}} \leq C_k
\]
\end{itemize}
\end{lem}
\begin{proof}
We begin by constructing some candidate coordinates.  Fix $\hat{z}:= (\hat{z}_1,\ldots, \hat{z}_4) \in V_{t}$.  Clearly $|\hat{z}_i| > \frac{1}{100}\|\hat{z}\|$ for some $1\leq i \leq 4$, and hence, without loss of generality we may assume $i=4$.   We claim that $w_i=z_i-\hat{z}_i$, for $1 \leq i \leq 3$ form a coordinate system near $\hat{z}$.  Indeed, consider the function
\[
F_t(z_1,z_2,z_3,z_4)= \sum_{i=1}^4 z_i^2-t.
\]
By the implicit function theorem, the coordinates $(z_1,z_2,z_3)$ form local coordinates on $\{F_t=0\}$ whenever $\frac{\del F}{\del z_4} = 2z_4 \ne 0$.  Let us examine these coordinates.  Since $r^2 = \|z\|^{4/3}$ one can easily show that $\frac{3}{2}\ddb r^2$ is a globally defined, smooth metric on $\mathbb{C}^4\setminus\{0\}$, which we still denote by $g_{co,0}$.  We claim that, up to scaling and translating, the estimates in $(i)$ always hold in these coordinates.  Indeed, from the scaling relation $r(\lambda \cdot z) = |\lambda| r(z)$ we may as well assume that $\|\hat{z}\|=1$.  It is easy to see that, for any multi-index $\alpha=(k_1,k_2,k_3,k_4) \in \mathbb{Z}_{\geq 0}^4$ we have
\[
\sup_{\{\frac{1}{8}<\|z\|<8 \}}\bigg| \frac{\del^{|\alpha|}}{\del z^\alpha} r \bigg|_{g_{euc},z}\leq C(\alpha).
\]
Since $10^{-3}<|z_4|<10$ on this region, the estimates in $(i)$ will follow from comparing $g_{euc,z}= \sum_{i=1}^4 |dz_i|^2$ to the Euclidean metric in the coordinates 
\[
(w_1,w_2,w_3)=(z_1-\hat{z}_1,z_2-\hat{z}_2,z_3-\hat{z}_3).
\]
Using the definition of the coordinates we have
\[
g_{euc,z} = \sum_{i=1}^{3}|dw_i|^2 + \frac{|\sum_{i=1}^3 z_i(w)dw_i|^2}{|z_4(w)|^2}.
\]
For clarity, let us denote by $g_{euc,w} = \sum_{i=1}^{3}|dw_i|^2 $ the Euclidean metric in the $w$ coordinates.  Then, from the Cauchy-Schwarz inequality, together with $|z_4| \geq \frac{1}{100} \|\hat{z}\| =\frac{1}{100}$ we have
\[
g_{euc,w} \leq g_{euc,z} \leq 10^3g_{euc,w}.
\]
Noting that $\{\frac{1}{8} <\|z\| < 8\}$ implies that $\{\frac{\hat{r}}{4} < r(z) < 4\hat{r}\}$ we see that $(i)$ holds in these coordinates.

Next, we address $(ii)$ in the special case $t=0$. In this case the desired estimates follow immediately from a scaling argument and the above estimates for $r$, since rescaling preserves $V_0$. 

It only remains to determine a bound for $\rho$ such that $\{\sum_{i=1}^{3}|w|<\rho\}$ implies $\{\frac{1}{8} < \|z(w)\|<8\}$. From the definition of the coordinates we have
\[
\|z(w)\|^2 = \sum_{i=1}^3 |w_i+\hat{z}_i|^2 + \big|\sum_{i=1}^3 (w_i+\hat{z}_i)^2\big|.
\]
On the one hand, we have
\[
 \sum_{i=1}^3 |w_i+\hat{z}_i|^2 \geq \frac{1}{2}\sum_{i=1}^{3}|\hat{z}_i|^2 - \sum_{i=1}^3 |w_i|^2, 
 \]
 while on the other hand
 \[
 |\sum_{i=1}^3 (w_i+\hat{z}_i)^2| \geq \big|\sum_{i=1}^3 \hat{z}_i^2\big| -5\sum_{i=1}^3 |w_i|^2 - \frac{1}{4}\sum_{i=1}^4 |\hat{z}_i|^2.
 \]
 Thanks to the fact that $1=\sum_{i=1}^3 |\hat{z}_i|^2 +\big|\sum_{i=1}^3\hat{z}_i^2\big|$ we get
 \[
 \|z(w)\|^2 \geq \frac{1}{4}- 6\sum_{i=1}^3 |w_i|^2 >\frac{1}{64}
 \]
 provided $\sum_{i=1}^3 |w_i|^2 <\rho \leq \frac{1}{8}$.  The upper bound is similar.

Next we consider the case when $\hat{z} \in V_{1}$ and $\{\|\hat{z}\|^2 > R\}$ for some large constant $R\geq 2$ to be determined.  In this region we will also use coordinates constructed from $(z_1,z_2, z_3)$, assuming as before  that $|\hat{z}_4| > \frac{1}{100}\|\hat{z}\|$.  Since we have already established the estimates $(i)$ in these coordinates, it suffices to prove $(ii)$.  Let $\hat{y} \in V_0$ with $\|\hat{y} \|^2 > \frac{1}{2}$ be defined by
\[
\Phi_1(\hat{y}) = \hat{z} \in V_1.
\]
Note that from the definition (\ref{eq: mapPhit}) of $\Phi_1$ we have $\|\hat{y}\|^2 > R/2$.  Define coordinates $(x_1,x_2,x_3)$ on $V_1$, near $\hat{z}$ by
\[
z_i = |\hat{z}|x_i+\hat{z}_i
\]
and let $(w_1,w_2,w_3)$ be the coordinates on $V_0$, centered at $\hat{w}$ constructed above.  Explicitly,
\[
z_i = |\hat{y}|w_i + \hat{y}_i, \quad 1 \leq i \leq 3
\]
where $\sum_{i=1}^3|w_i|^2 <\frac{1}{8}$.
The map $\Phi_1$ is given in these coordinates by
\be\label{eq: Phi1Coords}
w_i \longmapsto  \frac{\|\hat{y}\|}{\|\hat{z}\|} \left( w_i + \frac{\overline{w_i\|\hat{y}\| + \hat{y}_i}}{2\|\hat{y}\|\cdot\|z(w)\|^2} - \frac{\overline{\hat{y}_i}}{2\|\hat{y}\|\cdot\|\hat{z}\|^2}\right)
\ee
where
\[
\|z(w)\|^2= \sum_{i=1}^{3} \big|\|\hat{y}\|\cdot w_i+\hat{y}_i\big|^2 + \big|\sum_{i=1}^3(\|\hat{y}\|\cdot w_i+\hat{y}_i|)^2\big|.
\]
 When $\|y\|^2>\frac{1}{2}$ we have $1\leq \frac{\|\Phi_1(y)\|^2}{\|y\|^2} \leq 2$ and so 
 \be\label{eq: compareRadPhi1}
 \frac{1}{2}r(\Phi_1(y)) < r(y) < r(\Phi_1(y)).
 \ee
 From~\eqref{eq: Phi1Coords} it follows that there is a constant $c\in[\frac{1}{2}, 1]$ such that
 \[
\frac{\del}{\del w_j} x_i\circ \Phi_1= c \delta^{i}_{j} + O(R^{-1}), \qquad \frac{\del}{\del \bar{w}_j} x_i\circ \Phi_1=  O(R^{-1}),
 \]
 and, for all multi-indices $\alpha$ with $|\alpha|\geq 2$,
 \be\label{eq: Phi1smallDer}
\frac{\del^{|\alpha|}}{\del w^{\alpha}} x_i\circ \Phi_1= O(R^{-1}), \qquad \frac{\del^{|\alpha|}}{\del \bar{w}^{\alpha}} x_i\circ \Phi_1=  O(R^{-1}).
 \ee
Thus, by choosing $R$ sufficiently large we can ensure that
 \be\label{eq: compPhi1Euc}
 \frac{1}{2}g_w< \Phi_1^*g_x < 2g_w.
 \ee
 where $g_{x}$ denotes the Euclidean metric in $x$ coordinates, and $g_{w}$ denotes the Euclidean metric in $w$ coordinates. Furthermore,~\eqref{eq: Phi1smallDer} implies that, for all multi-indices $\alpha \ne 0$ we have
 \[
 \frac{\del^{|\alpha|}}{\del w^{\alpha}} \Phi_1^*g_{x}= O(R^{-1}).
 \]
By Lemma~\ref{lem: ConHeinDecay} and ~\eqref{eq: compareRadPhi1} we can choose $R$ sufficiently large so that
 \be\label{eq: 0-1uniformequiv}
 \frac{1}{2} g_{co,0} \leq \Phi_1^*g_{co,1} \leq 2 g_{co,0}.
 \ee
Combining~\eqref{eq: compPhi1Euc},~\eqref{eq: 0-1uniformequiv} and~\eqref{eq: compareRadPhi1} we deduce the estimates in $(ii)$ from Lemma~\ref{lem: ConHeinDecay} and the estimates for $g_{co,0}$ in the $w$ coordinates.  For example, we have
\[
\begin{aligned}
|\nabla_{g_{x}}r(\hat{z})^{-2} g_{co,1}|_{g_{x}} &= |\nabla_{\Phi_1^{*}g_{x}}r(\hat{z})^{-2} \Phi_1^*g_{co,1}|_{\Phi_1^*g_{x}} \\
& \leq C\left( |\nabla_{g_w}r(\hat{z})^{-2} \Phi_1^*g_{co,1}|_{g_{w}} + |O(R^{-1}) r(\hat{y})^{-2} g_{co,0}|_{g_{w}}\right)\\
&\leq C\left(|\nabla_{g_{co,0}}( \Phi_1^*g_{co,1}- g_{co,0})|_{g_{co,0}} + | r(\hat{y})^{-2}\del g_{co,0}|_{g_{w}} +1\right)\\
&\leq C_1.
\end{aligned}
\]
Higher order derivatives follow similarly by induction. Note that if $R$ is sufficiently large then $\Phi_1(\{ w : \sum_{i=1}^3|w_i|^2 <\frac{1}{8}\}) \supset \{x : \sum_{i=1}^3|x_i|^2 <\frac{1}{10}\}$ and so again, $\rho$ can be chose uniformly.

The portion of $V_1$ given by $\{z\in V_1 : \|z\|^2< 2R\}$ is compact and hence the desired coordinates can be constructed by a covering argument. It only remains to construct the coordinates on $V_t$ for $0<|t|<1$.  But for $t\ne0$ we can use the holomorphic rescaling map~\eqref{eq: scalingMapDef} to induce holomorphic coordinates on $V_t$ from those on $V_1$.  By ~\eqref{eq: gcotPullBack} the metrics $g_{co,t}$ are generated, up to a scaling parameter, by the holomorphic rescaling map.  On the other hand, the estimates in $(i),(ii)$ are invariant under this rescaling, and hence the lemma follows.
\end{proof}

\begin{rk}
Note that the construction in Lemma~\ref{lem: cylCoords} shows that there is a constant $R>0$ such that, $\hat{z} \in \{ \|z\|^2 \geq R|t|\}$ and $|\hat{z}_4| > \frac{1}{100} \|\hat{z}\|$, then the holomorphic cylindrical coordinates can be taken to be
\[
(w_1,w_2,w_3) = \frac{1}{\|\hat{z}\|}(z_1-\hat{z}_1, z_2-\hat{z}_2, z_{3}-\hat{z_3})
\]
\end{rk}

Before moving on from this local discussion we state a lemma regarding extending some of the local objects introduced above.

\begin{lem}\label{lem: trivGlob}
 If $\mu:\mathcal{X}\rightarrow \Delta$ is a global smoothing of a Calabi-Yau variety $X_0= \mu^{-1}(0)$ with ordinary double point singularities at points $\{p_1, \ldots, p_k\}$, then there are disjoint open sets $U_i \subset \mathcal{X}$, with $p_i \in U_i$ such that $U_{i}$ is biholomorphic to a neighborhood of $0$ in the model smoothing ~\eqref{eq: modelSmoothing} and,
 \begin{itemize}
 \item[$(i)$] there is a globally defined function $r: \mathcal{X} \rightarrow \mathbb{R}_{\geq0}$ such that, after identifying $U_i$ with the model smoothing, $r^2|_{U_{i}} =\|z\|^{4/3}$, and $r^{-1}(0) = \{p_1,\ldots, p_k\}$.
 \item[$(ii)$] There is a collection of closed sets $U_i' \subset U_i$, and closed sets $C^t_i\subset X_t$ smooth map $\Phi_t: X_0\setminus \cup_i U_i' \rightarrow \mu^{-1}(t)\setminus C^t_i$ such that, after identifying $U_i$ with the model smoothing we have
 \[
 U_{i}'= \left\{\|z\|^2 \leq \frac{|t|}{2}\right\}, \quad C_i^t = \{\|z\|^2=|t|\}
 \]
 and $\Phi_t\big|_{U_{i}}$ is the map $z \mapsto z+ \frac{t\bar{z}}{2\|z\|^2}$.
 \end{itemize}
 \end{lem}
 \begin{proof}
 We only sketch the proof, since it is straightforward.  Given a choice of the open sets $U_i$ we extend the locally defined functions $\|z\|^{4/3}$ on $\mathcal{X}$ to globally defined positive functions which only vanish that the singular points $p_i$.  This establishes $(i)$.  To prove $(ii)$ we just observe that the locally defined maps are given by the flow of a vector field which lifts $\frac{\del}{\del t}$.  We can extend this map globally by using a partition function to glue with any lift of $\frac{\del}{\del t}$ to $\mathcal{X}\setminus U_i$ (eg. by choosing a Riemannian metric on $\mathcal{X}$).
 \end{proof}
 
 \subsection{Balanced metrics on conifold transitions}
 
 We now review the work of Fu-Li-Yau \cite{FLY} who constructed balanced metrics on non-K\"ahler Calabi-Yau threefolds using conifold transitions and gluing.
 
 \begin{defn}
 Let $(X,g)$ be a complex manifold of complex dimension $n$ with a hermitian metric.  The metric $g$ is said to be balanced if the associated $(1,1)$ form $\omega$ satisfies
 \[
 d\omega^{n-1}=0.
 \]  
 \end{defn}
 
 We have the following theorem
 
 \begin{thm}[Fu-Li-Yau \cite{FLY}, Theorem 1.2]\label{thm: FLY}
 Let $X$ be a smooth, K\"ahler, Calabi-Yau threefold, and suppose that
 $X\rightarrow \underline{X} \leadsto X_t$ is a conifold transition.
 Then, for $|t|$ sufficiently small $X_t$ admits a balanced metric
 $\omega_{{\rm FLY},t}$.
 \end{thm}
 
 We will need to recall some aspects of the proof of Theorem~\ref{thm: FLY} as they will play an important role in subsequent sections.  The first step \cite{FLY} is to construct a balanced metric on $\underline{X} = X_0$ by appropriately gluing a Calabi-Yau metric on $X$ with the conical Calabi-Yau metric on the conifold.  To fix notation, for each ordinary double point $p_i \in X$ we fix an identification of a neighborhood of $p_i$ with a neighborhood of the singular point in the conifold.  Define
 \[
 U_i(\epsilon)= \left\{ (z_1,z_2,z_3,z_4)\in \mathbb{C}^4 : \sum_{i=1}^4 z_i^2=0, \text{ and } \|z\|^2 \leq \epsilon\right\}
 \]
and let $U(\epsilon) = \bigcup_{i}U_{i}(\epsilon)$.  We state the following result of Fu-Li-Yau \cite{FLY};  for the reader's convenience we have given a self-contained proof in Appendix~\ref{sec: FLYapp}.
 
 \begin{prop}[Fu-Li-Yau \cite{FLY}, Proposition 2.6]\label{prop: FLY2.6}
 With the above notation, for every $0<\epsilon \ll 1$ sufficiently
 small there is a constant $M_0>0$ such that there exists a hermitian metric $g_0$ on $\underline{X}_{reg}$ whose associated $(1,1)$ form $\omega_0$ has the following properties
 \begin{itemize}
 \item[$(i)$] On $\underline{X}\setminus U(1)$ we have $\omega_0 = \omega_{CY}$ where $\omega_{CY}$ is a Calabi-Yau metric on $X$.
 \item[$(ii)$] On $U(\epsilon)\cap \underline{X}_{reg}$ we have $\omega_0 = M_0^{1/2}\epsilon^{-1/3}\omega_{co,0}$.
 \item[$(iii)$] On $U(1)\setminus U(\epsilon)$, $\omega_0^2$ is $\ddb$-exact.
 \end{itemize}
 In particular, $g_0$ is balanced on $\underline{X}_{reg}$.
 \end{prop}

For the remainder of the paper, $\epsilon>0$ is fixed and we will use
the corresponding metric $g_0$. It will be useful to use to have a comparison of $\omega_{CY}$ with $\omega_0$ on
$X$, away from the contracted rational curves. In a neighborhood of
the curves, the metric $\omega_{CY}$ is uniformly equivalent to the
smooth reference metric
$\omega_{sm} = \ddb r^3 + \pi^*_{\mathbb{P}^1} \omega_{FS}$ (where we recall that $r^3 = \| z
\|^2$). For $0<\lambda<1$, we can rescale by $S_\lambda: \{ 1 \leq r \leq 2 \} \rightarrow \{ \lambda \leq r \leq 2 \lambda \}$ with $S_\lambda(z) = \lambda^{3/2}z$. This gives
\[
S_\lambda^* \omega_{sm} = \lambda^3 \ddb r^3 + \pi^*_{\mathbb{P}^1} \omega_{FS},
\quad S_\lambda^* \omega_{co,0} = \lambda^2 \omega_{co,0},
\]
and by the uniform equivalence of metrics $C^{-1} \omega_{sm} \leq \omega_{co,0} \leq C \omega_{sm}$ on $\{ 1 \leq r \leq 2 \}$, we obtain $C^{-1} \lambda^2 S_\lambda^* \omega_{sm} \leq  S_\lambda^* \omega_{co,0} \leq C \lambda^{-1} S_\lambda^* \omega_{sm}$. Thus
\be \label{eq: compare-cy-co}
C^{-1} r^2 \omega_{CY} \leq \omega_0 \leq C r^{-1} \omega_{CY}
\ee
on $\{ 0 < r < 1 \}$. 
 \smallskip
\par The next step in the proof of Theorem~\ref{thm: FLY} is to construct
 approximately balanced metrics on the smooth fibers $X_{t}$.  Let
 $U(\frac{1}{2})\subset \mathcal{X}$ be a small open set containing
 the singular points $p_i$. Note that $\underline{X}\setminus
 U(\frac{1}{2})$ is diffeomorphic to $X_{t}\setminus U(\frac{1}{\sqrt{2}})$
 by the map $\Phi_t$ constructed in Lemma~\ref{lem: trivGlob}. If $\omega_0^2$ is the metric from Proposition~\ref{prop: FLY2.6}, then $(\Phi_{t}^*\omega_{0}^2)^{(2,2)}$ is positive definite for $|t|$ sufficiently small and can be glued to the Candelas-de la Ossa metric $\omega_{co,t}^2$ to obtain a positive $(2,2)$ form.  The following result can be extracted from Fu-Li-Yau \cite[Section 3]{FLY}, see for example \cite[equation (3.4)]{FLY}.
 
 \begin{prop}[Fu-Li-Yau \cite{FLY}]\label{prop: FLYalmost}
 With notation as above, for $\epsilon, |t|$ sufficiently small and $M_0$ sufficiently large there is a hermitian metric $g_t$ on $X_t$ such that the associated $(1,1)$ form $\omega_{t}$ has
 \begin{itemize}
 \item[$(i)$] $\omega_{t} = M_0^{1/2}\epsilon^{-1/3}\omega_{co,t}$ is K\"ahler Ricci-flat in $U(\epsilon) \cap X_{t}$.
 \item[$(ii)$]There is a constant $C_k$, independent of $|t|$ so that  $|d\omega_{t}^2|_{C^k(X_t,\omega_{t})} \leq C_k|t|$.
 \item[$(iii)$] As $|t|\rightarrow 0$, $\Phi_{t}^*\omega_{t}$ converges smoothly, in compact subsets of $\underline{X}\setminus\{p_1,\ldots,p_k\}$  to the balanced metric $\omega_{0}$ of Proposition~\ref{prop: FLY2.6}.
 \end{itemize}
 \end{prop}

 With this result, Theorem~\ref{thm: FLY} is obtained by solving a
 $4$-th order linear equation with estimates in order to perturb the
 approximately balanced metric $\omega_t$ of Proposition~\ref{prop:
   FLYalmost} to a genuine balanced metric $\omega_{{\rm FLY},t}$ for $|t|$ sufficiently small.
 \smallskip
 
 \subsection{Balanced metrics on the small resolution}
 
 The space $X$ can be viewed as a small resolution of the singular space $\underline{X}$, obtained by replacing the ordinary double points with $(-1,-1)$ rational curves.  It will be useful to us to have a sequence of degenerating balanced metrics on $X$.  
 
 Consider the total space $\pi_{\mathbb{P}^1}: \mathcal{O}_{\mathbb{P}^1}(-1)^{\oplus 2} \rightarrow \mathbb{P}^1$.  Let $h_{FS}$ denote the standard Fubini-Study metric on $\mathcal{O}_{\mathbb{P}^1}(-1)$.  If $z\in \mathbb{P}^1$, and $(u,v) \in \pi_{\mathbb{P}^1}^{-1}(z)$ denotes a point in the fiber over $z$, we define 
 \[
 r(u,v,z)^2 = \left(|u|^2_{h_{FS}} + |v|^2_{h_{FS}}\right)^{2/3}.
 \]
 It is easy to check that $r^2$ is the pull-back to $\mathcal{O}_{\mathbb{P}^1}(-1)^{\oplus 2}$  of the potential for the conical Calabi-Yau metric on the conifold.  Candelas-de la Ossa \cite{CO} constructed asymptotically conical Calabi-Yau metrics on the resolved conifold via the ansatz
 \[
 \omega_{co,a}:= \ddb f_{a}(r^3) + 4a^2\pi_{\mathbb{P}^1}^*\omega_{FS}
 \]
 where $\omega_{FS} = \ddb \log h_{FS}$ is the Fubini-Study metric on $\mathbb{P}^1$.  This metric is Calabi-Yau provided $f_{a}(x)$ satisfies the differential equation
 \[
 (xf_a'(x))^3+6a^2(xf_a'(x))^2= x^2, \qquad f_{a}'(x) \geq 0
 \]
 for $x \geq 0$.  From this expression it is straightforward to check that
 \[
 f_{a}(x) = a^2f_1\left(\frac{x}{a^3}\right).
 \]
In particular we see that
\be \label{eq: coa-scaling}
\omega_{co,a} = a^2S_{a^{-1}}^*\omega_{co, a=1}
\ee
where $S_{a}(u,v,z) = (a^{3/2}u, a^{3/2}v, z)$ is the scaling generated by the holomorphic Reeb vector field on the conifold.  The argument of Fu-Li-Yau carries over to give the following proposition; for the reader's convenience, we have provided the details in Appendix~\ref{sec: FLYapp}.

 \begin{prop}[Fu-Li-Yau \cite{FLY}, Proposition 2.6]\label{prop: FLY2.6-resolution}
For every $0<\epsilon \ll 1$ sufficiently small there is a constant
$M_0,M_1>0$ such that there exists a sequence of balanced hermitian metric
$\omega_a$ on $X$ with $a \rightarrow 0$
which has the following properties:
 \begin{itemize}
 \item[$(i)$] On $X \setminus U(1)$ we have $\omega_a = \omega_{CY}$ where $\omega_{CY}$ is a Calabi-Yau metric on $X$.
 \item[$(ii)$] On $U(\epsilon)\cap X$ we have $\omega_a =
   M_0^{1/2}\epsilon^{-1/3}\omega_{co,a}$.
 \item[$(iii)$] On $U(1)\setminus U(\epsilon)$, $\omega_a^2$ is
   $\ddb$-exact.
   \item[$(iv)$] As $a \rightarrow 0$, $\omega_a$ converges smoothly to $\omega_{0}$
     on compact subsets of $X \backslash \{ C_1, \dots, C_k \}$.
     \item[$(v)$] There is an estimate $M_1^{-1} \leq {\rm Vol}(X,g_a)
       \leq M_1$.
       \item[$(vi)$] We have $[\omega_{a}^2] = [\omega_{CY}^2] \in H^{4}(X,\mathbb{R})$. 
 \end{itemize}
 \end{prop}
 
 \begin{rk}
 Regarding notation, following the conventions of \cite{Chuan}, we will use $\omega_{co,a}$ to denote the Candelas-de la Ossa metric on the small resolution  of the conifold, while reserving $\omega_{co, t}$ for the Candelas-de la Ossa metric on the smoothing of the conifold.
 \end{rk}

 \subsection{Notation}
Before beginning the construction, we establish notation and conventions that will be
used throughout the paper.  Throughout the paper $T^{1,0}X$ will denote the holomorphic
tangent bundle to $X$, and  we will write the components of a hermitian
metric $H$ on $T^{1,0}X$ as $H=H_{\bar{k} j} \, dz^j
\otimes d \bar{z}^k$, and we will denote the inner product by
\[
\langle V,W \rangle_H = H_{\bar{k} j} V^j \overline{W^k}, \quad V=V^i
{\partial \over \partial z^i}, \ W= W^i {\partial \over \partial z^i}.
\]
The hermitian condition in this convention is $\overline{H_{\bar{k}
    j}} = H_{\bar{j} k}$. The inverse of $H$ will be denoted
$H^{p \bar{q}}$, so that matrix multiplication is $H^{j \bar{p}}
H_{\bar{p} k} = \delta^j{}_k$.  A hermitian metric $H$ induces metrics on
all the associated bundles in the usual fashion.

An endomorphism $A: T^{1,0}X\rightarrow T^{1,0}X$ has an adjoint $A^{\dagger}$ defined by
\[
\langle AV, W \rangle_{H} = \langle V, A^{\dagger}W \rangle_{H}.
\]
When the dependence on the metric is emphasized, we will write this as $A^{\dagger_{H}}$.  An endomorphism is $H$-self-adjoint when $A^{\dagger}=A$.  Note that, in this notation, the inner-product on endomorphisms is given by
\[
\langle A, B \rangle_{H}= {\rm Tr}(AB^{\dagger})
\]
The curvature $F_H$ of the Chern connection of the
metric $H$ will follow the convention
\[
(F_H)_{j \bar{k}}{}^p{}_q = - \partial_{\bar{k}} (H^{p
  \bar{r}} \partial_j H_{\bar{r} q}).
\]
From two hermitian metrics $\hat{H}$ and $H$ we can form the
relative endomorphism denoted $h=\hat{H}^{-1} H$, or in index notation denoted
\[
h^p{}_q = \hat{H}^{p \bar{r}} H_{\bar{r} q}.
\]
A formula which is the starting point for many computations is for the
difference of the curvature tensors of $\hat{H}$ and $H$.
\be \label{diff-chern-curv}
(F_H)_{j \bar{k}} - (F_{\hat{H}})_{j \bar{k}} = - \partial_{\bar{k}} (h^{-1}
\nabla^{\hat{H}}_j h).
\ee
Here the $p,q$ endomorphism indices are omitted for ease of notation,
and $\nabla^{\hat{H}}$ denotes the Chern connection of $\hat{H}$ acting on $h \in
\Gamma({\rm End} \, T^{1,0}X)$ by
\[
  \nabla_j^{\hat{H}} h = \partial_j h +
  [\hat{H}^{-1}\partial_j \hat{H}, h ].
\]
In this paper, $g$ will typically denote a balanced hermitian
metric with associated form $\omega=
\sqrt{-1} g_{\bar{k} j} \, dz^j \wedge d \bar{z}^k$. The contraction
operator $\Lambda_\omega$ acts on $F_H$ by
\[
(\sqrt{-1} \Lambda_\omega F_H)^p{}_q = g^{j \bar{k}} (F_H)_{j \bar{k}}{}^p{}_q.
\]
The curvature of $g$ will be denoted $(R_g)_{j \bar{k}}{}^p{}_q =
- \partial_{\bar{k}} (g^{p \bar{r}} \partial_j g_{\bar{r} q})$.  Since $g$ is not K\"ahler, it will have non-zero
torsion, which we denote by
\be \label{eq: torsion-def}
(T_g)^r{}_{ij} = (A_g)_i{}^r{}_j-(A_g)_j{}^r{}_i
\ee
where $(A_{g})_i{}^p{}_q = g^{p \bar{n}}\del_i g_{\bar{n} q}$ is the Chern connection of $g$. We will denote the $(n,n)$ form corresponding to $g$ by
\[
\dvol_g = {\omega^n \over n!}.
\]
We will denote the complex Laplacian acting on functions by $\Delta_g
f = g^{j \bar{k}} \partial_j \partial_{\bar{k}} f$. The balanced
condition $d \omega^{n-1}=0$ allows us to integrate by parts so that
\[
\int_X \psi \Delta_g \phi \, \dvol_g = \int_X  \phi \Delta_g \psi \, \dvol_g
\]
for any $\psi,\phi \in C^\infty(X)$. Lastly, we note that when obtaining estimates, we will use the convention where $C$ denotes a positive constant depending on known quantities which may vary line-by-line.

\section{Hermitian-Yang-Mills metrics on the central fiber} \label{sec: HYM-central}
In this section, we construct a Hermitian-Yang-Mills metric on the tangent bundle of the singular space $\underline{X}$ with respect to the Fu-Li-Yau balanced metric $\omega_0$. We will prove:

\begin{thm} \label{thm:H_0}
  There exists a hermitian metric $H_0$ on $T\underline{X}_{reg}$ satisfying
  \[
F_{H_0} \wedge \omega_0^2 =0
  \]
 where $\omega_0$ is the Fu-Li-Yau metric of Proposition \ref{prop:
  FLY2.6} and $F_{H_0}$ is the curvature of the Chern connection of $H_0$. For each $k \in \mathbb{Z}_{\geq 0}$, there is a constant $C_k>0$ such that the metric $H_0$ satisfies the estimates
  \be \label{eq:H_0-est}
|H_0|_{g_{0}} + |H_0^{-1}|_{g_{0}} \leq C_0, \quad |\nabla^k_{g_0} H_0|_{g_0} \leq C_k r^{-k}.
\ee
Here $r: \underline{X} \rightarrow \mathbb{R}_{\geq 0}$ is as in Lemma \ref{lem: trivGlob}.

\end{thm}

We will produce $H_0$ by extracting a limit from a sequence $\{(\omega_a,H_a) \}$ of Hermitian-Yang-Mills metrics with respect to the
degenerating sequence of background metrics $\{\omega_a\}$ from
Proposition \ref{prop: FLY2.6-resolution}. A similar approach is taken
in \cite{Chuan}.

\smallskip
\par Let $X$ be a simply-connected, compact K\"ahler manifold of dimension
$n=3$ with trivial canonical bundle. By Yau's theorem \cite{Yau78}, the bundle
$T^{1,0} X$ has a Ricci-flat metric $\omega_{\rm CY}$, and therefore
$T^{1,0}X$ is polystable with respect to $[\omega_{\rm CY}]$. In fact,
$T^{1,0}X$ is stable because it cannot holomorphically split. As noted in \cite{Yau93} , the de Rham
decomposition theorem implies that if the tangent bundle to Calabi-Yau manifold splits holomorphically, then the
manifold itself splits holomorphically and metrically as a product. In dimension $n=3$, at least one factor in this decomposition must be $1$-dimensional, and hence a torus.  When $X$ is simply connected, this is impossible. 
\smallskip
Thus, $(X,[\omega_{CY}])$ satisfies the stability condition
\[
{1 \over {\rm rk}(F)} \int_X c_1(F) \wedge \omega_{\rm CY}^2 < 0
\]
for all torsion-free, coherent subsheaves $F \subseteq T^{1,0}X$ of rank 1 or 2. It follows from Proposition \ref{prop: FLY2.6-resolution} that the same inequality holds if we replace $\omega_{CY}$ by $\omega_{a}$ for any $a>0$.

Therefore $T^{1,0}X$ is
stable with respect to the balanced classes $[\omega_{a}^2] \in H^{2,2}(X,\mathbb{R})$. By the Li-Yau \cite{LY86} generalization to Gauduchon metrics of the
Donaldson-Uhlenbeck-Yau theorem \cite{Donaldson,UY}, there exists a family of metrics $H_a$ on $T^{1,0}X$ such that
\[
F_{H_a} \wedge \omega_a^2 = 0.
\]
Our goal is to obtain a limiting metric $H_0$ as $a \rightarrow 0$.

\subsection{Reference metrics}

To study the sequence $(g_a,H_a)$, we will use a sequence of reference metrics $\hat{H}_a$, given by
\[
\hat{H}_a = e^{\psi_a} g_a
\]
where $\psi_a$ satisfies
\be \label{eq:reftwist}
\Delta_{g_a} \psi_a = {1 \over 3} {\rm Tr} \, \sqrt{-1}
\Lambda_{\omega_a} F_{g_a}, \quad \int_X \psi_a \, d {\rm vol}_{g_a} = 0.
\ee
The solvability of (\ref{eq:reftwist}) follows from the balanced condition since ${\rm Tr} \,
F_{g_a}$ is exact.  The advantage of $\hat{H}_a$ is that these metrics now have the property that
\[
{\rm Tr} \, \Lambda_{\omega_a} F_{\hat{H}_a}  =0.
\]
This follows from
\be \label{conf-change-F}
\sqrt{-1} \Lambda_{\omega_a} F_{\hat{H}_a} =- (\Delta_{g_a}
\psi_a) \, I + \sqrt{-1} \Lambda_{\omega_a} F_{g_a}.
\ee
If we form the relative endomorphism $\hat{h}_a = \hat{H}_a^{-1} H_a$,
then (\ref{diff-chern-curv}) implies
\[
\Delta_{g_a} \log \det \hat{h}_a = {\rm Tr} \, \sqrt{-1} \Lambda_{\omega_a}
F_{H_a} - \Tr \sqrt{-1} \Lambda_{\omega_a} F_{\hat{H}_a} = 0.
\]
Therefore, $\det \hat{h}_a$ is constant, and we will choose a normalization
for $H_a$ such that
\[
\det \hat{h}_a \equiv 1.
\]
We prove a uniform $C^0$ estimate for the conformal factor.
\begin{lem} \label{lem: confFactC0}
The sequence $\{ \psi_a \}$ satisfies a uniform bound
  \[
\| \psi_a \|_{L^\infty(X)} \leq C.
  \]
\end{lem}

\begin{proof}
Let $U(\delta)=\{ r < \delta \}$. Since $g_a = R
g_{co,a}$ on $U(\epsilon)$, we have $\Lambda_{\omega_a} F_{g_a} = 0$
on $U(\epsilon)$ and
\[
\Delta_{g_a} \psi_a = 0 \ {\rm on} \ U(\epsilon).
\]
By the maximum principle,
\[
\sup_{U(\epsilon)} |\psi_a| \leq \sup_{\partial U(\epsilon)} |\psi_a|.
\]
Define
\[
\varphi_a(z) = \sup_X \psi_a - \psi_a(z).
\]
Let $p \in X$ be a point where $0=\varphi_{a}(p) = \inf_{X} \varphi_a$. By the
above estimate, we may assume $p \in X \backslash U(\epsilon)$. Fix a
finite open cover $\{ V_i\}$ of coordinate charts of $K = X \setminus
U(\epsilon)$ such that the eigenvalues of $(g_a)^{i \bar{j}}$ are uniformly bounded
above and below on each chart. Denote by $B_i$ coordinate balls each
compact contained in $V_i$ which still cover $K$.  Note that since the metrics $g_{a}$ converge 
smoothly and uniformly on $X\setminus U(\epsilon)$ to the metric $g_0$ the sets $B_i, V_i$ can
be chosen independent of $a$.

Suppose $p \in B_1$ so that $\inf_{B_1} \varphi = 0$. By the
Harnack inequality for elliptic PDE (e.g. \cite[Theorem 5.10]{Han-Lin}),
\[
\sup_{B_1} \varphi_a \leq C \bigg(  \inf_{B_1} \varphi_a +
\| \Delta_{g_a} \varphi_a \|_{L^\infty} \bigg) = C \| \Delta_{g_a} \varphi_a \|_{L^\infty}.
\]
From Proposition~\ref{prop: FLY2.6-resolution},  the function ${\rm Tr} \, \Lambda_{\omega_a} F_{g_a}$ is supported on $U(1) \backslash U(\epsilon)$ and is bounded uniformly, independent of $a$. Therefore
\[
\sup_X |\Delta_{g_a} \varphi_a| \leq C.
\]
It follows that $\sup_{B_1} \varphi_a \leq C$. Let $B_2$
be another coordinate ball with $B_1 \cap B_2 \neq \emptyset$. By the
Harnack inequality,
\[
\sup_{B_2} \varphi_a \leq C \bigg( \inf_{B_1 \cap B_2} \varphi_a + \|\Delta_{g_a} \varphi_a\|_{L^\infty} 
\bigg) \leq C \bigg( \sup_{B_1} \varphi_a +1
\bigg) \leq C.
\]
Continuing this process for each $B_i$, we conclude
\[
\sup_K \varphi_a \leq C,
\]
and hence $\sup_X \varphi_a \leq C$. This gives a bound on the oscillation
\[
\sup_X \psi_a - \inf_X \psi_a \leq C.
\]
Since $\int_X \psi_a \, d {\rm vol}_{g_a} = 0$, there is a point $q \in X$ where $\psi_a(q) =0$. Therefore
\[
\sup_X |\psi_a| \leq C
\]
where $C$ is independent of $a>0$.
\end{proof}
\medskip
\par  We record some corollaries of this estimate.

\begin{cor} \label{cor: referenceH}
Let $H_{a}, \Hat{H}_{a}, g_{a}, h_{a}, \psi_{a}$ be as above, for $0<a \ll 1$.  Then we have
\begin{itemize}
\item[$(i)$] Let $Z= \cup_{i} C_i$ be the union of all $(-1,-1)$ curves contracted during the conifold transition.  On compact sets $K\subset X\setminus Z$, the metrics $\hat{H}_{a}$ converge smoothly to a limiting  metric $\hat{H}_0 = e^{\psi_0}g_0$.
\item[$(ii)$] For all $0\leq a \ll 1$, the metrics $\hat{H}_{a}$ are uniformly equivalent to the background metrics $g_{a}$.  That is, there is a uniform constant $C$, independent of $a$ such that
\[
C^{-1}g_{a} \leq \hat{H}_{a} \leq C g_{a}.
\]
In particular,  since $\det \hat{H}_a^{-1} H_a = 1$, the endomorphism $h_{a} = g_{a}^{-1}H_{a} = e^\psi_{a}\hat{H}_{a}^{-1}H_{a}$ satisfies
\[
C^{-1} \leq \det h_{a} \leq C
\]
for a uniform constant $C$, independent of $0\leq a \ll 1$.

\item[$(iii)$] The Hermitian-Yang-Mills tensor is bounded along the sequence
\[
    \sup_X |\sqrt{-1} \Lambda_{\omega_a} F_{\hat{H}_a} |_{\hat{H}_a} \leq C.
 \]
The full curvature of the limiting $\hat{H}_0$ satisfies
$  | F_{\hat{H}_0} |_{\hat{H}_0} \leq C r^{-2}$  on $\underline{X}_{reg}$.
\end{itemize}
\end{cor}
\begin{proof}
To prove convergence of $\hat{H}_a$ on a compact set $K$, we cover $K$
by finitely many coordinate charts and apply interior estimates for
the Laplace equation (\ref{eq:reftwist}) to $\psi_a$. On $K$, the
metrics $g_a$ converge uniformly to $g_0$, and hence after a
subsequence $\hat{H}_a = e^{\psi_a} g_a$ converges to a limiting
metric $\hat{H}_0$. The uniform bounds for $\hat{H}_a$ and $\det h_a$
follow from the $C^0$ estimate $\| \psi_a \|_{L^\infty} \leq C$.
\smallskip
\par On a neighborhood $U(\epsilon)$ containing the holomorphic
curves, we have $\Lambda_{\omega_a} F_{g_a}=0$ and $\Lambda_{\omega_a}
F_{\hat{H}_a}=0$. Outside of $U(\epsilon)$, the metrics
$(\hat{H}_a,\omega_a)$ are uniformly bounded, hence
$|\Lambda_{\omega_a} F_{\hat{H}_a}|_{\hat{H}_a} \leq C$.
  \smallskip
  \par The full curvature $F_{\hat{H}_0}$ does not vanish on
  $U(\epsilon)$, however in this neighborhood $|Rm_{g_0}|_{g_0} \leq C
  r^{-2}$. In $U(\epsilon)$ we have $\Delta_{g_0} \psi_0 = 0$, and
  estimates for the Laplacian in cylindrical coordinates imply
  \[
|\partial^2 \psi_0 |_{g_{euc}} \leq C
  \]
and hence $|\nabla^2_{g_0} \psi_0|_{g_0} \leq C r^{-2}$ by Lemma
\ref{lem: cylCoords}. It follows that $|F_{\hat{H}_0}|_{\hat{H}_0}
\leq C r^{-2}$ since $g_0$ and $\hat{H}_0$ are uniformly equivalent.

\end{proof}

\subsection{Uhlenbeck-Yau $C^0$ estimate}
In this section, we derive the following estimate:

\begin{prop} Along the sequence of endomorphisms $h_a = g_a^{-1} H_a$,
  we have the uniform $C^0$ estimate
\be \label{C0-est2}
C^{-1} I \leq h_a \leq C I
\ee
where $I$ denotes the identity endomorphism.
\end{prop}

We will prove
this by following the argument of Uhlenbeck-Yau \cite{UY}. Thanks to the
estimate $C^{-1} \leq \det h_a \leq C$, it suffices to show
\be \label{C0-est}
{\rm Tr} \, h_a \leq C.
\ee
Rather than $h_a = g_a^{-1} H_a$, we will work with the reference
metric $\hat{H}_a = e^{\psi_a} g_a$ from the previous section and
relative endomorphism $\hat{h}_a = \hat{H}_a^{-1} H_a $. The estimate $\| \psi_a \|_{L^\infty(X)} \leq C$ in Lemma~\ref{lem: confFactC0} shows that a bound ${\rm Tr} \, \hat{h}_a \leq C$ implies (\ref{C0-est}). 
\smallskip
\par To prove (\ref{C0-est}), suppose on the contrary that ${\rm Tr} \, \hat{h}_a \rightarrow \infty$ as $a \rightarrow 0$. Let
\[
\tilde{h}_a = {\hat{h}_a \over \sup_X {\rm Tr} \, \hat{h}_a}.
\]
The starting point in the proof of the $C^0$ estimate of Uhlenbeck-Yau is the following inequality (see \cite[equation (4.6)]{UY});
\begin{lem}\label{lem: ineq-K-0}
Fix $0<\sigma \leq 1$, and any two metrics $\hat{H},H$ on $T^{1,0}X \rightarrow X$.  Let $h= \hat{H}^{-1} H$, and let $g$ be a Hermitian metric on $X$.  Then we have
\be \label{ineq-K-0}
| h^{-\sigma/2} \hat{\nabla} h^\sigma|_{\hat{H},g}^2  \leq g^{j \bar{k}} \langle h^{-1} \hat{\nabla}_j h, \hat{\nabla}_k h^\sigma \rangle_{\hat{H}}
\ee
where $\hat{\nabla}$ is the Chern connection of $\hat{H}$.
\end{lem}

We rewrite~\eqref{ineq-K-0} using the identity
\be \label{partial-Tr-sigma}
\partial_j {\rm Tr} \, h^\sigma = \sigma \langle h^{-1} \hat{\nabla}_i h, h^\sigma \rangle_{\hat{H}},
\ee
which implies
\[
{1 \over \sigma} \Delta_{g} {\rm Tr} \, h^\sigma = g^{j \bar{k}} \partial_{\bar{k}} \langle h^{-1} \hat{\nabla}_j h, h^\sigma \rangle_{\hat{H}}.
\]
Therefore, ~(\ref{ineq-K-0}) is equivalent to
\bea \label{ineq-K}
| h^{-\sigma/2} \hat{\nabla} h^\sigma|_{\hat{H},g}^2 -{1 \over \sigma} \Delta_g {\rm Tr} \, h^\sigma  &\leq& g^{j \bar{k}} \langle h^{-1} \hat{\nabla}_j h, \hat{\nabla}_k h^\sigma \rangle_{\hat{H}} - g^{j \bar{k}} \partial_{\bar{k}} \langle h^{-1} \hat{\nabla}_j h, h^\sigma \rangle_{\hat{H}} \nonumber\\
&=&  - g^{j \bar{k}} \langle  \partial_{\bar{k}} (h^{-1} \hat{\nabla}_j h), h^\sigma \rangle_{\hat{H}}.
\eea
We will make use of inequality (\ref{ineq-K}) by relating the
right-hand side to the curvature tensor.  With the same notation as
Lemma~\ref{lem: ineq-K-0}, the  difference between curvatures of  the
Chern connections (\ref{diff-chern-curv}) defined by $H,\hat{H}$ shows that the key inequality (\ref{ineq-K}) can be written as
\[
| h^{-\sigma/2} \hat{\nabla} h^\sigma|_{\hat{H},g}^2 -{1 \over \sigma} \Delta_g {\rm Tr} \, h^\sigma  \leq \langle (\sqrt{-1} \Lambda_\omega F - \sqrt{-1} \Lambda_\omega \hat{F}), h^\sigma \rangle_{\hat{H}}.
\]
In our case, applying this inequality to the Hermitian-Yang-Mills metric $H_a$ and the reference metric $\hat{H}_a$, we obtain
\be \label{ineq-K-2}
| \tilde{h}_a^{-\sigma/2} \hat{\nabla} \tilde{h}_a^\sigma|^2_{\hat{H}_a,g_a} -{1 \over \sigma} \Delta_{\omega_a} {\rm Tr} \, \tilde{h}^\sigma_a  \leq - \langle \sqrt{-1} \Lambda_{\omega_a} F_{\hat{H}_a}, \tilde{h}_a^\sigma \rangle_{\hat{H}_a}.
\ee
Corollary~\ref{cor: referenceH} gives the bound $|\Lambda_{\omega_a} F_{\hat{H}_a}|_{\hat{H}_a} \leq C$
which together with and $0<\tilde{h}_a \leq I$ yields the estimate
\be \label{ineq-K-3}
| \tilde{h}_a^{-\sigma/2} \nabla \tilde{h}_a^\sigma|_{\hat{H}_{a}, g_a}^2 -{1 \over \sigma} \Delta_{g_a} {\rm Tr} \, \tilde{h}^\sigma_a  \leq C,
\ee
where $C$ is independent of $a,\sigma$. Integrating both sides using
the balanced condition gives
\be \label{nabla-sigma/2}
\int_X | \tilde{h}_a^{-\sigma/2} \hat{\nabla} \tilde{h}_a^\sigma|_{\hat{H}_a,g_a}^2 \, d
{\rm vol}_{g_a}  \leq C
\ee
by Proposition \ref{prop: FLY2.6-resolution} ($v$). Since $0 \leq \tilde{h}_a \leq I$, this implies
\be \label{eq: gaW12Est}
\int_X | \hat{\nabla} \tilde{h}_a^\sigma|^2_{\hat{H}_a,g_a} \, d
{\rm vol}_{g_a}   \leq C.
\ee
Let $K$ be a compact set which is the closure of an open set $K^{o}$
satisfying $K^{o} \cap Z = \emptyset$, where $Z = \cup_{i=1}^{k} C_i$ is the union of all $(-1,-1)$ curves being contracted. The metrics $(g_a, \hat{H}_a)$ are uniformly equivalent to $(g_1,\hat{H}_1)$ on $K$. Then
\be \label{W12-est}
\int_K | \hat{\nabla} \tilde{h}_a^\sigma|^2_{\hat{H}_1,g_1} \, d
{\rm vol}_{g_1}  \leq C_K,
\ee
where $\nabla$ is with respect to $g_1$. For each $0 < \sigma \leq 1$, we have weak convergence $
\tilde{h}_{a_k}^\sigma \rightharpoonup \tilde{h}_\infty^\sigma $ in $W^{1,2}(K)$ along a subsequence; here $W^{1,2}(K)$ denotes the Sobolev space defined by $(g_1,\hat{H}_1)$. By a diagonal argument, there is a subsequence $a_i \rightarrow 0$ satisfying
\[
\tilde{h}_{a_i}^\sigma \rightharpoonup \tilde{h}_\infty^\sigma
\]
in $W^{1,2}(K)$ for all $\sigma \in \{ 1 /n : n \in \mathbb{N}\}$. By semicontinuity of weak convergence, we have the estimate
\[
\int_K |\nabla h_\infty^\sigma|^2_{g_1} d {\rm vol}_{g_1} \leq \limsup_i \int_K | \nabla \tilde{h}_{a_i}^\sigma|^2_{g_1} \, \dvol_{g_1} \leq C_K.
\]
Let $\sigma_i = 1/i$, and define $\pi \in W^{1,2}(K)$ by 
\[
(I - h_\infty^{\sigma_i}) \rightharpoonup \pi
\]
in the weak limit $i \rightarrow \infty$ in $W^{1,2}(K)$. Exhausting $X \backslash Z$ with compact sets $K$, we obtain an endomorphism $\pi \in \Gamma(X\setminus Z, {\rm End} \, T^{1,0}X)$ with regularity $\pi \in W^{1,2}_{loc}(X\setminus Z)$. The definition of $\pi$ is such that it is the projection onto ${\rm Ker} \, h_\infty$, and it satisfies $\pi=\pi^{\dagger_{\hat{H}_0}}=\pi^2$ almost everywhere.
\smallskip
\par We need to verify
\begin{lem}
The projection $\pi$ is not trivial, in the sense that it is neither the identity nor the zero projection.
\end{lem}
\begin{proof} We show that $h_\infty$ is not identically zero.  We will use repeatedly the following inequality, which is a consequence of (\ref{ineq-K-2}) with $\sigma=1$;
\be\label{eq: ineqDropGrad}
\Delta_{g_a} {\rm Tr} \, \tilde{h}_a \geq \langle \sqrt{-1}\Lambda_{\omega_a} F_{\hat{H}_a}, \tilde{h}_a \rangle_{\hat{H}_a}.
\ee

By its normalization, we have ${\rm Tr} \, \tilde{h}_a \leq 1$ and there exists $x_a \in X$ such that $({\rm Tr} \, \tilde{h}_a)(x_a) = 1$. In $U(\epsilon)$, we have $\Lambda_{\omega_a} F_{\hat{H}_a} = 0$ and hence
\[
\Delta_{g_a} {\rm Tr} \, \tilde{h}_a \geq 0, \quad {\rm in } \ U(\epsilon).
\]
In particular, by the maximum principle, $\sup_{U(\epsilon)} {\rm Tr} \, \tilde{h}_a \leq \sup_{\partial
  U(\epsilon)} {\rm Tr} \, \tilde{h}_a$.  Thus, we may assume $x_a \in \{
r \geq \epsilon \}$. 

The metrics $g_a$ are uniformly equivalent on $\{
r \geq \epsilon/2 \}$.  In particular, we can fix a uniform number
$0<\delta \ll 1$ such that there is a coordinate ball $B_{\delta}(x_a) \subseteq \{ r \geq \epsilon /2 \}$
and, in local coordinates on $B_{\delta}(x_a)$ there is a uniform constant $M$, independent of $a$, such that
the eigenvalues of $(g_{a})_{\bar{k}j}$ are bounded above by $M$ and below by $M^{-1}$.

From~\eqref{eq: ineqDropGrad} we obtain the estimate
\[
\Delta_{g_a} {\rm Tr} \, \tilde{h}_a  - C \, {\rm Tr} \, \tilde{h}_a \geq 0.
\]
Applying the Moser iteration (e.g. \cite[Theorem 4.1]{Han-Lin}) gives
\[
1 = \sup_{B_{\delta}(x_a)} {\rm Tr} \, \tilde{h}_a \leq C \| {\rm Tr} \, \tilde{h}_a \|_{L^1( \{ r \geq \epsilon/2 \}, \dvol_{g_{1}})}
\]
for a uniform constant $C$. Let $K = \{ r \geq \epsilon /2 \}$.  By~(\ref{W12-est}) and Rellich's theorem, we have $\tilde{h}_a \rightarrow h_\infty$ in $L^1(K, d\vol_{g_{1}})$ and
\[
\| {\rm Tr} \, h_\infty \|_{L^1(K, \dvol_{g_1})} \geq C^{-1},
\]
therefore $h_\infty$ is not identically zero.

Finally, note that since $\tilde{h}_a$ converges to $h_{\infty}$ pointwise almost everywhere on $K$, and $\det \tilde{h}_{a} = (\sup_{X}\Tr \hat{h}_{a})^{-3} \rightarrow 0$ we see that $h_{\infty}$ has a non-trivial kernel almost everywhere on $K$.  Hence $\pi \ne 0$.
\end{proof}

\medskip
\par We have thus constructed a nontrivial projection $\pi$ which projects onto the kernel of $h_\infty$. To obtain a holomorphic subbundle, we need a further holomorphic condition on $\pi$.  Following Uhlenbeck-Yau \cite{UY} we have,

\begin{lem}
The projection $\pi$ satisfies $(I-\pi) \bar{\partial} \pi = 0$ on $X
\setminus Z$.
  \end{lem}

\begin{proof}
Following \cite{UY}, rather than work with $(I-\pi) \bar{\partial} \pi$, we differentiate $(I-\pi)\pi = 0$ to obtain
\[
|(I-\pi) \bar{\partial} \pi|^2_{\hat{H}_0,g_0} = |\bar{\partial} (I-\pi) \pi|^2_{\hat{H}_0,g_0} .
\]
Taking the adjoint with respect to $\hat{H}_0$ and using $\pi^\dagger
= \pi$ and $(\hat{\nabla}_i s)^\dagger = \partial_{\bar{i}} s$ for
self-adjoint endomorphisms $s$, we obtain
\[
|(I-\pi) \bar{\partial} \pi|^2_{\hat{H}_0,g_0} = |\pi \hat{\nabla} (I-\pi)|^2_{\hat{H}_0,g_0}
\]
where $\hat{\nabla}$ is the covariant derivative with respect to
$\hat{H}_0$. We approximate the integral of the quantity on the right-hand side by
\[
\int_X | (I- \tilde{h}^s_a) \nabla_{\hat{H}_a} \tilde{h}^\sigma_a|^2_{\hat{H}_a,g_a} \dvol_{g_a}.
\]
The elementary inequality
\[
\tilde{h}^{-\sigma/2} \geq { 2 s + {1 \over 2} \sigma \over s} (I- \tilde{h}^s_a)
\]
and inequality (\ref{nabla-sigma/2}) implies
\bea
& \ & \int_X | (I- \tilde{h}^s_a) \nabla_{\hat{H}_a} \tilde{h}^\sigma_a|^2_{\hat{H}_a,g_a} d
{\rm vol}_{g_a} \nonumber\\
&\leq& \bigg( {s \over 2s + {1 \over 2} \sigma} \bigg)^2 \int_X |h^{-\sigma/2}_a \nabla_{\hat{H}_a} h^\sigma_a|^2_{\hat{H}_a,g_a} d
{\rm vol}_{g_a} \leq C \bigg( {s \over 2s + {1 \over 2} \sigma} \bigg)^2.
\eea
Let $U_\delta = \{ r > \delta \}$. Then since $\tilde{h}_a^s \rightarrow \tilde{h}^s_\infty$ in $L^2(U_\delta)$ by (\ref{W12-est}) and Rellich's theorem, and $\tilde{h}_a^s \rightharpoonup \tilde{h}^s_\infty$ weakly in $W^{1,2}(U_\delta)$, we have that
\[
(I- \tilde{h}^s_a) \nabla_{\hat{H}_a} \tilde{h}^\sigma_a \rightharpoonup (I- \tilde{h}^s_\infty) \hat{\nabla} \tilde{h}^\sigma_\infty \quad {\rm weakly} \ {\rm in} \ L^2(U_\delta).
\]
We let $a \rightarrow 0$ and use semi-continuity of weak convergence to obtain
\[
\int_{U_\delta} | (I- \tilde{h}^s_\infty) \hat{\nabla} \tilde{h}^\sigma_\infty|^2_{\hat{H}_0,g_0} d
{\rm vol}_{g_0} \leq C \bigg( {s \over 2s + {1 \over 2} \sigma} \bigg)^2.
\]
We now let $s \rightarrow 0$, which implies
\[
\int_{U_\delta} | \pi \hat{\nabla} \tilde{h}^\sigma_\infty|^2_{\hat{H}_0,g_0} d
{\rm vol}_{g_0} \leq 0,
\]
and then taking $\sigma \rightarrow 0$, we conclude
\[
\int_{U_\delta} |\pi \hat{\nabla} (I-\pi)|^2_{\hat{H}_0,g_0} d {\rm vol}_{g_0} = 0,
\]
using semi-continuity of weak convergence.
\end{proof}

\medskip
\par Altogether, we have produced an endomorphism $\pi \in \Gamma( {\rm End} \, T^{1,0}X\big|_{X \setminus Z})$ such that
\smallskip

\par $\bullet$ $\pi \in W^{1,2}_{loc}(X \setminus Z)$ with respect to
the norms $(\hat{H}_0,g_0)$. 
\par $\bullet$ $\pi = \pi^\dagger = \pi^2$, where $\dagger$ is with
respect to $\hat{H}_0$.
\par $\bullet$ $(I-\pi) \bar{\partial} \pi = 0$
\smallskip
\par We will need a more precise $L^2$ bound on $|\hat{\nabla} \pi|^2$.

\begin{lem} \label{lem:na-pi-L2} For any $\delta>0$, we can estimate
  \[
\int_{\{r > \delta \}} | \hat{\nabla} \pi |^2_{\hat{H}_0,g_0} \, d
{\rm vol}_{g_0} \leq \int_{\{ r > \delta \}} ({\rm Tr} \, \Lambda_{\omega_0} F_{\hat{H}_0} \pi) d
{\rm vol}_{g_0}  .
\]
\end{lem}

\begin{proof} We work on the set $U_\delta = \{ r > \delta \}$ where $I- \tilde{h}^{\sigma_i}_{a_i}$ converges weakly as $i \rightarrow \infty$ to $\pi$ in $W^{1,2}(U_\delta)$ and $\hat{H}_a$ converges to $\hat{H}_0$ in $C^\infty(U_\delta)$.
\bea
\int_{U_\delta} ({\rm Tr} \, \sqrt{-1} \Lambda_{\omega_0} F_{\hat{H}_0} \pi) d {\rm vol}_{g_0} &=& \int_{U_\delta} {\rm Tr} \, [ \sqrt{-1} \Lambda_{\omega_0} F_{\hat{H}_0}  (\pi-I)] d {\rm vol}_{g_0} \nonumber\\
&=& -\lim_{i \rightarrow \infty} \int_{U_\delta} {\rm Tr} \, ( \sqrt{-1} \Lambda_{\omega_i} F_{\hat{H}_i} \tilde{h}^{\sigma_i}_{i})  d {\rm vol}_{g_i}
\eea
In the first equality we used ${\rm Tr} \, \Lambda_{\omega_0} F_{\hat{H}_0} = 0$. Using the formula (\ref{diff-chern-curv}) for the difference between the curvature tensors $F_{\hat{H}}$ and $F_H$, we obtain
\[
\int_{U_\delta} ({\rm Tr} \, \Lambda_{\omega_0} F_{\hat{H}_0} \pi) d {\rm vol}_{g_0} =  -\lim_{i \rightarrow \infty} \int_{U_\delta} {\rm Tr} \, g^{j \bar{k}} \partial_{\bar{k}} (\tilde{h}_i^{-1} \hat{\nabla}_j \tilde{h}_i)  \tilde{h}^{\sigma_i}_{i}) d {\rm vol}_{g_i}.
\]
The inequality (\ref{ineq-K-0}) can be written as
\[
|h^{-\sigma/2} \hat{\nabla} h^\sigma|^2_{\hat{H},g} - {1 \over \sigma} \Delta_g {\rm Tr} \, h^\sigma \leq - g^{j \bar{k}} \langle \partial_{\bar{k}} (h^{-1} \hat{\nabla}_j h), h^\sigma \rangle_{\hat{H}}.
\]
Since $h^\dagger=h$ (with respect to $\hat{H}$) and $\langle u,v \rangle_{\hat{H}} = {\rm Tr} (u v^\dagger)$, we obtain
\[
\int_{U_\delta} ({\rm Tr} \, \Lambda_{\omega_0} F_{\hat{H}_0} \pi) d {\rm vol}_{g_0} \geq \lim_{i \rightarrow \infty}  \int_{U_\delta} |\tilde{h}_i^{-\sigma/2} \hat{\nabla} \tilde{h}^{\sigma_i}_i|^2_{\hat{H}_i,g_i} d {\rm vol}_{g_i} - \lim_{i \rightarrow \infty} \int_{U_\delta} {1 \over \sigma} \Delta_{\omega_i} {\rm Tr} \, \tilde{h}_i^\sigma \, d {\rm vol}_{g_i}.
\]
By the balanced condition, this is
\[
\int_{U_\delta} ({\rm Tr} \, \Lambda_{\omega_0} F_{\hat{H}_0} \pi) d {\rm vol}_{g_0} \geq \lim_{i \rightarrow \infty}  \int_{U_\delta} |\tilde{h}_i^{-\sigma/2} \hat{\nabla} \tilde{h}^{\sigma_i}_i|^2_{\hat{H}_i,g_i} d {\rm vol}_{g_i} + \lim_{i \rightarrow \infty} \int_{\{r < \delta \}} {1 \over \sigma} \Delta_{\omega_i} {\rm Tr} \, \tilde{h}_i^\sigma \, d {\rm vol}_{g_i}.
\]
Let $0< \delta < \epsilon$, where $\epsilon$ is the transition radius in the construction of $\omega_a$, so that $\Lambda_{\omega_a} F_{\hat{H}_a} = 0$ on $\{ r < \delta \}$. It follows from (\ref{ineq-K-2}) that
\[
{1 \over \sigma} \Delta_{\omega_i} {\rm Tr} \, \tilde{h}_i^\sigma \geq 0, \quad {\rm on} \ \{ r < \delta \}.
\]
Combining this with $|\hat{\nabla} \tilde{h}^\sigma|^2 \leq | \tilde{h}^{-\sigma/2} \hat{\nabla} \tilde{h}^\sigma|^2 $, we obtain
\[
\int_{U_\delta} ({\rm Tr} \, \Lambda_{\omega_0} F_{\hat{H}_0} \pi) d {\rm vol}_{g_0} \geq  \lim_{i \rightarrow \infty} \int_{U_\delta} | \hat{\nabla} \tilde{h}_i^{\sigma_i}|^2_{\hat{H}_i,g_i} d {\rm vol}_{g_i}.
\]
We conclude by semi-continuity of weak convergence. 
\end{proof}

\medskip
\par We now apply the work of Uhlenbeck-Yau \cite{UY} (see also \cite{Popo}) to conclude that, at least over $X\setminus Z$, the projection $\pi$ defines a coherent subsheaf $\mathcal{E} \subset T^{1,0}X|_{X\setminus Z}$, which is locally free outside of a complex codimension $2$ set.  Let $k>0$ be the generic rank of $\mathcal{E}$.  We can view $\mathcal{E}$ as defining a meromorphic map 
\[
\mu_{\mathcal{E}}: X\setminus Z\rightarrow {\rm Gr}(k, T^{1,0}X)
\]
to the Grassmann bundle of $k$-planes in $T^{1,0}X$.  Locally near a point
in $Z$ we can trivialize $T^{1,0}X$ and, by taking Pl\"ucker coordinates on
${\rm Gr}(k,T^{1,0}X)$, we view $\mu_{\mathcal{E}}$ as a collection of
meromorphic functions defined on the complement of $Z$.  On the
other hand, since $Z$ has complex codimension $2$, a classical result
of Levi \cite{Levi} (see also \cite[Chapter 2]{DemBook}) implies that $\mu$ extends over $Z$.  It follows
that $\mathcal{E}$ extends over $Z$ to a coherent sheaf (also denoted
by $\mathcal{E}$) by taking the direct image of the tautological
bundle over ${\rm Gr}(k, T^{1,0}X)$.  We have thus produced a coherent sheaf
$\mathcal{E}\subset T^{1,0}X$, locally free outside a codimension $2$ set
$Z'$.  We will show that this sheaf contradicts the stability of $T^{1,0}X$.

\smallskip
\par To contradict stability, we need to show that
\[
c_1(\mathcal{E}) \cdot [\omega_{\rm CY}]^2 \geq 0.
\]
The only reason this does not follow immediately from the standard argument is that the metrics $\hat{H}_{0}$ and $\omega_0$ are not smooth on $X$.  Thus we need to show that the singularities do not contribute.  Denote by $\hat{H}_0'$ the metric induced by $\hat{H}_0$ on the subbundle $\mathcal{E}|_{X \setminus Z'} \subset T^{1,0} X|_{X \setminus Z'}$. We begin by computing the slope defined by $\hat{H}_0'$ and $\omega_0$.  Let us introduce the notation
\[
c_1(\mathcal{E}, \hat{H}_0') \cdot \omega_0^2 = \int_{X \setminus Z'} {\rm Tr} \,
\sqrt{-1} F_{\hat{H}'_0} \wedge {\omega_0^2 \over 2}.
\]
The identity for the curvature of
the induced connection on a subbundle defined by a projection $\pi$ is
( see, e.g. \cite[equation (4.16)]{UY})
\[
{\rm Tr} \, \sqrt{-1} \Lambda_{\omega_0} F_{\hat{H}_0'} = {\rm Tr} \, \sqrt{-1} \Lambda_{\omega_0} F_{\hat{H}_0} \pi - | \hat{\nabla} \pi|^2_{\hat{H}_0,g_0}.
\]
Therefore
\[
c_1(\mathcal{E}, \hat{H}_0') \cdot \omega_0^2  = \int_{X \setminus Z'} \bigg[  ({\rm Tr} \,
\sqrt{-1} \Lambda_{\omega_0} F_{\hat{H}_0} \pi)  - | \hat{\nabla}
\pi|^2_{\hat{H}_0,g_0} \bigg] d {\rm vol}_{g_0}.
\]
By letting $\delta \rightarrow 0$ in Lemma \ref{lem:na-pi-L2}, we see
that
\[
c_1(\mathcal{E}, \hat{H}_0') \cdot \omega_0^2 \geq 0.
\]
We note that this quantity is finite. We can estimate the endomorphism $\Lambda_{\omega_0}
(F_{\hat{H}_0} \pi)$ by using that $|F_{\hat{H}_0}|_{\hat{H}_0,g_0} \leq C
r^{-2}$, thanks to Corollary \ref{cor: referenceH}, and $|\pi|_{\hat{H}_0} \leq C$ since
$\pi=\pi^{\dagger_{\hat{H}_0}}=\pi^2$. Therefore
\[
\bigg| \int_{X\setminus Z}  ({\rm Tr} \,
\sqrt{-1} \Lambda_{\omega_0} F_{\hat{H}_0} \pi) \, d {\rm vol}_{g_0} \bigg| \leq C
\int_{X \setminus Z} r^{-2} d {\rm vol}_{g_0} \leq C,
\]
using that, near $\{r=0\}$, the metric $g_0$ is a cone over a five-dimensional link. 
\smallskip
\par The next step is to show that $c_1(\mathcal{E}, \hat{H}_0') \cdot
\omega_0^2$ is equal to $c_1(\mathcal{E}) \cdot [\omega_{CY}]^2$.  Recall that $\mathcal{E}$ defines a line bundle $L=\det
\mathcal{E}$, and if $e^\varphi$ is a smooth metric on $L$ then $\beta = -\ddb \varphi$ is a
representative of $c_1(\mathcal{E})$.  For concreteness, we let
$\beta$ be the curvature form associated to the metric
$g_{CY}\big|_{\mathcal{E}}$. We write the difference as
\bea
& \ & c_1(\mathcal{E})\wedge[\omega_{\rm CY}]^2 - c_1(\mathcal{E},\hat{H}_0').\omega_0^2\nonumber\\
&=& \int_{X \setminus Z'} (\beta-\sqrt{-1} \Tr
F_{\hat{H}'_0})\wedge \omega_0^2 + \int_{X \setminus Z} \beta \wedge
(\omega_{\rm CY}^2-\omega_0^2) \nonumber\\
&:=& ({\rm I}) + ({\rm II}) .\nonumber
\eea
We will treat each term individually. Recall $Z'$ is a codimension 2
analytic set containing the singularities of $\mathcal{E}$ (which
contains $Z = \cup C_i$). Let $\eta_\delta$ be a cutoff
function such that $\eta_\delta \equiv 1$ on $\{
{\rm dist}_{g_{CY}}(Z',\cdot) > 2 \delta \}$ and $\eta_\delta \equiv 0$ on $\{
{\rm dist}_{g_{CY}}(Z',\cdot) < \delta \}$ with $|\ddb
\eta_\delta|_{g_{CY}} \leq C \delta^{-2}$.
\par $\bullet$ Term $({\rm I})$. Working near a point where $\mathcal{E}$ is locally free we have
\[
\beta-\sqrt{-1}\Tr F_{\hat{H}_0'} = -\ddb \log \left(\frac{\det g_{CY}\big|_{\mathcal{E}}}{\det \hat{H}_0\big|_{\mathcal{E}}}\right).
\]
Note that by the AM-GM inequality we have
\[
\left(\frac{\det g_{CY}\big|_{\mathcal{E}}}{\det \hat{H}_0\big|_{\mathcal{E}}}\right)^{\frac{1}{{\rm rk}(\mathcal{E})}} \leq \frac{1}{{\rm rk}(\mathcal{E})} \Tr\left(\hat{H}_0\big|_{\mathcal{E}}^{-1} g_{CY}\big|_{\mathcal{E}}\right) \leq \frac{1}{{\rm rk}(\mathcal{E})} \Tr\left(\hat{H}_0^{-1} g_{CY}\right)
\]
where in the second inequality we used that $\hat{H}_0^{-1} g_{CY}$ is positive definite.  Similarly we have
\[
\left(\frac{\det \hat{H}_0\big|_{\mathcal{E}}}{\det g_{CY}\big|_{\mathcal{E}}}\right)^{\frac{1}{{\rm rk}(\mathcal{E})}}  \leq \frac{1}{{\rm rk}(\mathcal{E})} \Tr\left(g_{CY}^{-1} \hat{H}_0\right).
\]
  On the other hand from ~\eqref{eq: compare-cy-co} we have
\[
{\rm dist}_{g_{CY}}(Z,\cdot)^{4/3} g_{CY} \leq g_{0} \leq {\rm dist}_{g_{CY}}(Z,\cdot)^{-2/3} g_{CY}
\]
since $r^3 = \| z \|^2 \sim {\rm dist}^2_{g_{CY}}(Z,\cdot)$ near the
singular points. Hence the same estimates hold for $\hat{H}_0$ and so we have
\[
\bigg|\log \left(\frac{\det g_{CY}\big|_{\mathcal{E}}}{\det \hat{H}_0\big|_{\mathcal{E}}}\right)\bigg| \leq -C\log{\rm dist}_{g_{CY}}(Z,\cdot) + C
\]
for a uniform constant $C$. Integrating by parts gives
\[
\int_{X\setminus Z'} \eta_{\delta}  (\beta-\sqrt{-1} \Tr
F_{\hat{H}'_0})\wedge \omega_0^2  =\int_{X\setminus Z'}\log \left(\frac{\det g_{CY}\big|_{\mathcal{E}}}{\det \hat{H}_0\big|_{\mathcal{E}}}\right) \ddb \eta_{\delta} \wedge \omega_0^2.
\]
From the definition of $\eta_{\delta}$ and the bound $\omega_0 \leq {\rm dist}_{g_{CY}}(Z,\cdot)^{-2/3} \omega_{CY}$ we get
\[
\bigg|\int_{X\setminus Z'} \eta_{\delta}  (\beta-\sqrt{-1} \Tr
F_{\hat{H}'_0})\wedge \omega_0^2 \bigg| \leq C \delta^{4-2-4/3}(-\log(\delta))  
\]
since ${\rm Vol}_{g_{CY}}(\{x \in X : \delta < {\rm
  dist}_{g_{CY}}(Z',x) < 2\delta\})\sim \delta^4$. It follows that
term $({\rm I})$ vanishes.

\smallskip
\par $\bullet$ Term $({\rm II})$. By Proposition \ref{prop: FLY2.6-resolution},
\[
({\rm II}) = \int_{U \backslash Z} \beta \wedge (\omega_{CY}^2-\omega_0^2)
\]
where $U$ is a tubular neighborhood of $Z$, which is a disjoint union of tubular neighborhoods of the $(-1,-1)$ rational curves.  Since $U$ retracts onto $Z$, which has complex dimension $1$ (and hence $H^{4}(U)=0$) we can write
\[
\omega_{CY}^2 = d\Phi
\]
for a smooth $3$-form $\Phi$.  On the other hand, by Proposition \ref{prop: FLY2.6-resolution} $\omega_{co}^2 = \ddb \Phi'$ where, $\Phi'$ is smooth in $U\setminus Z$ and near each $(-1,-1)$ curve $C_i$ there is a constant $\lambda_i>0$ such that we have
\[
 \Phi'= \lambda_i r^2 \ddb r^2 = \lambda_i r^2 \omega_0
 \]
where we recall that $r^2= \|z\|^{4/3}$.  Thus, we have
\[
\int_{X\setminus Z} \eta_{\delta} \beta\wedge( \omega_{CY}^2- \omega_{0}^2) = \int_{U}d\eta_{\delta} \wedge \beta \wedge \Phi - \int_{U\setminus Z} \ddb \eta_{\delta} \wedge \beta \wedge \Phi'.
\]
The first integral is easily seen to be of order $\delta^{3}$.  For the second integral, we use the bound $g_0 \leq {\rm dist}_{g_{CY}}(Z,\cdot)^{-2/3}g_{CY}$ together with the definition of $\eta_{\delta}$ and $r$ to conclude
\[
\big|\int_{U\setminus Z} \ddb \eta_{\delta} \wedge \beta \wedge \Phi'\big| \leq C\delta^{4-2+\frac{4}{3}-\frac{2}{3}}.
\]
It follows that term ${\rm (II)}$ vanishes, and hence
\[
c_1(\mathcal{E}, \hat{H}_0') \cdot
\omega_0^2 = c_1(\mathcal{E}) \cdot [\omega_{CY}]^2
\]
and hence $c_1(\mathcal{E}) \cdot [\omega_{\rm CY}]^2 \geq 0$, which
contradicts the stability of $T^{1,0}X$. We conclude that $\sup_X {\rm Tr} \,
\hat{h}_a \leq C$ as $a \rightarrow 0$, which proves the $C^0$
estimate (\ref{C0-est2}).

\subsection{Gradient estimate} In this section, we will show that, along
the sequence $(g_a,H_a)$, there holds an estimate of the form
\be \label{C1-est}
|\nabla_{g_a} H_a|_{g_a} \leq C r^{-1}.
\ee
We will use the ideas from Calabi's $C^3$ estimate \cite{CalabiC3}, as applied in
complex geometry by Yau \cite{Yau78} and further developed by Phong-Sesum-Sturm
\cite{PSS} (for other applications of this technique, see e.g. \cite{FPPZ,SherWein,Tosatti10}). We will work with the quantity
\[
S = |\nabla_{H_a} h_a h_a^{-1}|^2_{g_a,H_a},
\]
where $h_a = g^{-1}_a H_a$ and by the mixed norms we mean
\[
|\nabla h_a h_a^{-1}|^2_{g_a,H_a} = (g_a)^{j \bar{k}} (H_a)_{\bar{\beta} \alpha} (H_a)^{\mu \bar{\nu}} (\nabla_j h_a h_a^{-1})^\alpha{}_\mu \overline{(\nabla_k h_a h_a^{-1})^\beta{}_\nu}.
\]
The quantity $S$ can be understood as the difference of
connections by the formula
\be \label{diff-connections}
\nabla_H h h^{-1} = A_{H} - A_{g},
\ee
where
\[
A_{H} = H^{-1} \partial H, \quad A_{g} = g^{-1} \partial g.
\]
For ease of notation, in this section we omit the sequence subscript $a$. We will obtain the following estimate.

\begin{prop} \label{prop: gradest}
Let $(X,g)$ be a compact Hermitian complex manifold with smooth function
$r:X \rightarrow [0,\infty)$ satisfying $|\nabla r|_g \leq \Lambda$ for $\Lambda >0$. Let $H$ be a second Hermitian metric on $T^{1,0}X$
satisfying $g^{j \bar{k}} (F_H)_{j \bar{k}}{}^\alpha{}_\beta = 0$ and
$ C_0^{-1} g \leq H \leq C_0 g$. Let $\epsilon >0$. Suppose that on
$\{ r \leq \epsilon \}$, the metric $g$ is K\"ahler Ricci-flat, and satisfies
\[
|Rm_g|_g \leq C_1 r^{-2}.
 \]
Suppose on the set $ \{ r > \epsilon \}$, we have the estimate
\[
|T_g|+ |Rm_g| + |\nabla^g Rm_g| \leq \Lambda,
\]
where $T_g$ is the torsion of $g$ and $Rm_g$ is the curvature of $g$. Then
\be 
r^2 |\nabla_H h h^{-1}|^2_{g,H} \leq C(C_0, C_1, \Lambda,\epsilon)
\ee
where $h= g^{-1} H$ and $\nabla_H$ is the Chern connection of $H$.
\end{prop}

\par The sequence $(g_a,H_a)$ satisfies the hypothesis of the
proposition on $X$ with a function $r$ which is an extension of $\| z \|^{2/3}$ from
$U(1/2)$ to all
of $X$ with $r^{-1}(0) = \cup C_i$, and the constants are uniform in
$a$. Indeed, the uniform bounds
\[
|Rm_{g_a}|_{g_a} \leq C r^{-2}, \quad |\nabla r|_{g_a} \leq C
\]
can be seen by a scaling argument. First, the bounds hold on $\{ r \geq
\epsilon \}$ since the geometry of $g_a$ is uniform there. Second, on $\{ r <
\epsilon \}$ these bounds hold when $a=1$, and to obtain uniform
bounds for all $a$ we work in coordinates $(u,v,z)$ used previously on
$\mathcal{O}_{\mathbb{P}^1}(-1)^{ \oplus 2}$ and use the scaling map
$S_{a^{-1}}(u,v,z)=(a^{-3/2}u,a^{-3/2}v,z)$. Since
$r^3=(1+|z|^2)(|u|^2+|v|^2)$, we have $S_{a^{-1}}^* r = a^{-1} r$, and we also
have $S_{a^{-1}}^*
\omega_{co,1} = a^{-2}\omega_{co,a}$, see, e.g. (\ref{eq: coa-scaling}). Pulling back $|Rm_{g_1}|_{g_1}
\leq C r^{-2}$ gives the uniform estimate in $a$, and similarly
for $|\nabla r|$.
\smallskip
\par Therefore, by proving Proposition \ref{prop: gradest}, we can conclude the
gradient estimate (\ref{C1-est}). Indeed
\[
  |\nabla_{g_a} H_a|_{g_a} =
|(\nabla_{g_a}-\nabla_{H_a}) H_a|_{g_a} = |(\nabla_{H_a} h_a
h_a^{-1}) H_a|_{g_a} \leq Cr^{-1},
\]
since the $C^0$ estimate (\ref{C0-est2}) is $C^{-1} g_a \leq H_a \leq C g_a$.

\subsubsection{Laplacian of $S$} The proof of Proposition
\ref{prop: gradest} will occupy the remainder of this section. To estimate $S = |\nabla h h^{-1}|^2_{g,H}$, we start by differentiating it once
\[
\nabla_{\bar{k}} S = \langle \nabla_{\bar{k}} (\nabla h h^{-1}),\nabla h h^{-1} \rangle_{g,H} + \langle \nabla h h^{-1}, \nabla_{k} (\nabla h h^{-1}) \rangle_{g,H}.
\]
The covariant derivative $\nabla_k$ here is the Chern connection of
$H$ on indices measured with $H$, and the Chern connection of $g$ on
indices measured with $g$. Concretely, we mean
\bea \label{mixed-cov-deriv}
\nabla_{k} (\nabla_j h h^{-1})^\alpha{}_\beta &=& \partial_k (\nabla_j h h^{-1})^\alpha{}_\beta -  (\nabla_r h h^{-1})^\alpha{}_\beta(A_{g})_k{}^r{}_j \nonumber\\
&& + (A_{H})_k{}^\alpha{}_\gamma (\nabla_j h h^{-1})^\gamma{}_\beta -  (\nabla_j h h^{-1})^\alpha{}_\gamma (A_{H})_k{}^\gamma{}_\beta. 
\eea
Differentiating $S$ twice gives
\[
\begin{aligned}
g^{j \bar{k}} \nabla_j \nabla_{\bar{k}} S &= |\nabla (\nabla h h^{-1})|^2_{g,H} +|\bar{\nabla} (\nabla h h^{-1})|^2_{g,H} \\
&\quad + g^{j \bar{k}} \langle \nabla_j \nabla_{\bar{k}} (\nabla h h^{-1}),\nabla h h^{-1} \rangle_{g,H}\\
&\quad +  \langle \nabla h h^{-1}, g^{k \bar{j}} \nabla_{\bar{j}} \nabla_{k} (\nabla h h^{-1}) \rangle_{g,H} 
\end{aligned}
\]
where
\[
|\nabla (\nabla h h^{-1})|^2_{g,H} = g^{j \bar{k}} H^{p \bar{q}} H^{\mu \bar{\nu}} H_{\bar{\beta} \alpha} \nabla_j (\nabla_p h h^{-1})^\alpha{}_\mu \overline{\nabla_k (\nabla_q h h^{-1})^\beta{}_\nu}.
\]
Our curvature conventions imply the commutator relations
\[
[\nabla_j,\nabla_{\bar{k}}] V^\alpha = F_{j \bar{k}}{}^\alpha{}_\gamma V^\gamma, \quad [\nabla_j, \nabla_{\bar{k}}] V_\alpha = -  V_\gamma F_{j \bar{k}}{}^\gamma{}_\alpha,
\]
which gives
\bea
\nabla_{\bar{j}} \nabla_k (\nabla_r h h^{-1})^\alpha{}_\beta &=&  \nabla_k \nabla_{\bar{j}} (\nabla_r h h^{-1})^\alpha{}_\beta +  (\nabla_s h h^{-1})^\alpha{}_\beta (R_g)_{k \bar{j}}{}^s{}_r \nonumber\\
&& - F_{k \bar{j}}{}^\alpha{}_\gamma (\nabla_r h h^{-1})^\gamma{}_\beta +  (\nabla_r h h^{-1})^\alpha{}_\gamma F_{k \bar{j}}{}^\gamma{}_\beta.
\eea
Since $\Lambda_{\omega} F_{H} = 0$ and we write $g^{jk}
(R_g)_{\bar{k}j}{}^s{}_r = (R_g)^s{}_r$, we have
\[
g^{k \bar{j}} \nabla_{\bar{j}} \nabla_k (\nabla_r h h^{-1})^p{}_q =
g^{j \bar{k}}  \nabla_j \nabla_{\bar{k}} (\nabla_r h h^{-1})^p{}_q +
(\nabla_s h h^{-1})^\alpha{}_\beta (R_g)^s{}_r.
\]
Therefore
\bea
&\ & \Delta_{g} S \nonumber\\
&=&  2 {\rm Re} \, \langle g^{j \bar{k}} \nabla_j \nabla_{\bar{k}}
(\nabla h h^{-1}),\nabla h h^{-1} \rangle + |\nabla (\nabla h h^{-1})|^2_{g,H}
+|\bar{\nabla} (\nabla h h^{-1})|^2_{g,H} \nonumber\\
&&+ g^{s \bar{r}}H_{\bar{\beta}\alpha}H^{\mu\bar{\nu}}(\nabla_s h h^{-1})^{\alpha}{}_\mu \overline{ (\nabla_p hh^{-1})^\beta{}_\nu (R_g)^p{}_r } . \nonumber
\eea

We relate the highest order terms of order $\nabla^3 h$ to the curvatures $R_g$, $F_H$ and their derivatives.

\begin{lem}
The following identity holds:
\bea
g^{j \bar{k}} \nabla_j \nabla_{\bar{k}} (\nabla_i h h^{-1})^\alpha{}_\beta &=& g^{j \bar{k}} \nabla^g_j (R_{g})_{i \bar{k}}{}^\alpha{}_\beta + g^{j \bar{k}} (\nabla_j h h^{-1})^\alpha{}_\gamma (R_g)_{i \bar{k}}{}^\gamma{}_\beta \nonumber\\
&&- g^{j \bar{k}} (R_g)_{i \bar{k}}{}^\alpha{}_\gamma (\nabla_j h h^{-1})^\gamma{}_\beta \nonumber\\
&&+ g^{j \bar{k}} \partial_{\bar{k}} (\nabla_r h h^{-1})^\alpha{}_\beta (T_g)^r{}_{ij}- g^{j \bar{k}} (R_g)_{r \bar{k}}{}^\alpha{}_\beta (T_g)^r{}_{ij},
\eea
where $\nabla^g R_g$ is the covariant derivative of the curvature tensor of $g$ with respect to the Chern connection of $g$ (in particular the connection of $H$ is not involved in this term).
\end{lem}
\begin{proof}
By (\ref{diff-connections}),
\be \label{diff-connections2}
\partial_{\bar{k}} (\nabla_i h h^{-1}) = (R_{g})_{i \bar{k}} - (F_{H})_{i \bar{k}}.
\ee
Therefore
\be \label{S-id1}
g^{j \bar{k}} \nabla_j \nabla_{\bar{k}} (\nabla_i h h^{-1})^\alpha{}_\beta = g^{j \bar{k}} \nabla_j (R_{g})_{i \bar{k}}{}^\alpha{}_\beta - g^{j \bar{k}} \nabla_j (F_{H})_{i \bar{k}}{}^\alpha{}_\beta.
\ee
Recall that our notation (\ref{mixed-cov-deriv}) is such that $\nabla$ acts by the Chern connection of $H$ on $\alpha,\beta, \gamma$ indices and acts by the Chern connection of $g$ on $i,j,k$ indices. The Bianchi identity is
\be \label{2nd-bianchi}
\nabla_j (F_{H})_{i \bar{k}}{}^\alpha{}_\beta =  \nabla_i (F_{H})_{j \bar{k}}{}^\alpha{}_\beta + (F_H)_{r \bar{k}}{}^\alpha{}_\beta (T_g)^r{}_{ij}
\ee
where $T_g$ is the torsion (\ref{eq: torsion-def}) of the metric $g$.
Contracting (\ref{2nd-bianchi}) and using $\Lambda_\omega F_H = 0$, we obtain
\[
g^{j \bar{k}} \nabla_j (F_H)_{i \bar{k}}{}^\alpha{}_\beta = g^{j \bar{k}} (F_H)_{r \bar{k}}{}^\alpha{}_\beta (T_g)^r{}_{ij}.
\]
Substituting this into (\ref{S-id1}) gives
\be \label{S-id2} 
g^{j \bar{k}} \nabla_j \nabla_{\bar{k}} (\nabla_i h h^{-1})^\alpha{}_\beta = g^{j \bar{k}} \nabla_j (R_{g})_{i \bar{k}}{}^\alpha{}_\beta - g^{j \bar{k}} (F_H)_{r \bar{k}}{}^\alpha{}_\beta (T_g)^r{}_{ij}.
\ee
Our notation (\ref{mixed-cov-deriv}) means that $\nabla_j (R_g)_{i \bar{k}}{}^\alpha{}_\beta$ involves connection terms $H^{-1}\partial H$ on the $\alpha,\beta$ indices. We will now convert the Chern connection of $H$ into the Chern connection of $g$ via
\[
\begin{aligned}
\nabla_j (R_g)_{i \bar{k}}{}^\alpha{}_\beta &= \nabla^g_j (R_g)_{i \bar{k}}{}^\alpha{}_\beta + [(A_H)_j{}^\alpha{}_\gamma-(A_g)_j{}^\alpha{}_\gamma] (R_g)_{i \bar{k}}{}^\gamma{}_\beta\\
&\quad - (R_g)_{i \bar{k}}{}^\alpha{}_\gamma [(A_H)_j{}^\gamma{}_\beta-(A_g)_j{}^\gamma{}_\beta].
\end{aligned}
\]
By (\ref{diff-connections}), equation (\ref{S-id2}) becomes
\bea
g^{j \bar{k}} \nabla_j \nabla_{\bar{k}} (\nabla_i h h^{-1})^\alpha{}_\beta &=& g^{j \bar{k}} \nabla^g_j (R_{g})_{i \bar{k}}{}^\alpha{}_\beta + g^{j \bar{k}} (\nabla_j h h^{-1})^\alpha{}_\gamma (R_g)_{i \bar{k}}{}^\gamma{}_\beta \nonumber\\
&&- g^{j \bar{k}} (R_g)_{i \bar{k}}{}^\alpha{}_\gamma (\nabla_j h h^{-1})^\gamma{}_\beta - g^{j \bar{k}} (F_H)_{r \bar{k}}{}^\alpha{}_\beta (T_g)^r{}_{ij}.
\eea
Using (\ref{diff-connections2}), we obtain the statement in the lemma.
\end{proof}

\medskip
\par Altogether, the Laplacian of $S$ is
\bea \label{Laplace-S}
\Delta_g S &=& |\nabla (\nabla h h^{-1})|^2_{g,H} +|\bar{\nabla}
(\nabla h h^{-1})|^2_{g,H} \nonumber\\
&&+ 2 {\rm Re} \, \bigg[ ({\rm I}) + ({\rm II}) + ({\rm IIIa}) + ({\rm IIIb}) \bigg]
\eea
where
\bea
({\rm I}) &=&  g^{i \bar{\ell}} H_{\bar{\mu} \alpha} H^{\beta \bar{\nu}} g^{j \bar{k}} \partial_{\bar{k}} (\nabla_r h h^{-1})^\alpha{}_\beta (T_g)^r{}_{ij} \overline{(\nabla_\ell h h^{-1})^\mu{}_\nu}, \nonumber\\
({\rm II}) &=& g^{i \bar{\ell}} H_{\bar{\mu} \alpha} H^{\beta \bar{\nu}} [g^{j \bar{k}} \nabla^g_j (R_{g})_{i \bar{k}}{}^\alpha{}_\beta- g^{j \bar{k}} (R_g)_{r \bar{k}}{}^\alpha{}_\beta (T_g)^r{}_{ij}]\overline{(\nabla_\ell h h^{-1})^\mu{}_\nu}  \nonumber\\
({\rm IIIa}) &=& g^{i \bar{\ell}} H_{\bar{\mu} \alpha} H^{\beta
  \bar{\nu}} [g^{j \bar{k}} (\nabla_j h h^{-1})^\alpha{}_\gamma
(R_g)_{i \bar{k}}{}^\gamma{}_\beta]\overline{(\nabla_\ell h h^{-1})^\mu{}_\nu} \nonumber\\
&& - g^{j \bar{k}} (R_g)_{i\bar{k}}{}^\alpha{}_\gamma (\nabla_j h h^{-1})^\gamma{}_\beta\overline{(\nabla_\ell h h^{-1})^\mu{}_\nu} \nonumber\\
({\rm IIIb}) &=& g^{i \bar{\ell}}H_{\bar{\mu}\alpha}H^{\beta\bar{\nu}}(\nabla_i h h^{-1})^{\alpha}{}_\beta (R_g)_{\bar{\ell}}{}^{\bar{p}} \overline{ (\nabla_p hh^{-1})^\mu{}_\nu }\nonumber
\eea
Recall $\epsilon>0$ divides the manifold into two regions $\{ r \leq
\epsilon \}$ and $\{ r > \epsilon \}$.
\medskip
\par $\bullet$ On $\{r > \epsilon \}$, the geometry $(X,g)$ is uniformly
bounded by a constant $\Lambda$. The $C^0$ estimate $C_0^{-1} g \leq H \leq C_0 g$ allows us to
use norms with respect to $g$ or $H$ interchangeably up to the cost
of constant. Hence on $\{ r > \epsilon \}$ we have
\bea
& \ & \bigg| ({\rm I}) + ({\rm
  II}) + ({\rm IIIa}) + ({\rm IIIb}) \bigg| \nonumber\\
&\leq&  C_{\Lambda,C_0} \bigg[ |\bar{\nabla}
(\nabla h h^{-1})|_{g,H} |\nabla h h^{-1}|_{g,H} +  |\nabla h h^{-1}|_{g,H} + |\nabla h h^{-1}|^2_{g,H} \bigg] \nonumber
\eea
and
\bea
& \ & \bigg| ({\rm I}) + ({\rm
  II}) + ({\rm IIIa}) + ({\rm IIIb}) \bigg| \nonumber\\
&\leq& {1 \over 4} (|\nabla (\nabla h h^{-1})|^2_{g,H} +|\bar{\nabla}
(\nabla h h^{-1})|^2_{g,H}) + C (S+1) 
\eea
where $C$ depends on $C_0$ and $\Lambda$.
\medskip
\par $\bullet$ On $\{ r < \epsilon \}$, the metric $g$ is K\"ahler
Ricci-flat. Term $( {\rm I})$ vanishes since the torsion $T_g=0$. Term
$({\rm II})$ vanishes by the Bianchi identity
\[
g^{j \bar{k}} \nabla_j R_{i \bar{k}}{}^\alpha{}_\beta = g^{j \bar{k}}
\nabla_i R_{j \bar{k}}{}^\alpha{}_\beta = 0
 \]
combined with the Ricci-flat condition. Term $({\rm IIIb})$ also
vanishes and we are left with
\[
\Delta_g S = |\nabla (\nabla h h^{-1})|^2_{g,H} +|\bar{\nabla}
(\nabla h h^{-1})|^2_{g,H} + 2 {\rm Re} \, ({\rm IIIa}) .
 \]
 By uniform equivalence of the metrics $g$ and $H$, we may estimate
 this as
 \[
|({\rm IIIa})| \leq C(C_0) |Rm_g|_g S.
 \]
 By the estimate $|Rm_g| \leq C_1 r^{-2}$, we see that on the
 entirety of $X$ we can estimate
 \bea \label{delta-S-est}
& \ & \bigg| ({\rm I}) + ({\rm
  II}) + ({\rm IIIa}) + ({\rm IIIb}) \bigg| \nonumber\\
&\leq& {1 \over 4} (|\nabla (\nabla h h^{-1})|^2_{g,H} +|\bar{\nabla}
(\nabla h h^{-1})|^2_{g,H}) + C r^{-2} (S+1) 
\eea
 where $C$ depends on $C_0$, $C_1$, $\Lambda$, $\epsilon$.
 
\subsubsection{Test function}
To construct a test function to control $S$, we will use ${\rm Tr} \,
h$. Contracting the formula (\ref{diff-chern-curv}) for the difference
of curvature tensors and using $\Lambda_\omega F_H=0$, we see that
\[
\sqrt{-1} \Lambda_\omega Rm_g = -g^{j \bar{k}} \partial_{\bar{k}} ( h
\nabla^H_j h^{-1}) = g^{j \bar{k}} \partial_{\bar{k}} (\nabla^H_j h h^{-1}).
\]
Therefore
\[
\Tr (\sqrt{-1} \Lambda_\omega Rm_g) h = \Delta_g \Tr h - g^{j \bar{k}}
\Tr \nabla_j h h^{-1} \nabla_{\bar{k}} h .
\]
We note that since $\Lambda_{\omega} Rm_g=0$ in $\{ r \leq \epsilon \}$, we have the
bound
\[
|\Lambda_\omega Rm_g| \leq C(\Lambda).
\]
Let $\delta>0$. Let $\zeta(s): [0,\infty) \rightarrow [0,1]$ be a cutoff function satisfying $\zeta(s) \equiv 1$ when $s \geq 2$ and $\zeta(s) \equiv 0$ when $s \leq 1$, and $|\zeta'|^2 \leq 9 \zeta$. We will use the test function
\[
P(z) = \zeta_\delta(z) S(z) + {A \over \delta^2} \, {\rm Tr} \, h(z), \quad \zeta_\delta(z) = \zeta \left( {r(z) \over \delta} \right),
\]
for $A(\Lambda,C_0,C_1) \gg 1$ to be chosen later. Then
\[
\Delta_g P = \zeta_\delta \Delta_g S + S \Delta_g
\zeta_\delta + 2 {\rm Re} \, g^{j \bar{k}} \partial_j \zeta_\delta
\partial_{\bar{k}} S + A \delta^{-2} \Delta_g {\rm Tr} \, h.
\]
By (\ref{Laplace-S}),
\bea
\Delta_g P &=&  \zeta_\delta (|\nabla (\nabla h h^{-1})|^2_{g,H}
+|\bar{\nabla} (\nabla h h^{-1})|^2_{g,H}) + A \delta^{-2} g^{j \bar{k}} {\rm Tr} \ \nabla_j h h^{-1} \nabla_{\bar{k}} h \nonumber\\
&&+  2 \zeta_\delta {\rm Re} \, \bigg[ ({\rm I}) + ({\rm II}) + ({\rm III}) \bigg] + ({\rm IV}) + ({\rm V})+ ({\rm VI})
\eea
where
\bea
({\rm IV}) &=& S \Delta_g \zeta_\delta \nonumber\\
({\rm V}) &=& 2 {\rm Re} \, g^{j \bar{k}} \partial_j \zeta_\delta [ \langle \nabla_{\bar{k}} (\nabla h h^{-1}), \nabla h h^{-1} \rangle + \langle \nabla h h^{-1}, \nabla_k (\nabla h h^{-1}) \rangle] \nonumber\\
({\rm VI}) &=& A \delta^{-2} \Tr (\sqrt{-1} \Lambda Rm_g) h \nonumber
\eea
We want to show that $P$ is bounded by $C(C_0,C_1,\Lambda) \delta^{-2}$. If $P$ attains a maximum
on $\{ r \leq \delta \}$, then $\zeta_\delta =0$ and $P$ is bounded
by $A \delta^{-2} \sup_X {\rm Tr} \, h$. Suppose $P$ attains a maximum at a
point $x \in \{ r > \delta \}$. Our good term will be
\be
A \delta^{-2} g^{j \bar{k}} {\rm Tr} \ \nabla_j h h^{-1}
\nabla_{\bar{k}} h \geq {A \over \delta^2 C(C_0)} | \nabla h h^{-1}|^2_{g,H} = {A \over C \delta^2} S
\ee
using $C_0^{-1} g \leq H \leq C_0 g$ and $h^\dagger = h$. For the term $({\rm V})$,
\[
|({\rm V})| \geq - C S^{1/2} |\nabla \zeta_\delta|_g (|\nabla (\nabla h h^{-1})|_{g,H} + |\overline{\nabla} (\nabla h h^{-1})|_{g,H}). 
\]
We have
\[
|\nabla \zeta_\delta|_g = \delta^{-1} |\zeta'| |\nabla r|_g \leq C(\Lambda) \delta^{-1} | \zeta'|. 
\]
Since $|\zeta'| \leq 3 |\zeta|^{1/2}$, we have
\[
|({\rm V})| \geq - C \delta^{-1} S^{1/2}  |\zeta_\delta|^{1/2} (|\nabla (\nabla h h^{-1})| + |\overline{\nabla} (\nabla h h^{-1})|)
\]
Using $2ab \leq a^2 + b^2$, we can estimate
\[
|({\rm V})| \geq - {1 \over 4} |\zeta_\delta| (|\nabla (\nabla h
h^{-1})|^2 + |\overline{\nabla} (\nabla h h^{-1})|^2) - C \delta^{-2} S.
\]
In (\ref{delta-S-est}), we showed the estimate
\[
2 \zeta_\delta {\rm Re} \, \bigg[ ({\rm I}) + ({\rm II}) + ({\rm III})\bigg] \geq - {1 \over 4} |\zeta_\delta| (|\nabla (\nabla h
h^{-1})|^2 + |\overline{\nabla} (\nabla h h^{-1})|^2) - C |\zeta_\delta|r^{-2}S.
\]
We are working in the region where $r^{-2} < \delta^{-2}$, hence we
obtain at $x$ the inequality
\bea
0 &\geq& \Delta P(x) \nonumber\\
&\geq& {\zeta_\delta \over 2} (|\nabla (\nabla h h^{-1})|^2_{g,H}
+|\bar{\nabla} (\nabla h h^{-1})|^2_{g,H}) + {A \over C} \delta^{-2}S - C\delta^{-2} S -CA \delta^{-2}. \nonumber
\eea
For $A \gg 1$ depending on $C_0,C_1,\Lambda$, we obtain
\[
S(x) \leq C.
\]
It follows that $P(x)$ is bounded by $C \delta^{-2}$. From the bound of $P$, we obtain a bound for $S$ on $\{ r \geq 2 \delta \}$ since $\zeta \equiv 1$ there. 
\[
\sup_{ \{ r \geq 2 \delta \}} S \leq C \delta^{-2}.
\]
It follows that for any point $x \in X$, we have the estimate
\[
S(x) \leq C r(x)^{-2}.
\]
This completes the proof of Proposition
\ref{prop: gradest}.

\subsection{Higher estimates}

By combining the $C^0$ and $C^1$ estimates along the degenerating
sequence $(g_a,H_a)$, we can apply regularity theory of elliptic
equations to obtain higher order estimates and obtain a limit as $a
\rightarrow 0$.
\bigskip
\par
 {\it Proof of Theorem \ref{thm:H_0}:} To extract a limit from $(g_a,H_a)$, we fix $\delta>0$
and work on the set $U_\delta = \{ r > \delta \}$. Since $g_a
\rightarrow g_0$ smoothly uniformly on $U_\delta$, we can cover
$U_\delta$ by finitely many charts where the matrices representing the metrics $g_a$ satisfy
\[
\Lambda^{-1}_\delta \delta_{kj} \leq (g_a)_{\bar{k} j} \leq \Lambda_\delta \delta_{kj}, \quad \| (g_a)_{\bar{k} j} \|_{C^k} \leq \Lambda_{\delta,k}.
\]
The $C^0$ estimate (\ref{C0-est2}) for $H_a$ implies that on this cover of $U_\delta$, the local matrices satisfy
\[
C^{-1} \delta_{kj} \leq (H_a)_{\bar{k} j} \leq C \delta_{kj}.
\]
The $C^1$ estimate (\ref{C1-est})  gives a uniform bound 
\[
\| \partial (H_a)_{\bar{\mu} \nu} \|_{L^\infty(U_\delta)} \leq C
\]
uniform in $a$ on the local matrices $(H_a)_{\bar{\mu} \nu}$. The equation $\Lambda_{\omega_a} F_{H_a} = 0$ is given in local charts as
\[
(g_a)^{j \bar{k}} \partial_j \partial_{\bar{k}} (H_a)_{\bar{\mu} \nu} = (g_a)^{j \bar{k}} \partial_{\bar{k}} (H_a)_{\bar{\mu} \gamma} (H_a)^{\gamma \bar{\alpha}} \partial_j (H_a)_{\bar{\alpha} \nu}.
\]
The right-hand side is bounded in $L^\infty$. By the local $C^{1,\alpha}$ estimate for elliptic PDE, the local matrices $(H_a)_{\bar{\mu} \nu}$ satisfy
\[
\| (H_a)_{\bar{\mu} \nu} \|_{C^{1,\alpha}(U_\delta)} \leq C.
\]
By the local Schauder estimates, we can take a smooth limit of the sequence $(g_a,H_a)$ on $U_\delta$ as $a \rightarrow 0$. The limiting metric $H_0$ satisfies
\[
\Lambda_{\omega_0} F_{H_0} = 0
\]
on $U_\delta$. We can now let $\delta \rightarrow 0$ to obtain a
limiting metric $H_0$ on $\underline{X}_{reg}$. The $C^0$ and $C^1$
estimates imply
\be \label{eq:c0+c1}
|H_0|_{g_0} + |H_0^{-1}|_{g_0} + r |\nabla_{g_0} H_0|_{g_0} \leq C.
\ee
To obtain higher estimates, we work near the singularities of
$\underline{X}$, which can be identified with a neighborhood of $V_0$
with $g_0=g_{co,0}$. In holomorphic cylindrical coordinates (see Lemma
\ref{lem: cylCoords}), we have $r^2 g^{-1}_{co,0} = O(I)$ (notation
$O(I)$ is used for a matrix uniformly equivalent to the identity) and the equation
\[
r^2 (g_{co,0})^{j \bar{k}} \partial_j \partial_{\bar{k}}
(H_0)_{\bar{\mu} \nu} = r^2 (g_{co,0})^{j \bar{k}} \partial_{\bar{k}} (H_0)_{\bar{\mu} \gamma} (H_0)^{\gamma \bar{\alpha}} \partial_j (H_0)_{\bar{\alpha} \nu}.
\]
Estimate (\ref{eq:c0+c1}) in these coordinates is $H=r^2O(I)$ and
$\partial H = r^2 O(1)$. Therefore this local equation is of the form
\[
a^{i \bar{j}} \partial_i \partial_{\bar{j}} H_0 = f, \quad f= r^2 O(1), \ a^{i
  \bar{j}} = O(I).
\]
Local $C^{1,\alpha}$ estimates for elliptic PDE imply
\[
\| (H_0)_{\bar{\mu} \nu} \|_{C^{1,\alpha}(B_{1/2},g_{euc})} \leq C(\|
H_{\bar{\mu} \nu} \|_{L^\infty(B_1, g_{euc})}+\| f
\|_{L^\infty(B_1)}) \leq C r^2.
\]
Local Schauder estimates then imply $\| (H_0)_{\bar{\mu} \nu}
\|_{C^{k,\alpha}(B_{1/2},g_{euc})} \leq C_k r^2$. Converting these local estimates to norms using $g_{co,0}$ gives estimates of the form
\[
  |\nabla^k_{g_{co,0}} H_0|_{g_{co,0}}  \leq C_k r^{-k}
\]
for each $k \in \mathbb{Z}_{\geq 0}$. This completes the estimate. $\qed$

\section{Quantitative convergence to the tangent cone} \label{sec: quant-decay}

In this section we show that the estimates for the
Hermitian-Yang-Mills metric  $H_0$ constructed in Theorem \ref{thm:H_0} can be improved to obtain the decay of $H_0$ towards the Candelas-de la Ossa metric.  This will be an essential ingredient in the perturbation argument later in the paper.  The main goal of this section is to prove

\begin{thm}\label{thm: decayToTangent}
Let $(V_0,g_{co,0})$ denote the conifold equipped with Candelas-de la Ossa Ricci-flat K\"ahler cone metric, and let $0\in V_0$ denote the tip of the cone.  Suppose $H$ is a Hermitian-Yang-Mills metric on $T^{1,0}V_0$ over $B_1(0)\setminus\{0\}$.  Assume that there is a constant $C>0$ so that $H$ satisfies
\[
C^{-1}g_{co,0} < H<Cg_{co,0}.
\]
Then there are constants $c_0>0, \lambda \in (0,1)$, and for each $k\in \mathbb{Z}_{\geq 0}$ a constant $C_{k}>0$, such that the following estimate holds
\[
|\nabla_{g_{co,0}}^k\left( H-c_0g_{co,0}\right)|_{g_{co,0}} \leq C_k r^{\lambda -k},
\]
where, as usual, $r(x)= dg_{co,0}(x,0)$ is the distance to $0\in V_0$ with respect to $g_{co,0}$.
\end{thm}

The proof of Theorem~\ref{thm: decayToTangent} follows closely the work of Jacob-S\'a Earp-Walpuski \cite{JacobEarpWalpuski} who studied related quantitative convergence results in the case of punctured balls in $\mathbb{C}^n$.  Related results for stationary Yang-Mills connections were obtained by Yang using a \L ojasiewicz inequality \cite{YangYM}.  Chen-Sun \cite{ChenSun1, ChenSun2, ChenSun3, ChenSun4} obtained a general characterization tangent cones of Hermitian-Yang-Mills connections on reflexive sheaves on the ball in $\mathbb{C}^n$ without estimates for the convergence rate.  For our applications, the polynomial decay rate, as well as the convergence at the level of metrics (rather than connections) obtained in Theorem~\ref{thm: decayToTangent} is crucial.

The first step towards establishing Theorem~\ref{thm: decayToTangent} is to prove the following Poincar\'e inequality.

\begin{lem}\label{lem: Poincare}
Let $(V_0,g_{co,0})$ be the conifold equipped with the Candelas-de la Ossa metric.  There is a uniform constant $C>0$ with the following property: for any $\rho \in (0, 1]$ and any $s \in C^{\infty}(\{r=\rho\}, \sqrt{-1}\mathfrak{su}(T^{1,0}V_0, g_{co,0}))$ we have
\[
\int_{\{r=\rho\}}|s|_{g_{co,0}}^2 dS(\rho)_{g_{co,0}} \leq  C\rho^2 \int_{\{r=\rho\}} |\nabla_{g_{co,0}} s|_{g_{co,0}}^2 dS(\rho)_{g_{co,0}}
\]
where $dS(\rho)_{g_{co,0}}$ denotes the surface measure on $\{r=\rho\}$ induced by $g_{co,0}$.
\end{lem}
\begin{proof}
The result follows from standard elliptic theory and scaling provided we can show that there are no parallel sections of $\sqrt{-1}\mathfrak{su}(T^{1,0}V_0, g_{co,0})$ on the link of the cone $\{r=1\}$.

To begin, recall from Section~\ref{sec: geoConifold} that $V_0$ can be identified with the complement of the zero section in $\iota^*\mathcal{O}_{\mathbb{P}^3}(-1)$ where
\[
\iota:\mathbb{P}^1\times\mathbb{P}^1 \rightarrow \mathbb{P}^3
\]
is the Segre embedding.  In particular, there is a projection $\pi: V_0\setminus\{0\}\rightarrow \mathbb{P}^1\times \mathbb{P}^1$ whose fibers are orbits of the holomorphic Reeb vector field.  Let $E\rightarrow \mathbb{P}^1\times \mathbb{P}^1$ be the holomorphic vector bundle generated by the invariant sections of $T^{1,0}V_0$, so that $T^{1,0}V_0 = \pi^*E$.  We can describe $E$ explicitly; if $\mathcal{L} \subset T^{1,0}V_0$ denotes the trivial line bundle generated by the non-vanishing holomorphic Reeb field, then we have an exact sequence
\[
0 \rightarrow \mathcal{L} \rightarrow T^{1,0}V_0 \rightarrow \pi^*T^{1,0}(\mathbb{P}^1\times \mathbb{P}^1)\rightarrow 0.
\]
Note that the $\mathcal{L}$-valued $(1,0)$ form on $V_0$ given by $\xi \otimes \del \log r$ is precisely the orthogonal projection $T^{1,0}V_0\rightarrow \mathcal{L}$ given by the Calabi-Yau metric $g_{co,0}$.  Therefore the second fundamental form of $\mathcal{L} \subset T^{1,0}V_0$ is represented by 
\[
\xi \otimes \ddb\log r =\xi \otimes \pi^*\omega_{KE}
\]
where $\omega_{KE}$ is the K\"ahler-Einstein metric on $\mathbb{P}^1\times \mathbb{P}^1$ satisfying
\[
\Ric(\omega_{KE}) = 3\omega_{KE}.
\]
Since $\mathcal{L}$ is trivial, we can view $E$ as the bundle corresponding to $\frac{1}{3}c_1(\mathbb{P}^1\times \mathbb{P}^1)$ under the isomorphisms
\[
{\rm Ext}^1(T^{1,0}(\mathbb{P}^1\times \mathbb{P}^1), \mathcal{O}_{\mathbb{P}^1\times \mathbb{P}^1}) \cong H^{1}(T^{1,0}(\mathbb{P}^1\times \mathbb{P}^1)^{\vee}) \cong H^{1,1}(\mathbb{P}^1\times\mathbb{P}^1, \mathbb{C})
\]
Thus $E$ sits in an exact sequence
\begin{equation}\label{eq: ExtSeq}
0 \rightarrow \mathcal{O}_{\mathbb{P}^1\times \mathbb{P}^1} \rightarrow E \rightarrow T^{1,0}(\mathbb{P}^1\times\mathbb{P}^1)\rightarrow 0.
\end{equation}
Furthermore, since $g_{co,0}$ is Calabi-Yau on the cone, one can easily show that $E$ admits a natural Hermitian-Yang-Mills metric with respect to the K\"ahler class $c_1(p_1^*\mathcal{O}_{\mathbb{P}^1}(1)\otimes p_2^*\mathcal{O}_{\mathbb{P}^1}(1))$, see e.g. \cite{Tian1} for related discussion.

It is easy to show, by direct calculation, that any parallel section
of $s \in C^{\infty}(\{r=\rho\}, \sqrt{-1}\mathfrak{su}(T^{1,0} V_0, g_{co,0}))$ descends to a trace-free, global holomorphic section $s_0\in H^{0}(\mathbb{P}^1\times \mathbb{P}^1,{\rm End}(E))$.  Thus, it suffices to show that the only global holomorphic endomorphisms of $E$ are multiples of the identity map.  This will follow from the usual result for stable vector bundles provided we can show that $E$ is indecomposable \cite{SiuBook}. Let $H_{E}$ denote the Hermitian-Yang-Mills connection on $E$ and let $\omega= p_1^*\omega_{FS}+p_2^*\omega_{FS}$ where $\omega_{FS}$ denotes the Fubini-Study metric on $\mathbb{P}^1$.  By the exact sequence~\eqref{eq: ExtSeq}, the slope of $E$ is given by
\[
\mu(E) = \frac{c_1(E) \cup [\omega]}{{\rm rk}(E)} = \frac{2\int_{\mathbb{P}^1\times\mathbb{P}^1}\omega^2}{3} = \frac{4}{3}
\]
Now suppose that $E$ can be holomorphically decomposed as $E= E_1\oplus E_2$ where $ {\rm rk}(E_1)=1,2$.  A standard computation shows that the decomposition $E= E_1\oplus E_2$ is orthogonal with respect to the Hermitian-Yang-Mills metric and the restriction $H_{E}\big|_{E_1}$ is Hermitian-Yang-Mills with slope $\mu(E_1)=\mu(E)$ \cite{SiuBook}.  Thus we have
\[
c_1(E_1)\cup[\omega] = \frac{4}{3}{\rm rk}(E_1).
\]
However, since $c_1(E_1), [\omega] \in H^{2}(\mathbb{P}^1\times\mathbb{P}^1, \mathbb{Z})$ this implies that $\frac{4}{3}{\rm rk}(E_1) \in \mathbb{Z}$, which is impossible since ${\rm rk}(E_1)=1,2$. Therefore $E$ is indecomposable and hence stable.  The result follows.
\end{proof}

In the remainder of this section we will show that Lemma~\ref{lem: Poincare} implies Theorem~\ref{thm: decayToTangent}.  Much of the argument is based on the following well-known formula:  if $H, \hat{H}$ are Hermitian metrics on a holomorphic vector bundle $E$, then the positive definite,  hermitian (with respect to either $H,\hat{H}$) endomorphism $h= \hat{H}^{-1}H$ satisfies
\begin{equation}\label{eq: keyCurve}
F_{j\bar{k}} - \hat{F}_{j\bar{k}}= -\del_{\bar{k}}(h^{-1} \hat{\nabla}_{j} h)
\end{equation}
where $F$ (resp. $\hat{F}$) denotes the curvature of the Chern connection $\nabla$ (resp. $\hat{\nabla}$) defined with respect to $H$ (resp. $\hat{H}$).  In our case we will take $\hat{H} =g$ to be the Calabi-Yau cone metric on $V_0$.  We begin with the following lemma, which shows that, at least at the level of the determinant, the metric $H$ decays towards $g_{co,0}$.

\begin{lem}\label{lem: determinantDecay}
Let $(V_0,g)$ be a Calabi-Yau cone of real dimension $n>2$. Suppose $H$ is a Hermitian-Yang-Mills metric on $T^{1,0}V_0 \rightarrow B_{1}(0)\setminus\{0\}$ with slope $0$.  Suppose there is a constant $C>0$ such that the $h= g^{-1}H$ satisfies $C^{-1}Id<h<CId$.   Then, there are constants $C_k, C_*, \gamma>0$, with $\gamma$ depending only on $(V_0,g_{co,0})$ so that, for each $k\in \mathbb{N}$ we have
\[
\bigg|\nabla_g^k\left(\log\left(\det h\right) -C_*\right)\bigg|_g \leq C_kr^{\gamma-k}
\]
\end{lem}
\begin{proof}
Since $g$ is Ricci flat, and $H$ is Hermitian-Yang-Mills with slope $0$, it follows from~\eqref{eq: keyCurve} that
\[
g^{j\bar{k}}\del_{j}\de_{\bar{k}} \log \det h  = - g^{j\bar{k}}F_{j\bar{k}} =0.
\]
On the other hand, since $\log \det h$ is bounded the result follows from separation of variables.  To see this recall that if $\phi_{\lambda}$ is a function on the link $L:= \del B_1(0) \subset V_0$ satisfying
\[
\Delta_{g_{L}} \phi_{\lambda} + \lambda \phi_{\lambda} =0
\]
with $\lambda \geq 0$, then we can produce harmonic functions $u_{\lambda}^{\pm}$ on $V_0$ given by
\[
u_{\lambda}^{\pm}= r^{a(\lambda)_{\pm}}\phi_{\lambda}
\]
where $2a(\lambda)_{\pm} = -(n-2) \pm \sqrt{(n-2)^2+4\lambda}$ (since $n>2$). Fixing an orthonormal basis of eigenfunctions $\{\phi_{\lambda}\}_{\lambda \in {\rm Spec}(\Delta_{g_{L}})}$, standard elliptic theory says we can write
\[
\log\det h\big|_{L} = \sum_{\lambda\in{\rm Spec}(\Delta_{g_{L}})} c_{\lambda} \phi_{\lambda}
\]
for constant $c_{\lambda}$.  Now, we claim that since $\log\det h$ is bounded we have
\[
\log \det h = \sum_{\lambda \in {\rm Spec}(\Delta_{g_{L}})}c_{\lambda}u_{\lambda}^{+} := u.
\]
This follows from the Caccioppoli inequality;  let $\eta$ be a
standard, smooth cut-off function on $\mathbb{R}$ with $\eta(x) =0$ for $x \leq 1$, and $\eta(x)=1$ for $x \geq 2$.  Consider $\eta_{\epsilon}(y) = \eta(\epsilon^{-1}d_{g}(y,0))$.  Since $\hat{u} := \log \det h -u$ is harmonic, bounded and vanishes on the link $L$ we have
\[
0 =  -\int_{B_{1}} \eta_{\epsilon}^2 \hat{u}\Delta_{g}\hat{u}\dvol_g = \int_{B_{1}} \langle \nabla(\eta_{\epsilon}^2 \hat{u}), \nabla \hat{u} \rangle \dvol_{g} \\
\]
Applying Cauchy-Schwarz we obtain
\[
\int_{B_{1}} \eta_{\epsilon}^2 |\nabla \hat{u}|^2 \leq C \int_{B_{2\epsilon}\setminus B_{\epsilon}} |\nabla \eta_{\epsilon}|^2 \hat{u}^2.
\]
Now we have $|\nabla \eta_{\epsilon}|^2 \leq C\epsilon^{-2}$, while $|B_{2\epsilon}\setminus B_{\epsilon}| \leq C\epsilon^{n}$ and so, since $n>2$ we can take the limit as $\epsilon \rightarrow 0$ to obtain $|\nabla \hat{u}|\equiv 0$ and the claim follows.

Now the result for $k=0$ follows from the fact that the only harmonic function on the link $\{r=1\}$ is a constant.  Combining this with standard estimates for harmonic functions and scaling, we obtain the result for $k\geq 1$.

\end{proof}

Next we prove that the relative endomorphism $h= g_{co,0}^{-1}H$ is $W^{1,2}$ on the cone.

\begin{lem}\label{lem: LocalSubsolution}
With $h= g_{co,0}^{-1}H$ as above, we have $|\nabla_{g_{co,0}} h|_{g_{co,0}} \in L^{2}(B_1(0), g_{co,0})$.
\end{lem}
\begin{proof}
To ease notation let us denote $g=g_{co,0}$ and $\nabla=\nabla_{g_{co,0}}$.  To prove the $L^2$ bound, observe that~\eqref{eq: keyCurve} implies
\[
g^{j\bar{k}}\del_{j}\del_{\bar{k}} \Tr h = g^{j\bar{k}} \Tr(\del_{\bar{k}}h h^{-1}\nabla_{j} h).
\]
On the other hand, since $h$ is hermitian and bounded we have
\[
C^{-1}|\nabla h|^2_{g} \leq g^{j\bar{k}} \Tr(\del_{\bar{k}}h h^{-1}\nabla_{j} h) \leq C |\nabla h|^2_{g}.
\]
This implies an $L^2$ estimate for $|\nabla h|_{g}$.  To see this, let $\eta(x)$ be a standard cut-off function such that $\eta(x) \equiv 1$ for $x\in[0,1]$, $\eta(x) \equiv 0$ for $x\in [2,\infty)$ and $|\eta'| + |\eta''| <10$.  For $y>0$ we set $
\eta_{y} = \eta\left(y^{-1}r\right)$. Since $\Tr h$ is uniformly bounded we have
\[
\begin{aligned}
\int_{B_1} g^{j\bar{k}} \Tr(\del_{\bar{k}}h h^{-1}\nabla_{j} h)\dvol_{g} & \leq \lim_{\epsilon \rightarrow 0} \int_{B_2} (1-\eta_{\epsilon})^2(\eta_{1})^2g^{j\bar{k}} \Tr(\del_{\bar{k}}h h^{-1}\nabla_{j} h) \dvol_{g} \\
&=\lim_{\epsilon \rightarrow 0} \int_{B_2} (1-\eta_{\epsilon})^2\eta_1^2 (\Delta_g \Tr h)\dvol_{g}\\
&=\lim_{\epsilon \rightarrow 0}  \int_{B_2} \left(\Delta_g(1-\eta_{\epsilon})^2\eta_1^2\right)  \Tr h \dvol_{g}\\
& \leq  C\lim_{\epsilon \rightarrow 0} \epsilon^{-2}{\rm Vol}_{g}(B_{2\epsilon}\setminus B_{\epsilon}) + C
\end{aligned}
\]
and the result follows since $\epsilon^{-2}{\rm Vol}_{g}(B_{2\epsilon}\setminus B_{\epsilon})  \leq C\epsilon^4$. 
\end{proof}

The next step is to establish some decay for the endomorphism $h= g_{co,0}^{-1}H$.

\begin{lem}\label{lem: decayInt}
Define a hermitian endomorphism $s$ by
\[
e^{s} =\left(\frac{\det{g_{co,0}}}{\det{H}}\right)^{\frac{1}{3}} \cdot g_{co,0}^{-1}H.
\]
Then are constants $C>0, \alpha \in (0,1)$ depending on $(V_0,g_{co,0})$ and $|s|_{L^{\infty}(B_1(0))}$ such that
\[
\mu(\tau) := \int_{B_{\tau}} r^{2-2n}|\nabla_{g_{co,0}}
s|^2_{g_{co,0}} \, \dvol_{g_{co,0}} \leq C\tau^{2\alpha}
\] 
for all $\tau \leq 1$.
\end{lem}
\begin{proof}
Again, we denote $g=g_{co,0}$ and $\nabla=\nabla_{g_{co,0}}$ to ease notation.  It is not hard to check that
\[
\langle e^{-s}\nabla_je^{s}, s \rangle_g =  \langle \nabla_j s, s\rangle_g.
\]
Therefore, from~\eqref{eq: keyCurve} we have
\[
\begin{aligned}
g^{j\bar{k}}\del_{j}\del_{\bar{k}}|s|^2_{g} &= g^{j\bar{k}}\del_{\bar{k}}(\langle \nabla_js, s \rangle_{g} + g^{j\bar{k}}\langle s, \del_{\bar{j}} s \rangle_{g})\\
&= g^{j\bar{k}}\del_{\bar{k}}(\langle e^{-s}\nabla_je^s, s \rangle_{g} + g^{j\bar{k}}\langle s, e^{-s}\del_{\bar{j}} e^s \rangle_{g})\\
&=g^{j\bar{k}}\langle e^{-s}\nabla_je^s, \nabla_{k}s \rangle_{g} + g^{j\bar{k}}\langle \del_{\bar{k}} s, e^{-s}\del_{\bar{j}} e^s \rangle_{g}.
\end{aligned}
\]
Note that the bound $C^{-1}g < H < Cg$ implies that $|s|<C$.  Thanks to \cite[Lemma 2.1]{UY} there is a uniform constant $A>0$ depending only on $C$ so that
\[
g^{j\bar{k}}\langle e^{-s}\nabla_je^s, \nabla_{k}s \rangle_{g} + g^{j\bar{k}}\langle \del_{\bar{k}} s, e^{-s}\del_{\bar{j}} e^s \rangle_{g} \geq  A^{-1}|\nabla s|^2.
\]
In summation, we have show that
\begin{equation}\label{eq: subsolution}
\Delta|s|^2_{g} \geq A^{-1}|\nabla s|^2_{g}.
\end{equation}

We now show that $\mu(\tau)<C$ for some constant $C$ independent of $\tau$.  Let $\eta_y(r)$ be the cut-off function from Lemma~\ref{lem: determinantDecay}. First note that, for any $\epsilon>0$, integration by parts
 using~\eqref{eq: subsolution}, $\Delta_g r^{2-2n}=0$ and $|\Delta_g
 \eta_\tau| \leq C \tau^{-2}$, together with the bound for $|s|$ yields
\begin{equation}\label{eq: musEstimate}
\begin{aligned}
\int_{B_{\tau}\setminus B_{2\epsilon}}r^{2-2n}|\nabla s|^2_{g} \dvol_g &\leq A\int_{B_{2\tau}\setminus B_{\epsilon}} (1-\eta_{\epsilon}) \eta_{\tau} r^{2-2n}\Delta |s|^2_{g} \dvol_g\\
&\leq A'\left(\epsilon^{-2n}\int_{B_{2\epsilon}\setminus B_{\epsilon}}|s|^2 dV_g + \int_{B_{2\tau}\setminus B_{\tau}} r^{-2n} |s|^2 \dvol_g\right)
\end{aligned}
\end{equation}
which is bounded thanks to the $L^{\infty}$ bound for $s$.  To improve the estimate we decompose the second integral appearing on the right of~\eqref{eq: musEstimate} as
\[
\int_{B_{2\tau}\setminus B_{\tau}}|s|^2 \dvol_g \leq\left( \int_{\tau}^{2\tau} dr \cdot \int_{\del B_{r}(0)} |s|^2 dS_g(r) \right)
\]
where $dS_g(r)$ denotes the surface measure on $\del B_{r}$.  Since $s$ is trace free we can apply the Poincar\'e inequality in Lemma~\ref{lem: Poincare} to get
\[
\int_{\del B_{r}} |s|^2dS_g(r) \leq Cr^2 \int_{\del B_{r}(0)} |\nabla^{T} s|^2 dS_g(r) \leq Cr^2 \int_{\del B_{r}} |\nabla s|^2 dS_{g}(r)
\]
where we wrote $\nabla^T$ for the covariant derivative tangent to $\del B_{r}$.  Thus, we have
\[
\int_{B_{2\tau}\setminus B_{\tau}}r^{-2n} |s|^2 \dvol_{g} \leq C\int_{B_{2\tau}\setminus B_{\tau}} r^{2-2n} |\nabla s|^2 \dvol_{g}= C(\mu(2\tau)-\mu(\tau))
\]
Arguing similarly for the first term yields
\[
 \int_{B_{2\epsilon}\setminus B_{\epsilon}} \epsilon^{-2n} |s|^2 \dvol_g \leq C\mu(2\epsilon)
\]
All together this implies
\[
\mu(\tau) \leq \frac{C}{C+1}\left(\mu(2\tau) + \mu(2\epsilon)\right).
\]
Since this estimate holds for all $\epsilon>0$ and, thanks to Lemma~\ref{lem: LocalSubsolution}, $\mu(\epsilon)\rightarrow 0$ by the dominated convergence theorem, we conclude 
\[
\mu(\tau) \leq \frac{C}{C+1}\mu(2\tau).
\]
The lemma follows by a standard iteration argument.
\end{proof}

Theorem~\ref{thm: decayToTangent} will  now follow from
Lemma~\ref{lem: decayInt} together with the regularity theory for the
Hermitian-Yang-Mills equation, which we recall below.  The regularity
theory is originally due to Bando-Siu \cite[Proposition 1]{BandoSiu},
but we refer the reader to the paper of Jacob-Walpuski \cite[Theorem
C.1]{JacobWalpuski} for the precise statement which implies the one
below.

\begin{prop}\label{prop: intRegHYM}
Let $(Y,g, J)$ be a K\"ahler manifold of dimension $n$ with bounded geometry, and let $E\rightarrow Y$ be a holomorphic vector bundle.  If $H_0, H$ are hermitian metrics on $E$, $H$ is Hermitian-Yang-Mills and $s:= \log(H_0^{-1}H) \in C^{\infty}(Y, \sqrt{-1}\mathfrak{su}(E, H_0))$, then, for all $k\in\mathbb{N}$ and $p\in(1,\infty)$ there is a function $f_{k,p}(y) >0$ depending only on $k,p$ and the geometry of $(Y,g)$ such that $f_{k,p}(0)=0$ and
\[
r^{k+2-\frac{2n}{p}} \|\nabla_{H_0}^{k+2} s\|_{L^{p}(B_r(x))} \leq f_{k,p}\left(\|s\|_{L^{\infty}(B_{2r}(x))} + \sum_{i=0}^{k} r^{2+i}\|\nabla^i_{H_0}F_{H_0}\|_{ L^{\infty}(B_{2r}(x))}\right).
\]
\end{prop}

We can now prove Theorem~\ref{thm: decayToTangent}.

\begin{proof}[Proof of Theorem~\ref{thm: decayToTangent}]
Let $H$ be as in the statement of the theorem, and set $g=g_{co,0}$.  Throughout the proof $C$ will denote a constant which can change from line to line, but depends only on $(V_0, g)$ and the positive upper and lower bounds for $g^{-1}H$. Define $s$ by
\[
e^{s} =\left(\frac{\det{g_{co,0}}}{\det{H}}\right)^{\frac{1}{3}} \cdot g_{co,0}^{-1}H.
\]
so that $s \in C^{\infty}(V_0, \sqrt{-1}\mathfrak{su}(T^{1,0}V_0, g))$.

Fix $0<R\ll 1$ and consider the annulus $B_{4R}\setminus B_{R/4}$.  Let $m_{R}: V_0\rightarrow V_0$ be the map $m_R(p) = R^{-1}\cdot p$ where $\cdot$ denotes the natural scaling action on the cone.  
Let $\hat{s} = m_{R}^*s$.  From the scale invariance of $\mu(\tau)$ we have
\[
\int_{B_{4}\setminus B_{\frac{1}{4}}} |\nabla\hat{s}|_g^2 \dvol_g \leq CR^{2\alpha}
\]
while the Poincar\'{e} inequality proved in Lemma~\ref{lem: Poincare} implies
\begin{equation}\label{eq: firstDecay}
\int_{B_4\setminus B_{\frac{1}{4}}} |\hat{s}|^2 \leq C\int_{B_{4}\setminus B_{\frac{1}{4}}} |\nabla \hat{s}|^2 \leq CR^{2\alpha}.
\end{equation}
Note that the $L^{\infty}$ bound for $\hat{s}$ together with the interior estimates, Proposition~\ref{prop: intRegHYM}, yield
\[
\|\nabla^{k+2} \hat{s}\|_{L^{p}(B_4\setminus B_{1/4})} \leq C_{k,p}
\]
and hence, by the Sobolev imbedding theorem we get 
\begin{equation}\label{eq: ckbndhats}
\|\hat{s}\|_{C^{k}(B_{3}\setminus B_{1/2})} \leq C_k.
\end{equation}

\smallskip
\par To improve this bound to a decay estimate we appeal to the Hermitian-Yang-Mills equation.  From~\eqref{eq: keyCurve}, we have 
\begin{equation}\label{eq: ellipticRegForm}
g^{j\bar{k}}\nabla_{\bar{k}}\nabla_{j} \hat{s} + B(\nabla \hat{s}\otimes \nabla \hat{s}) = 0
\end{equation}
where $B(\cdot)$ is linear with coefficients depending on $\hat{s}$, but not on any of its derivatives.  The result now follows from standard elliptic regularity and bootstrapping.  By elliptic regularity we have
\[
\|\hat{s}\|_{W^{2,2}(B_{2.5}\setminus B_{3/4}) } \leq C\left( \|\Delta \hat{s}\|_{L^{2}(B_{3}\setminus B_{1/2})} + \|\hat{s}\|_{L^2(B_{3}\setminus B_{1/2})}\right).
\]
On the other hand, from~\eqref{eq: ellipticRegForm} we have
\[
\|\Delta \hat{s}\|_{L^{2}(B_{3}\setminus B_{1/2})} \leq C\left(\int_{B_3\setminus B_{1/2}} |\nabla \hat{s}|^4\right)^{\frac{1}{2}} \leq C\left(\int_{B_{3}\setminus B_{1/2}} |\nabla \hat{s}|^2\right)^{\frac{1}{2}}
\]
where, in the second inequality, we used~\eqref{eq: ckbndhats} with $k=2$.  From~\eqref{eq: firstDecay} we conclude that
\[
\| \hat{s}\|_{W^{2,2}(B_{2.5}\setminus B_{3/4})} \leq CR^{\alpha},
\]
for a uniform constant $C>0$.  By differentiating~\eqref{eq: ellipticRegForm} a straightforward boot-strapping argument yields
\[
\| \hat{s}\|_{W^{k,2}(B_{2}\setminus B_{1})} \leq C_kR^{\alpha}
\]
for uniform constants $C_k>0$.  All together we obtain
\[
|\hat{s}|_{C^{k}(B_2\setminus B_1)} \leq C_k R^{\alpha}
\]
from the Sobolev imbedding theorem.  Rescaling yields
\be \label{eq: local-s-Decay}
|s|_{C^{k}(B_{2R}\setminus B_R)} \leq C_k R^{\alpha-k}.
\ee
Finally, Theorem \ref{thm: decayToTangent} follows from this estimate together with Lemma~\ref{lem: determinantDecay}.
\end{proof}

\begin{rk}
In the setting where we apply Theorem~\ref{thm: decayToTangent}, where the metric $H_0$ obtained as a limit of
$(H_a,g_a)$, we can apply Theorem \ref{thm:H_0} to obtain (\ref{eq: ckbndhats}) bypassing the Bando-Siu regularity theorem.
\end{rk}

\section{Approximate Hermitian-Yang-Mills metrics} \label{sec: approx-hym}

In the previous section, we started from a Calabi-Yau metric $\omega_{\rm CY}$ on a simply connected K\"ahler Calabi-Yau threefold $(X,\Omega)$ and constructed a pair of hermitian metrics $(g_0,H_0)$ on a singular space $\underline{X}$ obtained by contracting $(-1,-1)$ curves $C_i$. These metrics satisfy $d \omega_0^2=0$ and $\Lambda_{\omega_0} F_{H_0} = 0$. Let $\{ p_i \}$ denote the nodal points of $X_0$. There are constants $c_i,R_i,\lambda>0$ such that we have the following local description near the nodes:
\smallskip
\par \noindent $\bullet$ The Fu-Li-Yau construction gives
\[
g_0 = R_i g_{co,0} , \quad {\rm near} \ {\rm node} \ p_i.
\]
$\bullet$ By (\ref{eq: local-s-Decay}) and Lemma \ref{lem: determinantDecay}, for $k \in \mathbb{Z}_{\geq 0}$ there exists $M_k>1$ such that
\be \label{H_0-local}
| \nabla_{g_{co,0}}^k (H_0 - c_i  g_{co,0}) |_{g_{co,0}} \leq M_k
r^{\lambda-k}, \quad {\rm near} \ {\rm node} \ p_i.
\ee
 For ease of notation, in this section we will work at a single node point with scale constants $c_1=R_1=1$. The metric $H_0$ on $\underline{X}$ has conical singularities which we will desingularize by gluing in the asymptotically conical metric $g_{co,t}$ on $V_t$. For other work in geometry using this technique, see e.g. \cite{Chan,Joyce,Kari,Pacini}. The glued metric $H_t$ will approximately solve the Hermitian-Yang-Mills equation on the smoothings $X_t$ for $t$ sufficiently small.
\smallskip
\par  Recall that under the assumption of Theorem \ref{thm: Friedman} there is a smoothing $\mu: \mathcal{X} \rightarrow
\Delta$ with $\mu^{-1}(t)=X_t$, $\mu^{-1}(0)=\underline{X}$, and the family $\mathcal{X}$ is locally
described by $\{ (z,t) : z \in V_t\}$ near the nodes, with $V_t= \{
\sum_{i=1}^4 z_i^2=t \} \subset \mathbb{C}^4$. We will show there exists $\gamma,\epsilon \in (0,1)$ and $C>1$ such that
the approximate solution $H_t$ satisfies
\[
\| \Lambda_{\omega_t} F_{H_t} \|_{C^{0,a}_{\beta-2}} \leq C |t|^\gamma
\]
for all $0<|t| \leq \epsilon$, and for suitably defined weighted H\"older spaces with $\beta \in [-2,0]$, and $0<a<1$; see~Section \ref{section:holder} below for a precise definition. We
recall that we denote by
$\omega_t$ the Hermitian metric from Proposition \ref{prop: FLYalmost}
constructed by Fu-Li-Yau, which satisfies $\omega_t=\omega_{co,t}$
near the nodes and converges back to the balanced metric $\omega_0$ on
compact sets as $t \rightarrow 0$.

\subsection{Definition of the approximate solution}

To construct a Hermitian metric $H_t$ on $X_t$ which approximately
solves the Hermitian-Yang-Mills equation, we will glue $g_{co,t}$ to a deformation of the singular metric $H_0$ on the annulus region
\[
 \{ |t|^\alpha \leq \|z \|^2 \leq 2 |t|^\alpha \} \subset V_t.
\]
Here $0<\alpha<1$, and specifically we will take
$\alpha=(1+\lambda/3)^{-1}$ where $\lambda>0$ is the rate in
(\ref{H_0-local}). Let
\[
\chi(z)=\zeta(|t|^{-\alpha} \|z\|^2)
\]
be a cutoff function on this annulus region, i.e. the function
$\zeta:[0,\infty) \rightarrow [0,1]$ satisfies $\zeta \equiv 1$ on
$[0,1]$ and $\zeta \equiv 0$ on $[2,\infty)$. Our glued Hermitian
metric on $X_t$ is
\be \label{H_t-defn}
H_t = \chi g_{co, t} + (1-\chi) K_t,
\ee
where
\[
K_t = [(\Phi_t^{-1})^* H_0]^{1,1}
\]
is the $J_t$-invariant part of the pullback $(\Phi_t^{-1})^*
H_0$. Explicitly, we define $(A^{1,1})_{\alpha \beta} = {1 \over 2}
(A_{\alpha \beta} + J^\mu{}_\alpha A_{\mu \nu} J^\nu{}_\beta)$ for a
symmetric 2-tensor $A$. Recall $\Phi_t$ is defined in Lemma \ref{lem: trivGlob}, and note
that $K_t$ is defined on $X_t \backslash \{ \|z\|^2 = |t| \}$ and so
$H_t$ is defined on all of $X_t$.
\smallskip
\par We will need estimates on the glued metric $H_t$ which are uniform
in $t$.

\begin{lem} \label{lem: H_t-est}
There exists $\epsilon>0$ such that for any $k \in \mathbb{Z}_{\geq 0}$ there exists $C_k>1$ such that for all $0 < |t| < \epsilon$ we have
\be \label{g_t-unif-H_t}
C_0^{-1} g_{t} \leq H_t \leq C_0 g_{t}, \quad | \nabla_{g_{t}}^k H_t|_{g_{t}} \leq C_k r^{-k}.
 \ee
\end{lem}

\begin{proof} We work region-by-region.
\smallskip
\par $\bullet$ Region $\{ \| z \|^2 \leq |t|^\alpha \}$. Here $H_t=g_t=g_{co,t}$ so the estimates are trivial.
\smallskip
\par $\bullet$ Region $\{ |t|^\alpha \leq \|z \|^2 \leq 1
\}$. Here $g_t=g_{co,t}$ and $\| z \|^2 \gg |t|$ for all $t$ small
enough. The estimate (\ref{eq:H_0-est}) reads $|H_0|_{g_{co,0}}\leq C$ and $|H_0^{-1}|_{g_{co,0}} \leq
C$, and so pulling back by by $\Phi_t^{-1}$ and using Lemma
\ref{lem:transplant} gives 
\[
C^{-1} g_{co,t} \leq (\Phi_t^{-1})^* H_0 \leq C g_{co,t}.
\]
Since $|[(\Phi_t^{-1})^* H_0]^{1,1}|_{g_{co,t}} \leq  |(\Phi_t^{-1})^*
H_0|_{g_{co,t}}$ and similarly for $H_0^{-1}$, this proves that
$C^{-1} g_{co,t} \leq H_t \leq C g_{co,t}$. Next, pulling back $|\nabla^k_{g_{co,0}} H_0|
\leq C_k r^{-k}$ by Lemma
\ref{lem:transplant} gives
\[
|\nabla^k_{g_{co,t}} K_t|_{g_{co,t}} \leq C r^{-k}.
\]
This proves (\ref{g_t-unif-H_t}) in the region $\{ 2 |t|^\alpha
\leq \| z \|^2 \leq 1 \}$ where $H_t= K_t=[(\Phi_t^{-1})^*H_0]^{1,1}$. In the
transition region $\{|t|^{\alpha} \leq \|z\|^2 \leq 2\|t\|^\alpha\}$ we have,
\[
\nabla H_t = (\nabla \chi) (g_{co,t} - K_t) + (1-\chi) \nabla K_t
\]
and
\[
|\nabla_{g_t} H_t|_{g_t} \leq C |\nabla \chi|_{g_t} + C |\nabla_{g_t} K_t|_{g_t}.
\]
We estimate
\be \label{nabla-chi}
|\nabla \chi|_{g_t} \leq C |t|^{-\alpha} |\nabla r^3|_{g_t} \leq C
|t|^{-\alpha} r^2 \leq C r^{-1}.
\ee
Here we used $|\nabla r|_{g_t} \leq C$ and $r^3=\| z \|^2 \leq 2 |t|^\alpha$ in the
transition region. Thus
\be 
| \nabla_{g_{co,t}} H_t|_{g_{co,t}} \leq C r^{-1}
 \ee
and the higher order estimates are similar. 
\smallskip
\par $\bullet$ Region $\{1 \leq r \}$. Here $H_t=[(\Phi_t^{-1})^*
H_0]^{1,1}$ and the metrics $\Phi_t^*g_t$
converge smoothly uniformly to $g_0$ as $t \rightarrow 0$ by Lemma \ref{prop:
  FLYalmost}. The estimates for $H_t$ follow from pulling back the estimates for
$H_0$ obtained in Theorem~\ref{thm:H_0}.
\end{proof}

\subsection{Weighted H\"older spaces} \label{section:holder}
In the upcoming analysis we will work in weighted H\"older spaces on $X_t$ using the weight
function $r$, the metric $g_t$ which is equal to the model metric
$g_{co,t}$ near the nodes, and the glued metric $H_t$. We will use the connection $\nabla_{H_t}$
when differentiating. For endomorphisms $h \in \Gamma({\rm
  End} \, T^{1,0} X_t) $, we use the norm 
\[
\| h \|_{C^k_\beta(g_t,H_t)} = \sum_{i=0}^k \sup_{X_t} |r^{-\beta+i} \nabla_{H_t}^i h|_{g_{t}}.
\]
For $\Phi \in \Gamma( (T X_t)^p \otimes (T^* X_t)^q)$, we define the semi-norm
\[
[\Phi]_{C^{0,a}_\beta} = \sup_{x \neq y} \bigg[
\min(r(x),r(y))^{-\beta} \, {|\Phi(x)-\Phi(y)|_{g_t} \over d(x,y)^a} \bigg]
\]
where the sup is taken over points $x,y$ with distance less than the
injectivity radius and $\Phi(x)-\Phi(y)$ is understood by $\nabla_{g_t}$-parallel
transport along the minimal $g_t$ geodesic connecting $x$ and $y$. The weighted H\"older norms are then
\[
\| h \|_{C^{k,a}_\beta(g_t,H_t)} = \| h \|_{C^k_\beta(g_t,H_t)} + [ \nabla_{H_t}^k h]_{C^{0,a}_{\beta-k-a}}.
\]
This definition is well adapted to work on annuli $U_{\hat{r}}= \{ (1/2)
\hat{r} \leq r(z) \leq 2 \hat{r} \}$ at a given scale
$\hat{r}>0$. The norm over $U_{\hat{r}}$ is equivalent to
\be \label{weighted-alt-norm}
\| h \|_{C^{0,a}_\beta(U_{\hat{r}})} = \hat{r}^{-\beta} \bigg[ \sup_{U_{\hat{r}}} \| h
\|_{C^0(\hat{r}^{-2} g_t)} + \sup_{x,y \in U_{\hat{r}}}
{|h(x)-h(y)|_{\hat{r}^{-2} g_t} \over d_{\hat{r}^{-2}
    g_t}(x,y)^a} \bigg] ,
\ee
where norms on the endomorphism $h \in \Gamma({\rm
  End} \, T^{1,0} X_t)$ are now all with respect to the rescaled metric
$\hat{r}^{-2} g_t$. We will often estimate global H\"older norms by
estimating them on local annuli $U_{\hat{r}}$.

\begin{lem} \label{holder-local2global}
Let $\beta \leq 0$ and $h \in \Gamma({\rm End} \, T^{1,0}X_t)$. Suppose there is a uniform bound on the local estimates $\| h \|_{C_\beta^{0,a}(U_{\hat{r}})}
\leq K$ for all $\hat{r}>0$, where $U_{\hat{r}} = \{ (1/2) \hat{r}
\leq r \leq 2 \hat{r} \} \subset X_t$. Then $\| h \|_{C_\beta^{0,a}(X_t)}
\leq CK$ for a constant $C>1$ which is independent of $t$ .
  \end{lem}
  \begin{proof}
The local bounds imply $\| h \|_{C^0_\beta(X_t)} \leq K$, so we need to estimate the global H\"older semi-norm. Let $x,y \in
X_t$ and suppose $r(x) \leq r(y)$. If $y$ lies in the set $U = \{
(1/2) r(x) < r < 2 r(x) \}$, the estimate is assumed. If $r(x) \geq
\epsilon>0$, the geometry is uniform in $t$ and the estimate holds for
$C(\epsilon)$. In the remaining
case, we assume $x,y \in V_t$ with $2r(x) < r(y)$ and $g_t=g_{co,t}$ and we claim that
\[
  d_{g_t}(x,y) \geq C^{-1} r(x)
\]
for $C>1$ independent of
$t$. Indeed, this inequality holds at $t=1$ for a constant $C>1$, and
the uniform bound in $t$ follows from the bound when
$t=1$ by scaling $S(z)=t^{-1/2} z$ with $S^* g_{co,1} = |t|^{-2/3}
g_{co,t}$ and $S^* r = |t|^{-1/3} r$.  Therefore
\[
r(x)^{-\beta+a} {|h(x)-h(y)| \over d_{g_t}(x,y)^a}
\]
is bounded by $C \| h \|_{C^0_\beta(X_t)}$ since $\beta \leq 0$.
\end{proof}
\medskip
\par We end this discussion with the following remark: if
\[
  r^{-\beta}|h|_{g_t}+r^{-\beta+1} |\nabla_{g_t} h|_{g_t} \leq K,
\]
then for $0<a<1$ we can estimate $\| h \|_{C^{0,a}_\beta(X_t)} \leq CK$ where $C$ is independent of $t$. This can be seen for example from expression (\ref{weighted-alt-norm}), since $\hat{r}^{-2} g_t$ is uniformly (in $t$) equivalent to the Euclidean metric in holomorphic cylindrical coordinates (Lemma \ref{lem: cylCoords}). Also, the difference of connections satisfies the bound $r|A_{g_t}-A_{H_t}|_{g_t} \leq C$ by Lemma \ref{lem: H_t-est}, so to estimate $\| h \|_{C^{0,a}_\beta(X_t)}$ we could equivalently estimate $r^{-\beta+1} |\nabla_{H_t} h|_{H_t}$ instead of $r^{-\beta+1} |\nabla_{g_t} h|_{g_t}$.

\subsection{Smallness of the approximate solution}

\par The main objective of this section is to show that the glued
metric $H_t$ has small Hermitian-Yang-Mills tensor.

\begin{prop} \label{prop-small-hym} Let $H_t$ be the glued metric as
  in (\ref{H_t-defn}). There exists $C>0$ and $\epsilon>0$ such that
  for any $0<a<1$, and any $t \in \mathbb{C}^*$ with $|t| \leq \epsilon$ we have
\be
\| \Lambda_{\omega_t} F_{H_t} \|_{C^{0,a}_{-2}} \leq C
 |t|^{\alpha \lambda \over 3},
\ee
where $\lambda>0$ is the rate in~\eqref{H_0-local} (see Theorem~\ref{thm: decayToTangent}), and $\alpha= (1+\frac{\lambda}{3})^{-1}$.
\end{prop}
The approximate solution will be estimated in four regions.
\smallskip
\par $\bullet$ Region $\{ \| z \|^2 \leq |t|^\alpha \}$. Here $H_t = g_{co,t}$ and
$\Lambda_{\omega_t} F_{H_t} = 0$.
\smallskip
\par $\bullet$ Region $\{ {1 \over 4}|t|^\alpha \leq \| z \|^2 \leq 4
|t|^\alpha \}$. This contains the transition region, and we will show that here
\[
\| \Lambda_{\omega_t} F_{H_t} \|_{C^{0,a}_{-2}} \leq C |t|^{(\alpha \lambda)/3}
\]
in Lemma \ref{H_t-transition} below. 
\smallskip
\par $\bullet$ Region $\{ 2 |t|^\alpha \leq \| z \|^2 \leq 2 \}$. In
this region, $H_t = [(\Phi_t^{-1})^* H_0]^{1,1}$ and
we need to control the Hermitian-Yang-Mills tensor of $[(\Phi^{-1}_t)^*H_0]^{1,1}$. We will estimate
\[
\| \Lambda_{\omega_t} F_{H_t} \|_{C^{0,a}_{-2}} \leq C |t|^{1-\alpha}
\]
in this region in Lemma \ref{curv-K-est} below. 
\smallskip
\par $\bullet$ Region $\{ r \geq 1 \}$. In this region the geometry is
smoothly varying, and so since $\Lambda_{\omega_0} F_{H_0}=0$ then
\[
\| \Lambda_{\omega_t} F_{H_t} \|_{C^{0,a}_{-2}} \leq C |t|.
\]
\par By Lemma \ref{holder-local2global}, it suffices to check the H\"older estimate on these local pieces to obtain the global estimate. We start by estimating the Hermitian-Yang-Mills tensor in the transition region.

\begin{lem} \label{H_t-transition}
With notation as in Proposition~\ref{prop-small-hym}, the estimate $\| \Lambda_{\omega_t} F_{H_t} \|_{C^{0,a}_{-2}(U)} \leq C
 |t|^{\alpha \lambda \over 3}$  holds in the region $U= \{ {1 \over 4}
 |t|^\alpha \leq
\| z \|^2 \leq 4 |t|^\alpha \}$.
\end{lem}
\begin{proof} If we decompose
\[
H_0 = g_{co,0} + E_0,
\]
then the glued metric is
\[
H_t = g_{co,t} + (1-\chi) \bigg[ [(\Phi^{-1}_t)^*E_0]^{1,1} +  [(\Phi^{-1}_t)^*
g_{co,0}]^{1,1} - g_{co,t} \bigg].
\]
Since $\Lambda_{\omega_t} F_{g_{co,t}} = 0$, the formula (\ref{diff-chern-curv}) for the
difference of curvature tensors gives
\be \label{F_H-t}
\sqrt{-1} \Lambda_{\omega_t} F_{H_t} = - (g_{co,t})^{j \bar{k}}
\partial_{\bar{k}} (h_t^{-1} (\nabla_{g_{co,t}})_j h_t)
\ee
where
\bea \label{h_t-defn}
h_t &=& I + (1-\chi) \bigg[ g_{co,t}^{-1}[(\Phi^{-1}_t)^*E_0]^{1,1} +  g_{co,t}^{-1} ( [(\Phi^{-1}_t)^*
g_{co,0}]^{1,1} - g_{co,t}) \bigg] \nonumber\\
&:=& I + (1-\chi) \E.
\eea
During this proof, we simply write $\nabla = \nabla_{g_{co,t}}$. We claim that
\be \label{transition-claim}
C^{-1} I \leq h_t \leq C I, \quad |\nabla^k h_t|_{g_{co,t}} \leq C
r^{-k} |t|^{\lambda \alpha/3}.
\ee
Assuming this, (\ref{F_H-t}) and $|t|^{\lambda
  \alpha/3}<1$ imply
\[
|\Lambda_{\omega_t} F_{H_t}|_{g_{co,t}} \leq |h^{-1} \nabla h|^2_{g_{co,t}} +
C|h^{-1}\nabla^2 h_t|_{g_{co,t}} \leq C r^{-2} |t|^{\lambda
  \alpha/3}.
\]
Similarly
\[
|\nabla \Lambda_{\omega_t} F_{H_t}|_{g_{co,t}} \leq C r^{-3} |t|^{\lambda
  \alpha/3}
\]
and this proves the estimate $|\Lambda_{\omega_t} F_{H_t}|_{C^a_{-2}(U)}
\leq C |t|^{\lambda \alpha/3}$.
\smallskip
\par We now prove the claim (\ref{transition-claim}). Estimate
(\ref{H_0-local}) implies $|\nabla^k E_0|_{g_{co,0}} \leq C_k
r^{\lambda-k}$, which by Lemma \ref{lem:transplant} and (\ref{eq:decayGco0Pull}) yields
\[
|\nabla^k (\Phi_t^{-1})^* E_0|_{g_{co,t}} +  | \nabla^k [(\Phi^{-1}_t)^*
g_{co,0} - g_{co,t}]|_{g_{co,t}} \leq C r^{-k} (r^\lambda + |t| r^{-3}) .
\]
Since $r^3 \sim |t|^\alpha$, this implies
\[
|\nabla^k \E|_{g_{co,t}} \leq C r^{-k} (|t|^{\lambda \alpha /3} + |t|^{1-\alpha}).
\]
We choose $\alpha$ such that $\lambda \alpha /3 = 1 - \alpha$, so that
$|\E|\leq C |t|^{\lambda \alpha /3}$ and
\[
|h_t-I|_{g_{co,t}} \leq C |t|^{\lambda \alpha /3} \ll 1
\]
which implies $C^{-1} I \leq h_t \leq C I$. Taking a
derivative gives
\[
\nabla h_t = - \E \nabla \chi + (1-\chi) \nabla \E.
\]
Since $|\E| \leq C |t|^{\lambda \alpha/3}$, $|\nabla \E| \leq C r^{-1} |t|^{\lambda \alpha/3}$ and
$|\nabla \chi| \leq C r^{-1}$ (e.g. (\ref{nabla-chi})), we obtain
\[
|\nabla h_t| \leq C r^{-1} |t|^{\lambda \alpha/3}.
\]
The higher order estimates in the claim (\ref{transition-claim}) are similar.
\end{proof}
\medskip
\par We now consider the next region past the transition zone.

\begin{lem} \label{curv-K-est}
Let $F_{K_t}$ be the curvature of $K_t = [(\Phi_t^{-1})^* H_0]^{1,1}$. Then on
$D = \{ |t|^{\alpha} \leq \| z \|^2 \leq 2 \}$, we can estimate
\be \label{curv-K-2}
\| \Lambda_{\omega_t} F_{K_t} \|_{C^{0,a}_{-2}(D)} \leq C |t|^{1-\alpha}.
\ee
\end{lem}

\begin{proof}
  Let $(\hat{z}_1,\hat{z}_2,\hat{z}_3,\hat{z}_4,t)$ be a point in
$\mathcal{X} = \{ (z,t) : \sum_{i=1}^4 z_i^2=t \}$ with $\| \hat{z} \|^2 \geq
|t|^\alpha$, and suppose without loss of generality that $\hat{z}_4
\neq 0$. Let $\lambda = \| \hat{z} \|$ and $\hat{r}=\lambda^{2/3}$. We
take local coordinates on $U_{\hat{r}} = \{{1 \over 2} \lambda \leq \|
z \|
\leq 2 \lambda \} \subset X_t$ given by $w^i = {1 \over \lambda} z_i$.
These coordinates land in $
\bigg\{ {1 \over 4} \leq |w|\leq 4 \bigg\} \subset \mathbb{C}^3$ where
$|w|=|(w^1,w^2,w^3)|$ is the Euclidean norm on $\mathbb{C}^3$. The formula for the curvature on $V_t$ in
coordinates is
\[
(F_{K_t})_{j \bar{k}} = -K_t^{-1} \partial_j \partial_{\bar{k}} K_t +
K_t^{-1} \partial_j K_t K_t^{-1} \partial_{\bar{k}} K_t.
\]
We showed in Lemma \ref{lem: cylCoords} that in $\{w ^i \}$ coordinates, we have $g_t = \hat{r}^2
O(I)$, and by Lemma \ref{lem: H_t-est} the metric $K_t =
\hat{r}^2O(I)$. Here we write $O(I)$ for a matrix
which is positive-define with positivity and derivative bounds
independent of $t,\lambda$. Therefore
\be \label{eq:F_K_t-est}
|F_{K_t}|_{g_t}=
\hat{r}^{-2}O(1),
\ee
where $O(1)$ denotes a function
with smooth bounds independent of $t,\lambda$.
\smallskip
\par $\bullet$ We claim:
\be \label{variation-pullback-conehym}
 {\partial \over \partial t} \bigg[ [ (\Phi_t^{-1})^* g_{co,0}]^{j \bar{k}}
(F_{K_t})_{j \bar{k}} \bigg] = \hat{r}^{-5} O(1).
\ee
Here $[ (\Phi_t^{-1})^* g_{co,0}]^{j \bar{k}}$ are the local matrix
entries of the inverse of $(\Phi_t^{-1})^* g_{co,0}$ (not the raised
indices with respect to $g_t$). Assuming this for now, we complete the proof of the lemma. Since at
$t=0$ we have $[
(\Phi_t^{-1})^* g_{co,0}]^{j \bar{k}} (F_{K_t})_{j \bar{k}}= 0$, this implies
\[
[(\Phi_t^{-1})^* g_{co,0}]^{j \bar{k}} (F_{K_t})_{j \bar{k}}= t \hat{r}^{-5} O(1).
\]
We can then write
\[
\sqrt{-1} \Lambda_{\omega_t} F_{K_t} = [ (\Phi_t^{-1})^* g_{co,0}]^{j \bar{k}}
(F_{K_t})_{j \bar{k}} + \bigg[ g_{co,t}^{j \bar{k}} - [
(\Phi_t^{-1})^* g_{co,0}]^{j \bar{k}} \bigg] (F_{K_t})_{j \bar{k}},
\]
and by (\ref{eq:decayGco0Pull}) and (\ref{eq:F_K_t-est}) we have
\[
\sqrt{-1} \Lambda_{\omega_t} F_{K_t} = t \hat{r}^{-5} O(1)
\]
in $\{ w^i \}$ coordinates. Therefore
\[
|\nabla (\sqrt{-1} \Lambda_{\omega_t} F_{K_t})|_{g_t} = |t| \hat{r}^{-6}
O(1).
\]
Then for any $0<a<1$,
\[
\| \Lambda_{\omega_t} F_{K_t} \|_{C^{0,a}_{-2}(U_{\hat{r}})} \leq C |t|
\hat{r}^{-3} \leq C |t|^{1-\alpha}
\]
using that $|t|^\alpha \leq \lambda^2 \leq 1$ and $\lambda^2 =
\hat{r}^3$. By Lemma \ref{holder-local2global}, this gives the H\"older
estimate on all of $D=\{ |t|^\alpha \leq \| z \|^2 \leq 1 \}$.
\smallskip
\par $\bullet$ We now prove the claimed (\ref{variation-pullback-conehym}). We start with the variation of $K_t$. The metric $H_0$ is defined on
$V_0 = \{ \sum_i x_i^2 = 0 \}$ and here we use coordinates
$(x^1,x^2,x^3)$ given by $x^i={1
  \over \| \hat{x} \|} x_i$ where $\hat{x} \in V_0$ is the point such
that $\Phi_t(\hat{x})=\hat{z}$.
The map $\Phi_t$ (defined in (\ref{eq: mapPhit})) appears in coordinates $\{x^i\}$ and $\{w^i\}$ as
\be 
\Phi_t^i(x) = \bigg( x^i + {t \over 2 \|\hat{x} \|^2} {\bar{x}^i \over
  \sum_{i=1}^3 |x^i|^2 + |\sum_{i=1}^3 (x^i)^2|} \bigg) {\| \hat{x} \|
  \over \lambda}.
\ee
Recall that $\| x \|^2 \leq \| \Phi_t(x) \|^2 \leq 2 \| x \|^2$, and
so $\| \hat{x} \| \sim \lambda$ and
coordinates $z^i$ are in the range $\{ {1
  \over 4} \leq |z| \leq 4 \} \subset \mathbb{C}^3$. We may assume the coordinates $x^i$ on $V_0$ are in the range $\{ {1
  \over 2} \leq |x| \leq 2 \} \subset \mathbb{C}^3$. Abusing notation,
we simply write $w^i = w^i \circ \Phi_t(x)$. The change of coordinates is of
the form
\[
{\partial w^i \over \partial x^j} = {\| \hat{x} \| \over \lambda} \delta^i{}_j + {t \over \lambda^2} O(1)
\]
and hence
\be
{\partial \over \partial t} {\partial w^i \over \partial x^j} = \lambda^{-2} O(1),
\quad {1 \over 4} \delta^i{}_j \leq {\partial w^i \over \partial x^j} \leq 4 \delta^i{}_j.
\ee
Differentiating the inverse Jacobian then also gives
\[
{\partial \over \partial t} {\partial x^i \over \partial w^j} =
\lambda^{-2} O(1).
\]
We now compute the variation of $K_t$ in these
coordinates. In components $K_t = (K_t)_{\bar{k} j} \, dw^j \otimes
d \bar{w}^k$, we have
\[
{\partial \over \partial t} (K_t)_{\bar{k} j}(w) = {\partial \over
  \partial t} \bigg[ \overline{{\partial x^p \over \partial w^k}}
(H_0)_{\bar{p} q}(x(w)) {\partial x^q \over \partial w^j} \bigg].
\]
We note
\[
{\partial \over \partial t} \bigg[ (H_0)(x(w)) \bigg] =
{\partial H_0 \over \partial x^i}(x(w)) {\partial x^i \over \partial w^p}
{\partial w^p \over \partial t} = O(\lambda^{-2})  {\partial H_0
  \over \partial x^i} .
\]
Recall that in these coordinates, we have that $H_0 = r(\hat{x})^2
O(I)$, and we noted earlier that $r(\hat{x}) \sim
r(\hat{z})=\hat{r}$. Putting everything together, we have
\be \label{variation-K}
{\partial \over \partial t} (K_t)_{\bar{k} j}(w) =
\hat{r}^2 \lambda^{-2} O(1).
\ee
Since $K_t = \hat{r}^2
O(I)$ in these coordinates, it
follows that
\[
{\partial \over \partial t} (F_{K_t})_{j \bar{k}} = \lambda^{-2} O(1).
\]
Thus
\bea
& \ & {\partial \over \partial t} \bigg[ [(\Phi_t^{-1})^* g_{co,0}]^{j
  \bar{k}} (F_{K_t})_{j \bar{k}
  }\bigg] \nonumber\\
&=& \bigg[ {\partial \over \partial t} [(\Phi_t^{-1})^* g_{co,0}]^{j
  \bar{k}} (F_{K_t})_{j \bar{k}
  }\bigg] + \bigg[ [(\Phi_t^{-1})^* g_{co,0}]^{j \bar{k}} {\partial
    \over \partial t} (F_{K_t})_{j \bar{k}
  }\bigg] \nonumber\\
&=& \bigg[ {\partial \over \partial t} [(\Phi_t^{-1})^* g_{co,0}]^{j
  \bar{k}} (F_{K_t})_{j \bar{k}
  }\bigg] + \lambda^{-2} \hat{r}^{-2} O(1) .
\eea
Here we used (\ref{eq:decayGco0Pull}) and $g^{-1}_{co,t} = \hat{r}^{-2}O(I)$ in these coordinates. The same
computation as (\ref{variation-K}) gives
\[
{\partial \over \partial t} ( (\Phi_t^{-1})^* g_{co,0})_{\bar{k} j}(w)
= \hat{r}^2
\lambda^{-2} O(1)
\]
and therefore
\[
{\partial \over \partial t} \bigg[ [(\Phi_t^{-1})^* g_{co,0}]^{j
  \bar{k}} (F_{K_t})_{j \bar{k}
  }\bigg]  = \lambda^{-2} \hat{r}^{-2}O(1).
\]
Since $\lambda^2 = \hat{r}^3$, this completes the proof of (\ref{variation-pullback-conehym}).
\end{proof}

\section{Perturbation} \label{sec: perturbation}
At this stage in the construction, we have a pair of metrics $(g_t,H_t)$ on the smoothing $X_t$ such that both of these metrics agree with a scaling of $g_{co,t}$ near the vanishing cycles $\{ \| z\|^2 = t\}$. The metric $\omega_t$ is not balanced on all of $X_t$ and the metric $H_t$ is not Hermitian-Yang-Mills with respect to $\omega_t$ away from the vanishing cycles, but by the construction they are close to solving these equations. In this section we will perturb $(g_t,H_t)$ to a pair $(g_{{\rm FLY},t},\check{H}_t)$ solving the Hermitian-Yang-Mills equation. We will prove:

\begin{thm}\label{thm: mainTheorem}
 There exists $\epsilon>0$ such that for all $0<|t|<\epsilon$, there
 exists on $X_t$ a pair of hermitian metrics $(g_{{\rm FLY},t},\check{H}_t)$ solving
  \[
    d \omega_{{\rm FLY},t}^2 = 0, \quad F_{\check{H}_t} \wedge
    \omega_{{\rm FLY},t}^2=0.
  \]
Near the vanishing cycles, these metrics have the following local
description. There exists $\lambda,c_i,d_i>0$ such that for any $k \in
\mathbb{Z}_{\geq 0}$, there exists $C_k>0$ such that for all $|t|<\epsilon$
\be \label{eq:fly-local}
| \nabla^k_{g_{co,t}} (g_{{\rm FLY},t}-c_i g_{co,t}) |_{g_{co,t}} \leq C_k |t|^{2/3} r^{-k}
\ee
and
\be\label{eq: HS-local-decay}
| \nabla^k_{g_{co,t}} (\check{H}_t-d_i g_{co,t}) |_{g_{co,t}} \leq C_k |t|^\lambda r^{-k}
\ee
in the region
\[
\mathcal{R}_{\lambda} = \{ |t| \leq \|z\|^2 \leq |t|^{\frac{3}{3+\lambda}}\}.
\]
  \end{thm}
  
  \begin{rk}
  The decay estimates~\eqref{eq:fly-local} and~\eqref{eq: HS-local-decay} imply that, at an appropriate scale, the Hermitian-Yang-Mills metrics $H_t$ converge smoothly to a multiple of the CO metric $g_{co,0}$ as $|t|\rightarrow 0$.
  \end{rk}

\subsection{The $\theta$-perturbed Fu-Li-Yau metric}
We recall the construction of Fu-Li-Yau \cite{FLY} which
perturbs $\omega_t$ to a balanced metric $\omega_{{\rm FLY},t}$. The Fu-Li-Yau balanced metric is obtained via the ansatz
\be \label{fly-ansatz}
\omega_{{\rm FLY},t}^2 = \omega_t^2 + \theta_t + \bar{\theta}_t.
\ee
The $(2,2)$ form $\theta_t$ is constructed to satisfy
\[
\partial \theta_t = 0, \quad \bar{\partial} \theta_t = - \bar{\partial} \omega_t^2.
\]
More specifically, the correction
$\theta_t$ is of the form
\[
\theta_t = \partial
\bar{\partial}^\dagger \partial^\dagger \gamma_t
\]
where adjoints are
with respect to $g_t$ and $\gamma_t \in \Lambda^{2,3}(X_t)$ satisfies
$\partial \gamma_t=0$. Estimates
for $\gamma_t$ and $\theta_t$ were obtained by Fu-Li-Yau
\cite{FLY}. We will use the versions stated in \cite{Chuan}.
\cite[Proposition 3.6]{Chuan} states that
\[
|\theta_t|^2_{g_t} \leq C \| z \|^{-2/3} |t|,
\]
which using $\|z\|^2 \geq |t|$ implies
\be \label{FLY-estimate}
|\theta_t|_{g_t} \leq C |t|^{(2/3)}.
\ee
The proof of \cite[Proposition 3.8]{Chuan} uses
\be \label{gamma-L2}
\int_{X_t} |\gamma_t|^2_{g_t} \leq C |t|^2.
\ee
We will need the following higher estimate on $\nabla \theta$. 
\begin{lem} \label{lem-theta-grad-est}
\[
|\nabla^k \theta_t|_{g_t} \leq C_k |t|^{2/3} r^{-k}.
\]
\end{lem}
\begin{proof} This is similar to \cite[Proposition 3.7]{Chuan}. We
first show the estimate on a compact set $K$ which does not intersect the
vanishing cycles. The operator $E_t$ given by
\[
E_t = \partial \bar{\partial} \bar{\partial}^\dagger \partial^\dagger
+ \partial^\dagger \bar{\partial} \bar{\partial}^\dagger \partial +
\partial^\dagger \partial,
\]
where $\dagger$ is with respect to $g_t$, is a 4th order elliptic
operator. The form $\gamma_t$ satisfies
\[
E_t(\gamma) = \bar{\partial} \omega_t^2.
\]
In fact, it is obtained in \cite{FLY} by solving this equation. On
$K$, the geometry is uniform in $t$, hence by elliptic estimates we
have
\[
\| \gamma \|_{C^4(K)} \leq C( \| \gamma\|_{L^2(K)} + \| \bar{\partial}
\omega\|_{W^{k,p}(K)})
\]
for some $k,p>1$. As noted in Lemma \ref{prop: FLYalmost}, the construction of $\omega_t$ is
such that
\be \label{Theta-13-est}
|\bar{\partial} \omega_t^2|_{C^k(X_t,g_t)} \leq C_k |t|.
\ee
By (\ref{Theta-13-est}) and (\ref{gamma-L2}), we have
$\| \gamma \|_{C^4(K)} \leq C|t|$ and hence $\| \nabla \theta_t
\|_{L^\infty(K)} \leq C|t|$. Similarly, $|\nabla^k
\theta_t\|_{L^\infty(K)} \leq C_k |t|$.
\smallskip
\par We now prove the estimate stated in the lemma on a set $U(\epsilon) \cap X_t$ containing
the vanishing cycles and assume $g_t=R_0 g_{co,t}$. Here $\bar{\partial}
\theta =\bar{\partial} \omega_{co,t}^2= 0$, and we also have
\[
\bar{\partial}^\dagger \theta = \bar{\partial}^\dagger \partial
\bar{\partial}^\dagger \partial^\dagger \gamma_t = 0
\]
since $\partial$ and $\bar{\partial}^\dagger$ commute because
$g_t$ is K\"ahler on this set. Therefore
\[
\Delta_{\bar{\partial}} \theta_t = 0.
\]
Working in holomorphic cylindrical coordinates (see Lemma \ref{lem: cylCoords}), we can verify that the coefficients
of the equation $r^2 \Delta_{\bar{\partial}} \theta_t = 0$ are
uniformly bounded in $C^\alpha$. Indeed, by the Bochner-Kodaira formula,
\[
\Delta_{\bar{\partial}} \theta = - g_t^{i \bar{j}}
\partial_i \partial_{\bar{j}} \theta +
 \Gamma * \partial \theta +  \partial \Gamma * \theta + \Gamma * \Gamma * \theta+ Rm_{g_t} * \theta,
\]
and the uniform boundedness of the coefficients of $r^2
\Delta_{\bar{\partial}}$ follows from Lemma \ref{lem: cylCoords}. By the Schauder estimates in this coordinate chart, we obtain
\[
\sup_{B_{1/2}}| \partial \theta_t |_{g_{euc}} \leq C \sup_{B_1} | \theta_t |_{g_{euc}}.
\]
Using $g_{co,t} = r^2 O(I)$ in these coordinates, we obtain
\[
| \nabla \theta_t |_{g_t} \leq C r^{-1} \sup_{X_t} | \theta_t |_{g_t}.
\]
Since $| \theta_t |_{g_t} \leq C |t|^{2/3}$, we obtain the lemma for
$k=1$ and higher $k \geq 1$ are similar.
\end{proof}

\medskip
\par We now note some general facts on $(2,2)$ forms constructed via
the ansatz $\tilde{\omega}^2 = \omega^2 + \theta + \bar{\theta}$.

\begin{lem} \label{lem-squareroot} 
  On a complex manifold of dimension $n$, the equation
  \[
    \omega^{n-1} = \Psi>0
  \]
  has solution
\[
g_{\bar{k} j} = (\det \, \Psi)^{1/(n-1)} (\Psi^{-1})_{\bar{k} j},
\]
where $\omega = \sqrt{-1} g_{\bar{k} j} \, dz^j \wedge d \bar{z}^k$
and $\Psi$ is written as
\[
\sum_{k,j} c_{kj} \Psi^{k \bar{j}}
dz^1 \wedge d \bar{z}^1 \wedge \cdots \wedge \widehat{dz^k} \wedge d
\bar{z}^k \wedge \cdots \wedge dz^j \wedge \widehat{d \bar{z}^j} \wedge
\cdots \wedge dz^n \wedge d \bar{z}^n
\]
with $c_{kj} = (\sqrt{-1})^{n-1} (n-1)! sgn(k,j)$.
\end{lem}
\begin{proof} See e.g. \cite{Michelsohn} or \cite{PPZ2}. Direct computation of
$\omega^{n-1}$ gives
$(\det g) g^{j \bar{k}} = \Psi^{j \bar{k}}$ and the result follows
from taking the determinant of both sides and solving for $g$.
\end{proof}

\begin{lem} \label{lem-diff-squareroot} Let $\nabla$ be the Chern connection with respect to a Hermitian
  metric $\omega$. Let $\eta = \sqrt{-1} \eta_{\bar{k} j} \, dz^j \wedge d
  \bar{z}^k$ be a positive $(1,1)$ form solving $\eta^2 = \omega^2 + \theta + \bar{\theta}$,
  where
  \[
\theta = {1 \over 4} \theta_{s \bar{r} j \bar{k}} \, dz^s \wedge d
\bar{z}^r \wedge dz^j \wedge d \bar{z}^k.
  \]
  Then
\[
\nabla_i\eta_{\bar{k}
  j} = -{1 \over 2} \eta^{s \bar{r}} (\nabla_i \theta _{s \bar{r} j \bar{k}}
+ \nabla_i \bar{\theta} _{s \bar{r} j \bar{k}}) + {1 \over 8} \bigg[ \eta^{p \bar{q}} \eta^{s \bar{r}}(\nabla_i
\theta_{s \bar{r} j \bar{k}} + \nabla_i \bar{\theta}_{s \bar{r} p \bar{q}}) \bigg] \eta_{\bar{k} j} .
\]
  
  \end{lem}
\begin{proof}
 A similar computation can be found in \cite{PPZ5}. In components, the equation $\eta^2 = \omega^2 + \theta + \bar{\theta}$ is
  \[
- 2 \eta_{\bar{r} s} \eta_{\bar{k} j} + 2 \eta_{\bar{r} j} \eta_{\bar{k} s} =
(\omega^2 + \theta + \bar{\theta})_{s \bar{r} j \bar{k}}.
  \]
  Differentiating this equation leads to
 \[
- \nabla_i\eta_{\bar{r} s} \eta_{\bar{k} j}  - \eta_{\bar{r} s} \nabla_i\eta_{\bar{k}
  j} + \nabla_i \eta_{\bar{r} j} \eta_{\bar{k} s} + \eta_{\bar{r} j}
\nabla_i \eta_{\bar{k} s} = {1 \over 2} (\nabla_i \theta _{s \bar{r} j \bar{k}}
+ \nabla_i \bar{\theta} _{s \bar{r} j \bar{k}})
 \]
Contracting by $\eta^{s \bar{r}}$ gives
\[
- (\eta^{s \bar{r}} \nabla_i\eta_{\bar{r} s}) \eta_{\bar{k} j}  -  \nabla_i\eta_{\bar{k}
  j} = {1 \over 2} \eta^{s \bar{r}} (\nabla_i \theta _{s \bar{r} j \bar{k}}
+ \nabla_i \bar{\theta} _{s \bar{r} j \bar{k}})
\]
Contracting again by $\eta^{j \bar{k}}$ gives
\[
-4 \eta^{s \bar{r}} \nabla_i \eta_{\bar{r} s} = {1 \over 2} \eta^{j \bar{k}}\eta^{s \bar{r}} (\nabla_i \theta _{s \bar{r} j \bar{k}}
+ \nabla_i \bar{\theta} _{s \bar{r} j \bar{k}}).
\]
Combining the previous two identities proves the lemma.
\end{proof}
\medskip
\par Using what we have obtained so far in this subsection, we can
derive the main estimate of this subsection which shows that the
difference between $g_t^{-1}$ and $(g_{{\rm FLY},t})^{-1}$ is small.

\begin{lem} \label{lem:gfly-gt}
There exists $C_k>0$ and $\epsilon>0$ with the following property. For
all $0<|t| <\epsilon$, the $\theta$-perturbed Fu-Li-Yau metric
$g_{{\rm FLY},t}$ satisfies the estimates:
  \[
2^{-1} g_t \leq g_{{\rm FLY},t} \leq 2 g_t,
  \]
  and
 \[
| \nabla^k_{g_t}( g _{{\rm FLY},t} - g_t) |_{g_t} \leq C_k |t|^{2/3} r^{-k},
  \]
  for $k \in \mathbb{Z}_{\geq 0}$. Furthermore, we have
  \[
\| \omega^{-1}_{{\rm FLY},t} - \omega^{-1}_t
\|_{C^{0,\alpha}_0} \leq C_0 |t|^{2/3}.
  \]
 
\end{lem}

\begin{proof}
If $|\theta_t+\bar{\theta}_t|_{g_t}
\leq {1 \over 100}$, then by Lemma \ref{lem-squareroot} we have
\[
| g_{{\rm FLY},t}
|_{g_t} \leq 2, \quad | g_{{\rm FLY},t}^{-1}
|_{g_t} \leq 2.
\]
Next, we write the difference of metrics as
\[
\omega_{{\rm FLY},t} - \omega_t = \int_0^1 {d \over ds} \eta_{s} \, ds
\]
where $\eta_{s}$ solves $\eta_{s}^2 = \omega_t^2 + s
(\theta_t + \bar{\theta}_t)$. By the variation formula in Lemma
\ref{lem-diff-squareroot}, we have
\[
{d \over ds} (\eta_s) _{\bar{k} j}= - {1 \over 2} \eta_s^{s
  \bar{r}}(\theta+\bar{\theta})_{s \bar{r} j \bar{k}} + {1 \over 8}[
\eta_s^{p \bar{q}} \eta_s^{s \bar{r}} (\theta+\bar{\theta})_{s \bar{r} p
  \bar{q}} ]\, (\eta_s)_{\bar{k} j}.
\]
The same argument as above shows that $| \eta_s |_{g_t}
\leq C$ and $| \eta_s^{-1} |_{g_t} \leq C$ for $s \in [0,1]$. Therefore
\[
| g_{{\rm FLY},t} - g_t |_{g_t} \leq C |\theta_t |_{g_t}
\leq C|t|^{2/3}
\]
by (\ref{FLY-estimate}). Next, by Lemma \ref{lem-diff-squareroot} and
Lemma \ref{lem-theta-grad-est}, we have
\[
|\nabla_{g_t} (g_{{\rm FLY},t})|_{g_t} \leq C |\nabla_{g_t} \theta|_{g_t} \leq C
|t|^{2/3} r^{-1}.
\]
Higher order estimates for $|\nabla^k_{g_t} (g_{{\rm FLY},t})|_{g_t}$ are
similar.
\smallskip
\par It remains to estimate the difference, which can be done by:
\bea
\| \omega^{-1}_{{\rm FLY},t} - \omega^{-1}_t
\|_{C^{0,\alpha}_0} &=& \| \omega_{{\rm FLY},t}^{-1}
(\omega_{{\rm FLY},t} - \omega_t)
\omega_t^{-1}\|_{C^{0,\alpha}_0} \nonumber\\
&\leq& \| \omega_{{\rm FLY},t}^{-1} \|_{C^{0,\alpha}_0}
\| \omega_{{\rm FLY},t} - \omega_t\| _{C^{0,\alpha}_0} 
\| \omega_t^{-1}\|_{C^{0,\alpha}_0}.
\eea
Since $|\nabla \omega^{-1}_{{\rm
    FLY},t}|_{g_t} \leq C |t|^{2/3} r^{-1} $, we have
\[
\| \omega_{{\rm FLY},t}^{-1} \|_{C^{0,\alpha}_{0}} \leq C (1+r|\nabla \omega^{-1}_{{\rm
    FLY},t}|_{g_t}) \leq C.
\]
The estimate
\[
  | \omega_{{\rm FLY},t} - \omega_t|+r| \nabla_{g_t} (\omega_{{\rm FLY},t} - \omega_t)|
  \leq C|t|^{2/3}
\]
implies
\[
\| \omega_{{\rm FLY},t} - \omega_t \|_{C_0^{0,\alpha}} \leq C
|t|^{2/3} 
\]
which proves the lemma.
\end{proof}

\subsection{Uniform weighted H\"older estimates}

Let $g_t$ be the metric constructed by Fu-Li-Yau which satisfies
$g_t=g_{co,t}$ near the vanishing cycles (for ease of notation, in this section we assume that the constant $M_0^{1/2} \epsilon^{-1/3}$ in Proposition \ref{prop: FLYalmost} is equal to 1). Let $H_t$ be the metric on $X_t$ constructed in the previous section,
i.e. the glued approximate solution to
the Hermitian-Yang-Mills equation. We will use a linear operator $L_t$ which
acts on endomorphisms $h \in \Gamma({\rm End} \, T^{1,0} X_t)$ by
\[
L_{t} h = (g_t)^{j \bar{k}}  \partial_{\bar{k}} \nabla^{H_t}_j h + {1
  \over 2} [ \sqrt{-1} \Lambda_{\omega_t} F_{H_t}, h].
\]
The motivation for this operator is that it is close to the
linearization of the Hermitian-Yang-Mills equation $\sqrt{-1}
\Lambda_{\omega_{{\rm FLY},t}} F_H=0$ at the approximate solution
$H_t$, the difference being the use of $g_t$ instead of $g_{{\rm FLY},t}$. We start by proving uniform Schauder estimates independent of $t$ on
the deformations $X_t$. 

\begin{prop} \label{unif-schauder}
Let $\beta \leq 0$. There exists $C>1$ and $a \in (0,1)$ such that for all $t \in
\mathbb{C}^*$ and $h \in \Gamma({\rm
  End} \, T^{1,0} X_t) $, we can estimate
\[
\| h \|_{C^{2,a}_\beta(X_t)} \leq C (\| h \|_{C^0_\beta(X_t)} + \| L_{t} h
\|_{C^{0,a}_{\beta-2}(X_t)}  )
\]
where the weighted H\"older norms are defined in \S
\ref{section:holder} using $(g_t,H_t)$.
\end{prop}

\begin{proof}
  On $X_t \cap \{ r > 1 \}$, the geometry is uniform in $t$
and the
estimate holds by the usual Schauder
estimates. Let $\hat{x} \in X_t \cap \{ r(x) \leq 1 \}$. We denote
the scale of this point by the
constant $\hat{r} :=
r(\hat{x})$. We will work in
holomorphic cylindrical coordinates $\{ w^i \}$ given in Lemma
\ref{lem: cylCoords} on the set
\[
  U_{\hat{r}} =\{ (1/4) \hat{r} < r < 4 \hat{r} \}.
 \]
By Lemma \ref{lem: cylCoords} and Lemma \ref{lem: H_t-est}, the operator $\hat{r}^2 L_t$ is
uniformly elliptic with uniform derivative estimates in
coordinates $\{w^i\}$. The standard Schauder
estimates applied to each matrix entry $\hat{r}^{-\beta} h^i{}_j$ imply
\[
\| \hat{r}^{-\beta} h \|_{C^{2,\alpha}_{g_{\rm euc}}(B_{1/2})} \leq C (\|
\hat{r}^{-\beta} h
\|_{C^{0}_{g_{\rm euc}}(B_{1})} + \|\hat{r}^2 L_t (\hat{r}^{-\beta}
h) \|_{C^{\alpha}_{g_{\rm euc}}(B_{1})})
\]
with usual (non-scaled) norms
\[
\| u \|_{C^k_{g_{\rm euc}}(B_1)} = \sum_{i=1}^k \sup_{B_1} |D^i u| , \quad
[u]_{C^\alpha_{g_{\rm euc}}(B_1)} = \sup_{x \neq y} {|u(x)-u(y)| \over |x-y|^\alpha}.
\]  
As observed in (\ref{weighted-alt-norm}), since $\hat{r}^{-2} g_t$
is uniformly and smoothly equivalent to $g_{\rm euc}$ in
coordinates $\{ w^i \}$, the weighted H\"older norms are equivalent to these local
Euclidean norms, and we have
\[
\| h \|_{C^{2,a}_\beta(U_{\hat{r}})} \leq C ( \| h \|_{C^0_\beta(U_{\hat{r}})} + \| L_t
h \|_{C^a_{\beta-2}(U_{\hat{r}})} ).
\]
The norm $\| \cdot \|_{C^{2,a}_\beta(U_{\hat{r}})}$ involves
connection terms from $\nabla^{H_t}$, but these are bounded in
coordinates $\{w^i \}$ by Lemma \ref{g_t-unif-H_t}. By Lemma \ref{holder-local2global}, these local estimates on sets
$U_{\hat{r}}$ imply the global bound.
\end{proof}

\medskip
\par The next step is to improve this estimate for endomorphisms
orthogonal to the identity. For a related argument used in a gluing
construction of K\"ahler-Einstein
metrics on nodal surfaces, see \cite{Spotti}.

\begin{prop} \label{prop-invert-L}
Let $\beta \in (-2,0)$. There exists $C>1$ and $\alpha \in (0,1)$ with
the following property. Let $t \in
\mathbb{C}$ with $0<|t|<1$ be arbitrary. We can estimate
\[
\| h \|_{C^{2,\alpha}_\beta(X_t)} \leq C \| L_{t} h \|_{C^\alpha_{\beta-2}(X_t)} ,
\]
for all $h \in \Gamma({\rm  End} \, T^{1,0} X_t)$ satisfying $h^{\dagger_{H_t}}=h$ and $\int_{X_t} ({\rm Tr} \, h) \, d
{\rm vol}_{g_{{\rm FLY},t}} =0$.
\end{prop}

\begin{proof} Suppose there exists a sequence of $t_i \rightarrow 0$ such that
\[
\| h_i \|_{C^{2,\alpha}_\beta} \geq M_i \| L_{{t_i}} h_i \|_{C^\alpha_{\beta-2}} 
\]
with $M_i \rightarrow \infty$ and $h_i$ defined on $X_{t_i}$. Replacing $h_i$ with $h_i / \| h_i
\|_{C^{2,\alpha}_\beta}$, we have a sequence with
\[
\| h_i \|_{C^{2,\alpha}_\beta(X_{t_i})} = 1, \quad \| L_{{t_i}} h_i
\|_{C^\alpha_{\beta-2}(X_{t_i})} \rightarrow 0.
\]
Let $K \subset \underline{X}$ be a compact set on the central fiber disjoint from the singular
points. For all $t$ small enough, we have a sequence $\Phi_{t_i}^*h_i$ of tensors defined on $K$. Since $C^{-1} r(x) \leq r(\Phi_t(x)) \leq C r(x)$, we have a uniform bound on
\[
  \|
  \Phi_{t_i}^*h_i \|_{C^{2,\alpha}_\beta(K,\Phi_{t_i}^* g_i, \Phi_{t_i}^* H_i)} \leq C.
\]
By Lemma \ref{prop: FLY2.6}, $\Phi_{t_i}^* g_{t_i} \rightarrow g_0$
smoothly uniformly on compact sets, and the definition of $H_t$
implies that $\Phi_{t_i}^* H_{t_i} \rightarrow H_0$ smoothly uniformly on compact
sets. We can thus extract a limiting tensor $h_0 \in C^{2,\alpha/2}_{loc}(\underline{X})$ which satisfies the growth estimates
\[
|h_0|_{g_0} \leq C r^\beta, \quad |\nabla_{H_0} h_0|_{g_0} \leq C r^{\beta-1}.
\]
By the construction of $\Phi_t$ in Lemma \ref{lem: trivGlob}, the limit $h_0$ preserves $T^{1,0} \underline{X}$ and is an endomorphism of this bundle which satisfies the identities
\be \label{eq:blowuplimitX_0}
(g_{0})^{j \bar{k}} \partial_{\bar{k}} \nabla_j^{H_0} h_0 = 0, \quad
h_0^\dagger = h_0
\ee
where $\dagger$ is with respect to $H_0$. We now show that
\be \label{int-zero}
\int_{\underline{X}} ({\rm Tr} \, h_0) \, d {\rm vol}_{g_0} =0.
\ee
For this, we let $\delta>0$, so that $(\Phi_{t_i})^*h_{t_i}, (\Phi_{t_i})^*g_{t_i}$ converge uniformly
on $\{ r > \delta \}$. By Lemma
\ref{lem:gfly-gt}, we also have that $(\Phi_{t_i})^* g_{{\rm FLY},t_i}
\rightarrow g_0$ uniformly in the $C^0$ norm on $\{ r > \delta \}$. Therefore
\[
\int_{\underline{X} \cap \{r > \delta \}} ({\rm Tr} \, h_0) \, d {\rm
  vol}_{g_0} = \lim_i\int_{X_{t_i} \cap \{r > \delta \}} ({\rm Tr} \, h_i) \, d {\rm
  vol}_{g_{{\rm FLY},t_i}}.
\]
Since $\int_{X_{t_i}} {\rm Tr} \, h_i = 0$, then
\be \label{eq:int-zero-step}
\int_{\underline{X}} ({\rm Tr} \, h_0) \, d {\rm vol}_{g_0} = -\lim_{\delta
  \rightarrow 0} \lim_i \int_{X_{t_i} \cap \{ r < \delta \}} ({\rm Tr} \, h_i) \, d {\rm
  vol}_{g_{{\rm FLY},t_i}} .
\ee
From $d {\rm
  vol}_{g_{{\rm FLY},t}}  \leq C d {\rm vol}_{g_t}$ (Lemma
\ref{lem:gfly-gt}) and $|h_i| \leq r^\beta$, we have
\bea
\bigg| \int_{X_{t_i} \cap \{ r < \delta \}} ({\rm Tr} \, h_i) \, d {\rm
  vol}_{g_{{\rm FLY},t_i}} \bigg| &\leq& C \delta^\beta \int_{X_{t_i} \cap \{ r <
  \delta \}} d {\rm vol}_{g_{t_i}} \nonumber\\
&=& C \delta^\beta \int_{X_{1} \cap \{ r <
  \delta |t_i|^{-1/3}\}} S_{t^{1/3}}^* d {\rm vol}_{g_{t_i}} \nonumber\\
&=& C \delta^\beta \int_{X_{1} \cap \{ r <
  \delta |t_i|^{-1/3}\}} |t_i|^2 d {\rm vol}_{g_{co,1}} \nonumber
\eea
where $S_{t^{1/3}}: V_1 \rightarrow V_t$, $S_{t^{1/3}}(z)=t^{1/2}z$ is
the scaling action (\ref{eq: scalingMapDef}) which satisfies $S_{t^{1/3}}^* g_{co,t}
= |t|^{2/3} g_{co,1}$.  We have ${\rm Vol}_{g_{co,1}}(\{ r < R \}) =
O(R^6)$ since $g_{co,1}$ is asymptotically conical, and so
\[
\bigg| \int_{X_{t_i} \cap \{ r < \delta \}} ({\rm Tr} \, h_i) \, d {\rm
  vol}_{g_{{\rm FLY},t_i}} \bigg| \leq C \delta^{6+\beta}
\]
which together with (\ref{eq:int-zero-step}) proves (\ref{int-zero}).
\smallskip
\par Next, (\ref{eq:blowuplimitX_0}) implies that the identity
\[
\Delta_{g_0} |h_0|^2_{H_0} = 2 |\overline{\nabla} h_0|^2_{H_0,g_0}
\]
holds pointwise away from the nodes. Let $\eta_\delta$ be a cutoff
function such that $\eta_\delta \equiv 0$ on $\{ r < {\delta \over 2}
\}$, $\eta_\delta \equiv 1$ on $\{ r > \delta \}$ and $|\Delta_g
\eta_\delta| \leq C \delta^{-2}$. Then
\[
2 \int_{U_\delta} |\overline{\nabla}
h_0|^2_{H_0,g_0} \, d {\rm vol}_{g_0} \leq \int_{\underline{X}}
\eta_\delta \Delta_{g_0} |h_0|^2_{H_0} \, d {\rm vol}_{g_0}.
\]
Recall that $g_0$ is balanced on $\underline{X}$ and so we can
integrate by parts.
\[
2 \int_{U_\delta} |\overline{\nabla}
h_0|^2_{H_0,g_0} \, d {\rm vol}_{g_0} \leq C \delta^{-2} \int_{
  \{{\delta \over 2} < r < \delta \}}  |h_0|^2_{H_0}  d {\rm vol}_{g_{co,0}}.
\]
Since $h_0 \in C^1_\beta$ and $d {\rm vol}_{g_{co,0}}=r^5 dr \dvol_{g_L}$, then
\[
2 \int_{U_\delta} |\overline{\nabla}
h_0|^2_{H_0,g_0} \, d {\rm vol}_{g_0} \leq C \delta^{-2+2 \beta + 6}
\]
and we conclude that
\[
\limsup_{\delta \rightarrow 0}  \int_{U_\delta} |\overline{\nabla}
h_0|^2_{H_0} \, d {\rm vol}_{g_0} = 0
\]
for $\beta \in (-2,0)$. Therefore $\overline{\partial} h_0 = 0$ on
$\underline{X}$, and $h_0$ is a holomorphic endomorphism. By Hartog's theorem,
$h_0$ extends across the holomorphic curves to the small resolution
$X$. Since $T^{1,0} X$ is stable with respect to $\omega_{\rm CY}$, it
must be the case that $h_0 = c \, I$ is a multiple of the
identity. We showed in (\ref{int-zero}) that the integral of ${\rm Tr}
\, h_0$ is zero, so we conclude
that
\[
  h_0 \equiv 0.
\]
The goal now is to obtain a contradiction to this by using that $\|
h_i \|_{C^{2,\alpha}_\beta} = 1$ along the sequence $t_i \rightarrow
0$. The uniform Schauder estimates in Proposition \ref{unif-schauder} imply
\[
1 \leq C (\| h_i \|_{C^0_\beta} + M_i^{-1}),
\]
and hence
\[
|r^{-\beta} h_i| (z_i) \geq C^{-1}
\]
for a sequence $z_i \in X_{t_i}$. 
\medskip
\par $\bullet$ Case 1: Suppose $\liminf r(z_i)>0$. It then follows
that after taking a subsequence, we have $z_i
\rightarrow z_0 \in \underline{X}$ with $r(z_0)> 0$. Then
\[
|h_0(z_0)| \geq C^{-1} r(z_0)^\beta > 0
\]
which contradicts $h_0 \equiv 0$.
\medskip
\par $\bullet$ Case 2: Suppose $r(z_i) \rightarrow 0$. In this
case, we can assume that all points $z_i$ are in the region of $X_t$
which can be identified with a subset of $V_t = \{
\sum z_i^2 = t \}$ and where $g_t = g_{co,t}$. Define the function
$u_i: V_{t_i} \cap \{ \| z \|^2 \leq 1 \}
 \rightarrow \mathbb{R}$ given by
\[
u_i = |h_i|^2_{H_{t_i}}.
\]
The sequence $\{ u_i \}$ satisfies the uniform growth estimate
\be
\| u_i \|_{C^{2,\alpha}_{2\beta}(g_t)} \leq C,
\ee
which written in full is
\[
|u|+r^{-1}|\nabla
u|_{g_{co,t}}+r^{-2}|\nabla^2_{g_{co,t}} u|_{g_{co,t}} +
[\nabla^2_{g_{co,t}} u]_{C^{0,a}_{\beta-2-a}} \leq C r^{2\beta}.
\]
This definition of $\| u \|_{C^{2,\alpha}_{2\beta}(g_t)} $ is slightly different than the weighted H\"older norms
used previously for $h_t$, since we use $\nabla$ with respect to $g_t$
(rather than $H_t$). These estimates for $u_i$ follow from $\| h_i
\|_{C^{2,\alpha}_\beta(g_t,H_t)} \leq 1$ and (\ref{g_t-unif-H_t}), which allow
us to uniformly convert norms in $H_t$ to norms in $g_t$. 
\smallskip
\par Direct computation gives the identity
\[
\Delta_{g_i} u_i = 2 {\rm Re} \,
\langle L_i h_i, h_i \rangle_{H_i}  + 2 |\overline{\nabla} h_i|^2_{H_i,g_i}.
\]
We are assuming $\| L_{t_i} h_i \|_{C^\alpha_{\beta-2}} \leq
\epsilon_i$ with $\epsilon_i \rightarrow 0$, hence
\be \label{Delta-u-epsilon}
\Delta_{g_i} u_i \geq - C \epsilon_i r^{2\beta-2}.
\ee
We will rescale the functions $u_i$ to take a limit. For ease of notation, we write $\lambda_i =
r(z_i)$. We will use the scaling map $S_{\lambda_i} : V_{t_i
  \lambda_i^{-3}} \rightarrow V_{t_i}$ from (\ref{eq: scalingMapDef})
given by $S(x)= \lambda_i^{3/2} x$.
We rescale and pullback $u_i$ via $S_{\lambda_i}$ to obtain a function 
\[
\tilde{u}_i : V_{t_i \lambda_i^{-3}} \cap \{ \|x\|^2 \leq
\lambda_i^{-3} \} \rightarrow \mathbb{R}
\]
defined by
\[
\tilde{u}_i(x) = \lambda_i^{-2 \beta} u(\lambda_i^{3/2} x).
\]
The rescaling is setup so that the estimates for $u_i$ imply estimates
for $\tilde{u}_i$. For example, we have
\be \label{Delta-u-epsilon2}
\Delta_{g_{t_i \lambda_i^{-3}}} \tilde{u}_i \geq - C \epsilon_i  r^{2\beta-2}.
\ee
Indeed, pulling back the Laplacian gives
\[
S_{\lambda_i}^* (\Delta_{g_{t_i}} u_i) = \lambda_i^{2 \beta} \Delta_{S^*
  g_{t_i}} \tilde{u}_i = \lambda_i^{2 \beta} \lambda_i^{-2} \Delta_{g_{t_i \lambda^{-3}}} \tilde{u}_i
\]
by using the rescaling relation
$S^*_\lambda g_t = \lambda^2 g_{t \lambda^{-3}}$. Using
$r(\lambda^{3/2}x) = \lambda r(x)$, we obtain
(\ref{Delta-u-epsilon2}) from (\ref{Delta-u-epsilon}). Similar computations show that
\be \label{est-tilde-u}
|\tilde{u}_i|+r^{-1}|\nabla
\tilde{u}_i|_{g_{co,t_i
    \lambda_i^{-3}}}+r^{-2}|\nabla^2\tilde{u}_i|_{g_{co,t_i \lambda_i^{-3}}} +
[\nabla^2 \tilde{u}_i]_{C^{0,a}_{\beta-2-a}} \leq C r^{2\beta}.
\ee
Recall that the points $z_i$ satisfy $|h_i|(z_i) \geq C^{-1}
\lambda_i^\beta$. Then the points $x_i = \lambda_i^{-3/2} z_i$ satisfy
\[
\tilde{u}_i(x_i) \geq C^{-1}.
\]
The sequence $\{ t_i \lambda_i^{-3} \}$ lies in $[0,1]$, since
$\| z \|^2 \geq |t|$ implies $r^3(z_i) \geq |t_i|$. After taking a
subsequence, we have convergence $t_i \lambda_i^{-3} \rightarrow
\kappa$ for $\kappa \in [0,1]$.
\medskip
\par $\bullet$ Case 2a: $t_i \lambda_i^{-3} \rightarrow \kappa>0$. The points $x_i =
\lambda_i^{-3/2} z_i$ satisfy $\| x_i \|=1$. We may assume
\[
x_i \rightarrow x_\infty \in V_\kappa, \quad \| x_\infty \|=1.
\]
We can take a limit of $\{ \tilde{u}_i \}$ on compact sets and obtain a $C^{2,\alpha}_{loc}$ limiting function
$u_\infty \geq 0$ on $V_\kappa$ satisfying $u_\infty \leq C r^{2 \beta}$
and $\Delta_{g_\kappa} u_\infty \geq 0$. By the maximum principle,
\[
  \sup_{r \leq R} u_\infty  \leq \sup_{r =R}  u_\infty \leq R^{2 \beta}.\]
Letting $R \rightarrow \infty$, we obtain that $u_\infty
\equiv 0$ since $\beta <0$. This contradicts $u_\infty (x_\infty) > 0$.
\medskip
\par $\bullet$ Case 2b: $t_i \lambda_i^{-3} \rightarrow 0$. In this
case, the points $x_i=\lambda_i^{-3/2} z_i$ converge after a subsequence
to
\[
x_i \rightarrow x_\infty \in V_0, \quad \|x_\infty \|=1.
\]
Let $v_i: \{  |t_i| \lambda_i^{-3} < \| x\|^2 < {1 \over 2} \lambda_i^{-3} \} \cap V_0
\rightarrow \mathbb{R}$ with $v_i = \Phi_{t_i \lambda_i^{-3}}^* \tilde{u}_i$ be the
corresponding sequence of functions on the cone $V_0$ (recall $\Phi_t$ is defined in (\ref{eq: mapPhit})). Since $ \| x \|^2 \leq
 \| \Phi(x) \|^2 \leq 2 \| x \|^2$, we have the growth estimate
\[
v_i \leq C r^{2 \beta}.
\]
By pulling-back (\ref{est-tilde-u}), on compact sets $K$ we have the estimate
\[
\| v_i
\|_{C^{2,\alpha}(K,\Phi^*g_{t_i \lambda_i^{-3}})} \leq
C(K), \quad
\Delta_{\Phi^*g_{t_i \lambda_i^{-3}}} v_i \geq -\epsilon_i C(K) .
\]
Corollary \ref{cor: decayGcotPull} implies $\Phi^*g_{t_i \lambda_i^{-3}}
\rightarrow g_{co,0}$ uniformly on $K$. Taking a limit of $\{ v_i \}$ on compact sets produces a
$C^{2,\alpha}_{loc}$ limiting function $v_\infty
\geq 0$
on the cone $V_0$ satisfying
\[
\Delta_{g_{co,0}}
v_\infty \geq 0, \quad v_\infty \leq C r^{2\beta}
\]
for $\beta \in
(-2,0)$. Lemma \ref{lem-subharmonic} below implies that $v_\infty \equiv 0$ which contradicts
$v_\infty(x_\infty)>0$.
\end{proof}

\begin{lem} \label{lem-subharmonic}
Let $V_0$ be a Riemannian cone of dimension $n>2$ with metric $g = d
r^2 + r^2 g_L$. Let $u$ be a $C^2$ function satisfying
$\Delta_{g} u \geq 0$ and $u \geq 0$. Suppose there exists
 $M>0$ such that $u \leq M
r^{- \delta}$ where $\delta \in (0,n-2)$. Then $u \equiv 0$.
\end{lem}

\begin{proof} This is a standard PDE result  (e.g. \cite{Han-Lin}), but we give the proof for completeness. Recall that the real Laplacian is
\[
\Delta_{g} u = \partial_r \partial_r u + (n-1) r^{-1}
\partial_r u + r^{-2} \Delta_{g_L} u.
\]
Let $B_R(0) = \{ x \in V_0 : r(x) < R \}$. We start by noting that for any $\varphi \in
C^2(\partial B_R(0),\mathbb{R})$, there exists $h \in C^2(B_R(0),\mathbb{R})$ such that
$h|_{\partial B_R(0)} = \varphi$ and
\[
\Delta_{g} h = 0, \quad \sup_{B_R(0)} |h| \leq \sup_{\partial B_R(0)} |\varphi|.
\]
To obtain such a harmonic function $h$, we start by expanding
$\varphi|_{\partial B_R(0)} =
\sum_{\lambda \in {\rm Spec}(\Delta_{g_L})} c_\lambda \psi_\lambda$,
where $\psi_\lambda$ are an $L^2$ orthogonal basis of eigenfunctions of
$\Delta_{g_L}$ on the link $L=\{ r = 1\}$, with eigenvalue
convention $\Delta_L \psi_\lambda = - \lambda \psi_\lambda$. We then let
\[
h = \sum_{\lambda \in {\rm Spec}(\Delta_{g_L})} c_\lambda \bigg({r \over
  R} \bigg)^{a(\lambda)} \psi_\lambda
\]
where $a(\lambda) = {1 \over 2}( -(n-2)+ \sqrt{(n-2)^2+4 \lambda})
>0$. Direct computation gives $\Delta_g h = 0$, and by the maximum
principle
\[
\sup_{B_R \backslash B_\epsilon} |h| \leq \max \{ \sup_{\partial B_R}
|\varphi|, \sup_{\partial B_\epsilon} |h| \}.
\]
for any $\epsilon>0$. As $\epsilon \rightarrow 0$, we see that $\sup_{\partial
  B_\epsilon} |h|$ selects the $\lambda=0$ mode $c_0 \psi_0$. This is
a constant equal to ${1 \over {\rm Vol}(B_R)} \int_{\partial B_R}
\varphi$, which is bounded by $ \sup_{\partial B_R}
|\varphi|$.
\smallskip
\par We now prove that the subharmonic function $u$ given in the lemma
is constant. Let $R>1$. Let $h_R$ be the harmonic function mentioned
above with $h_R|_{\partial B_R} = u$. Then
\be \label{h-R-delta}
|h_R| \leq \sup_{\partial B_R} |u| \leq M R^{-\delta}.
\ee
Let $0<\epsilon<1$, and consider
\[ v = u - h_R - 2 M \epsilon^{n-2-\delta} r^{-(n-2)}  \]
defined on $B_R \backslash B_\epsilon$. Since $\Delta_g r^{-(n-2)}=0$
we have that
$\Delta_{g} v \geq 0$, and
\[
v|_{\partial B_\epsilon} \leq |u| + |h_R| - 2M \epsilon^{-\delta} \leq M
\epsilon^{-\delta} + MR^{-\delta} - 2M \epsilon^{-\delta} \leq 0.
\]
We also have
\[
v|_{\partial B_R} = - {2M \epsilon^{n-2-\delta} \over R^{n-2}} <0.
\]
By the maximum principle, $v \leq 0$ in $B_R \backslash
B_\epsilon$. We now fix $x \in B_R \backslash B_\epsilon$. Then $v(x)
\leq 0$ implies
\[
u(x) \leq h_R(x) + 2M {\epsilon^{n-2-\delta} \over r^{n-2}(x)}.
\]
This holds true for all $0<\epsilon<1$, hence taking $\epsilon
\rightarrow 0$ we obtain
\[
u(x) \leq h_R(x).
\]
By (\ref{h-R-delta}), we conclude
\[
u(x) \leq M R^{-\delta}.
\]
We now take $R \rightarrow \infty$ to conclude $u \leq 0$. Since $u
\geq 0$ by assumption, we obtain that $u \equiv 0$.
\end{proof}

\subsection{Inverting the linearized operator}
Let
  \be \label{eq:W_t}
W_t= \bigg\{ u \in \Gamma({\rm End} \, T^{1,0}X_t) \ : \ u^\dagger
= u , \ \int_{X_t} ({\rm Tr} \, u) d {\rm vol}_{g_{{\rm FLY},t}} =0 \bigg\}
\ee
where $\dagger$ is the adjoint with respect to $H_t$.  When linearizing the equation
$\Lambda_{\omega_{{\rm FLY},t}}
F_{H}=0$ at the approximate solution $H_t$ we obtain an
operator $\LL$ that acts on endomorphisms by
\[
\LL u = -(g_{{\rm FLY},t})^{j \bar{k}} \partial_{\bar{k}}
\nabla_j^{H_t} u - {1 \over 2} [ \sqrt{-1} \Lambda_{\omega_{{\rm FLY},t}} F_{H_t},u].
\]
This operator $\LL$ involves $g_{{\rm FLY},t}$ rather than $g_t$, so it is
a perturbation of the operator $L$ in Proposition
\ref{prop-invert-L}.
\smallskip
\par We note that $\LL: W_t \rightarrow W_t$. Indeed, since $u^\dagger =
u$, $(\sqrt{-1} \Lambda F)^\dagger = \sqrt{-1} \Lambda F$ and $(g^{j \bar{k}} \nabla_j \nabla_{\bar{k}} u)^\dagger = g^{j \bar{k}}
\nabla_{\bar{k}} \nabla_j u$, we have
\[
(\LL u)^\dagger = -(g_{{\rm FLY},t})^{j \bar{k}} \nabla_j^{H_t} \partial_{\bar{k}}
u +{1 \over 2} [ \sqrt{-1} \Lambda_{\omega_{{\rm FLY},t}} F_{H_t},u].
\]
The commutator identity for $[\nabla_j,\nabla_{\bar{k}}]$ now shows that $(\LL
u)^\dagger = \LL u$. Next, since $\omega_{{\rm FLY},t}$ is balanced, we do
have
\[
\int_{X_t} (\Delta_{g_{{\rm FLY},t}} {\rm
  Tr} \, u) \, (\omega_{{\rm FLY},t})^3 = 0.
\]
This verifies that $\LL: W_t \rightarrow W_t$ preserves the subspace
$W_t$.
\smallskip
\par Having obtained the estimate $\| h \| \leq C \| L h \|$ from
Proposition \ref{prop-invert-L}, we can invert $\LL$ with a bound on the inverse.

\begin{lem} \label{lem-invert-L}
 Let $\LL: C^{2,\alpha}_\beta(W_t) \rightarrow
 C^{0,\alpha}_{\beta-2}(W_t)$ for $\alpha \in (0,1)$ and $\beta \in (-2,0)$. There exists $|t_0|>0$ such that for all $0<|t| \leq |t_0|$, then $\LL$ is invertible and the operator norm
\be \label{inverse-est}
\| \LL^{-1} \| \leq C
\ee
is bounded independent of $t$.
\end{lem}

\begin{proof} We start by discussing the uniform estimate. Let $L$ be
the operator that was estimated in Proposition \ref{prop-invert-L}. For $u \in
C^{2,\alpha}_\beta(W_t)$, we have
\bea
\|  u \|_{C^{2,\alpha}_\beta} 
&\leq& C \| L u
\|_{C^{0,\alpha}_{\beta-2}} \nonumber\\ 
&\leq&C \| ((g_{\rm FLY})^{j \bar{k}} - g_t^{j \bar{k}})
\partial_{\bar{k}} \nabla_j^{H_t} u \|_{C^{0,\alpha}_{\beta-2}} 
\nonumber\\
&&+ C \|
g_{\rm FLY}^{-1} - g_t^{-1} \|_{C^{0,\alpha}_0} \|
F_{H_t}\|_{C^{0,\alpha}_{-2}} \| u \|_{C^{0,\alpha}_{\beta}} +C \| \LL u \|_{C^{0,\alpha}_{\beta-2}} \nonumber\\
&\leq& C |t|^{2/3} \|  u \|_{C^{2,\alpha}_\beta} + C |t|^{2/3} \|
F_{H_t} \|_{C^{0,\alpha}_{-2}} \| u \|_{C^{0,\alpha}_{\beta}}  +C \|
\LL u \|_{C^{0,\alpha}_{\beta-2}}  \nonumber
\eea
by Lemma \ref{lem:gfly-gt}. By Lemma \ref{g_t-unif-H_t}, we have $r^2 |F_{H_t}|_{g_t} + r^3 |\nabla_{g_t} F_{H_t}|_{g_t} \leq C $ and hence
\be \label{eq: F_H_t-bound}
\| F_{H_t} \|_{C^{0,\alpha}_{-2}} \leq C
\ee
uniformly in $t$. We
conclude that for $t$
small enough, then
\be \label{inverse-est0}
\|  u \|_{C^{2,\alpha}_\beta} \leq C \| \LL u
\|_{C^{0,\alpha}_{\beta-2}},  \quad u \in W_t.
\ee
The proof of the lemma now follows from standard elliptic PDE theory on
the smooth compact manifold $X_t$. We will use the space of
sections denoted by
\[
\HH = \bigg\{ u \in \Gamma({\rm End} \, T^{1,0}X_t) \ : \ u^\dagger
= u \bigg\}
\]
with $L^2$ inner product
\[
\langle s,h \rangle_{L^2} = \int_X \langle s,h \rangle_{H_t} \, d {\rm vol}_{g_{{\rm FLY},t}}
, \quad s,h \in \HH.
\]
Then we can orthogonally decompose $\HH = W \oplus \mathbb{C} I $. We consider $P: C^{2,\alpha}(\HH) \rightarrow C^{0,\alpha}(\HH)$ with
$Pu = \LL u$. From
(\ref{inverse-est0}), we see that $\mathbb{C} I = \ker
P$. Furthermore, by the balanced condition of $\omega_{{\rm FLY},t}$
we see that
\[
{\rm im} \, P \, \subseteq W.
\]
Therefore $(\ker P^\dagger)^\perp \subseteq W$ and so $\mathbb{C}
I \subseteq \ker P^\dagger$. We will show $\mathbb{C} I = \ker P^\dagger$.
\smallskip
\par An integration by parts argument using the balanced property
shows that the operator $\Delta = (g_{{\rm FLY},t})^{j \bar{k}} \partial_{\bar{k}}
\nabla_j^{H_t}$ is $L^2$ self-adjoint and so has degree zero. Thus $P$ also has degree
zero. Therefore $\dim \ker P^\dagger = 1$ and  $\mathbb{C}
 I  = \ker P^\dagger$. It follows that ${\rm im} \, P =
W$.  Hence $\LL: W \rightarrow W$ is invertible, and the bound on the
inverse in weighted norms is (\ref{inverse-est0}).
\end{proof}

\subsection{Fixed point theorem}
In this section, we perturb the glued metric $H_t$ to a solution of the Hermitian-Yang-Mills equation $\sqrt{-1} \Lambda_{\omega_{{\rm FLY},t}} F_H = 0$ on the smoothing $X_t$ via a fixed point theorem. The general approach of constructing an approximate solution and deforming it to a true solution to solve equations in differential geometry goes back to Taubes \cite{Taubes} and is now widely used; in this section we will follow the notation used in \cite{SzeBook}.

\smallskip
\par Our space of deformations will be the space $W_t$ defined in
(\ref{eq:W_t}). We introduce the operator
\[
\F: W_t \rightarrow W_t
\]
given
\be \label{nonlinearF-defn}
\F(u) = e^{u/2} (\sqrt{-1} \Lambda_{\omega_{{\rm
      FLY},t}}  F_u) e^{-u/2}, 
\ee
where $F_u$ is the curvature of the metric $H_{t,u} = H_t
e^u$. 

\smallskip
\par We note that $\F: W_t \rightarrow W_t$. Indeed, for $u \in W_t$, we have $\int {\rm Tr} \, \F(u) = 0$ since
\[
\int_{X_t} c_1(T^{1,0}X_t) \wedge \omega_{{\rm FLY},t}^2 = 0
\]
as $X_t$ has trivial canonical bundle. The adjoint action of $e^{u/2}$ in
(\ref{nonlinearF-defn}) is added to ensure that $\F(u)^\dagger =
\F(u)$. Indeed, $\sqrt{-1}
\Lambda F_{H_{t,u}}$ is self-adjoint with respect to the metric $H_{t,u}$. In coordinates where $H_t = I$ and $e^u$ is
diagonal, it is straight-forward to check that 
\[
(e^{u/2} (\sqrt{-1} \Lambda_{\omega_{{\rm
      FLY},t}}  F_u) e^{-u/2})^\dagger = e^{u/2} (\sqrt{-1} \Lambda_{\omega_{{\rm
      FLY},t}}  F_u) e^{-u/2}
\]
where $\dagger$ is the adjoint with respect to $H_t$.
\smallskip
\par A well-known computation gives the linearization of $F_u$.
\[
\delta (F_u) = \bar{\partial} \partial_{H_{t,u}} (e^{-u} \delta e^u).
\]
The linearization of $\F$ is then
\bea \label{linearized-eqn}
(\delta \F)|_u (\delta u) &=& - e^{u/2} [ (g_{{\rm FLY},t})^{j \bar{k}}
\partial_{\bar{k}} \nabla^{H_{t,u}}_j (e^{-u} \delta e^u)] e^{-u/2} \nonumber\\
&&+ (\delta e^{u/2}) (\sqrt{-1} \Lambda_{\omega_{{\rm
      FLY},t}} F_u) e^{-u/2} \nonumber\\
&&+ ( e^{u/2}) (\sqrt{-1} \Lambda_{\omega_{{\rm
      FLY},t}} F_u) \delta e^{-u/2}. \nonumber
\eea
Let $\LL$ be the linearization $(\delta F)|_0$ at $u=0$. That is, $\LL: W_t \rightarrow W_t$
with
\[
  \LL w = -(g_{FLY,t})^{j \bar{k}} \partial_{\bar{k}} \nabla_j^{H_t} w - {1
    \over 2} [\sqrt{-1} \Lambda_{\omega_{FLY,t}} F_{H_t},w].
\]
We previously inverted $\LL$ and gave a bound on the inverse $\LL^{-1}$ uniform
in $t$. We can write
\[
\F(u) = \F(0) + \LL(u) + \Q(u)
\]
where by definition
\be \label{Q-defn}
\Q(u) = \F(u) - \F(0) - \LL(u).
\ee

We define
\[
\mathcal{N} : C^{2,\alpha}_\beta(W) \rightarrow  C^{2,\alpha}_\beta(W)
\]
given by
\[
\mathcal{N}(u) = \LL^{-1}(-\F(0) - \Q(u)).
\]
To solve $\F(u)= 0$, it is equivalent to find a fixed point
\[\mathcal{N}(u)=u.\] 
\begin{prop} \label{prop: contract}
Let $a \in (0,1)$. There exists $\epsilon>0$ and $\beta \in (-1,0)$ with the following property. Let 
\[
\mathcal{U}_t = \{ u \in  C^{2,a}_\beta(W(X_t)) : \| u \|_{C^{2,a}_\beta}
<  |t|^{(2/3)|\beta|} \}.
\]
Then for all $0<|t|<\epsilon$, the mapping $\mathcal{N}$ preserves $\mathcal{U}_t$
and satisfies
\be \label{contract-map}
\| \mathcal{N}(u) - \mathcal{N}(v) \|_{C^{2,a}_\beta} \leq {1 \over 2} \| u - v \|_{C^{2,a}_\beta}
\ee
for all $u,v \in \mathcal{U}_t$.
\end{prop}

\begin{proof}
We start by assuming (\ref{contract-map}) and prove that $\mathcal{N}$ preserves
$\mathcal{U}_t$. We have
\[
\F(0) = \sqrt{-1} \Lambda_{\omega_{{\rm FLY},t}} F_{H_t}
\]
and we can estimate
\[
\| \F(0) \|_{C^{0,a}_{-2}} \leq \| \Lambda_{\omega_t}
F_{H_t} \|_{C^{0,a}_{-2}} + \| \omega_t^{-1} - \omega_{{\rm
      FLY},t}^{-1} \|_{C^{0,a}_{0}} \| F_{H_t} \|_{C^{0,a}_{-2}}.
\]
Proposition \ref{prop-small-hym}, Lemma
\ref{lem:gfly-gt} and (\ref{eq: F_H_t-bound}) imply
\[
\| \F(0) \|_{C^{0,a}_{-2}} \leq C (|t|^{2/3} +
|t|^{|\alpha \lambda|/3}).
\]
The contribution $|t|^{|\alpha \lambda|/3}$ is the slowest rate. For any $\beta \in (-1,0)$, we therefore have
\[
\| \F(0) \|_{C^{0,a}_{\beta-2}} \leq C |t|^{|\alpha \lambda|/3}.
\]
Since $\| \mathcal{L}^{-1} \| \leq C$ by (\ref{inverse-est}), it follows that
\[
\| \mathcal{N}(0) \|_{C^{2,a}_\beta} \leq C |t|^{|\alpha \lambda|/3},
\]
and hence (\ref{contract-map}) implies that for $u \in \mathcal{U}_t$ then
\[
\| \mathcal{N}(u) \| \leq \| \mathcal{N}(u)- \mathcal{N}(0) \| + \|
\mathcal{N}(0) \| \leq {1 \over 2} |t|^{(2/3)|\beta|} + C
|t|^{|\alpha \lambda| /3} < |t|^{(2/3)|\beta|}
\]
for $\beta = - {1 \over 4} |\alpha \lambda|$ and $t$ small enough. 
\smallskip
\par Thus it remains to
prove (\ref{contract-map}). By definition
\[
\mathcal{N}(u) - \mathcal{N}(v) = \mathcal{L}^{-1} ( \mathcal{Q}(v) - \mathcal{Q}(u)).
\]
Since $\| \mathcal{L}^{-1} \| \leq C$ by (\ref{inverse-est}), we have
\be \label{diff-N}
\| \mathcal{N}(u) - \mathcal{N}(v) \|_{C^{2,a}_\beta} \leq C \|
\mathcal{Q}(u) - \mathcal{Q}(v) \|_{C^{0,a}_{\beta-2}}.
\ee
So we need to estimate $\mathcal{Q}(u) - \mathcal{Q}(v)$. One way to
write this is
\[
\mathcal{Q}(u) - \mathcal{Q}(v) = \int_0^{1} {d \over ds}
\mathcal{Q}(w_s) \, ds
\]
where $w_s = s u + (1-s) v$. By the definition (\ref{Q-defn}) of
$\mathcal{Q}$, its variation is
\[
{d \over ds} \mathcal{Q}(w_s) = (\delta \F)|_{w_s} (u-v) - \LL (u-v).
\]
We claim that the approximate linearized operator $\LL$ is close enough
to the actual linearization $\delta \F$, so that for all $w,s \in
\mathcal{U}$ then
\be \label{diff-approxL-actualL}
\| (\delta \F)|_w(s) - \LL(s) \|_{C^{0,a}_{\beta-2}} \leq C \| w \|_{C^{2,a}_0} \|
s \|_{C^{2,a}_\beta}.
\ee
Assuming this, we conclude
\[
\|
\mathcal{Q}(u) - \mathcal{Q}(v) \|_{C^{0,a}_{\beta-2}} \leq C (\| u \|_{C^{2,a}_0}+\| v \|_{C^{2,a}_0} )
\| u -v \|_{C^{2,a}_\beta}.
\]
By (\ref{diff-N}), we see that 
\[
\| \mathcal{N}(u) - \mathcal{N}(v) \|_{C^{2,a}_\beta} \leq C (\| u \|_{C^{2,a}_0}+\| v \|_{C^{2,a}_0} )
\| u -v \|_{C^{2,a}_\beta}.
\]
If $u \in \mathcal{U}$, then $\| u \|_{C^{2,a}_\beta} < 
|t|^{{2 \over 3}|\beta|}$, and since $r^3\geq |t|$, we have
\[
\| u \|_{C^{2,a}_0} \leq {1 \over \min_{X_t} r^{|\beta|}} \| u \|_{C^{2,a}_\beta} \leq {1
  \over |t|^{|\beta|/3}} |t|^{(2/3)|\beta|} \leq \epsilon^{|\beta|/3}.
 \]
Thus if $\epsilon$ is small enough, then $\mathcal{N}$ is a contraction map. The proof is complete given (\ref{diff-approxL-actualL}), which is proved in the following lemma.
\end{proof}

\begin{lem}
There exists $C>1$ such that for all $t \in \mathbb{C}^*$ and $u,s \in W_t$ with $\| u \|_{C^{2,a}_0(X_t)} \leq 1$, we can estimate
\[
\| (\delta \F)|_u(s)- \LL(s) \|_{C^{0,a}_{\beta-2}(X_t)} \leq C \| u \|_{C^{2,a}_0(X_t)}  \|
s \|_{C^{2,a}_\beta(X_t)}.
\]
\end{lem}
\begin{proof}
To simplify notation, for this estimate we will write
$g=g_{{\rm FLY},t}$, $H_0=H_t$ and $H_u=H_{t,u}$. The linearization was computed in (\ref{linearized-eqn}),
and the difference is explicitly given by
\bea
& \ & (\delta \F)|_u(s)- \LL(s) \nonumber\\
&=& \bigg[ [\delta \exp]|_u(s/2) \bigg](\sqrt{-1} \Lambda_\omega F_{H_u}) e^{-u/2} - {1 \over 2} s (\sqrt{-1}
\Lambda_{\omega} F_{H_0})\nonumber\\
&&+ ( e^{u/2}) (\sqrt{-1} \Lambda_\omega F_{H_u}) \bigg[ [\delta
\exp]|_u (-s/2) \bigg] + {1 \over 2} (\sqrt{-1}
\Lambda_{\omega} F_{H_0} )s \nonumber\\
&& - e^{u/2} [ g^{j \bar{k}}
\partial_{\bar{k}} \nabla^{H_{u}}_j e^{-u} [\delta \exp]|_u (s)] e^{-u/2} + g^{j \bar{k}}
\partial_{\bar{k}} \nabla_j^{H_0} s\nonumber\\
&=& ({\rm I}) + ({\rm II}) + ({\rm III}), \nonumber
\eea
where $ ({\rm I}), ({\rm II}),({\rm III})$ denotes the terms on each line and the derivative of the matrix exponential is given by the formula
\[
[\delta \exp]|_u(\delta u) = \int_0^1 e^{\lambda u} (\delta u)
e^{(1-\lambda)u} d \lambda.
\]
 \par $\bullet$ Terms $({\rm I})$ and $({\rm II})$. These two terms are
estimated in the same way, so we only give the estimate for $({\rm I})$. We start by writing
\bea
({\rm I}) &=& [\delta \exp]|_u (s/2) (\sqrt{-1} \Lambda_{\omega} F_{H_u}) e^{-u/2} - {1 \over 2} s (\sqrt{-1}
\Lambda_{\omega} F_{H_0}) \nonumber\\
&=&\int_0^1  {d \over d \gamma}  \bigg[ [\delta \exp]|_{\gamma u} (s/2)(\sqrt{-1} \Lambda_{\omega} F_{H_{\gamma u}}) e^{-(\gamma/2) u } \bigg] d \gamma.
\eea
There are 3 terms depending on where ${d \over d
  \gamma}$ lands. The first is
\[
({\rm Ia})_\gamma =  \bigg[ {d \over d \gamma}  [\delta
\exp]|_{\gamma u} (s/2)
\bigg] (\sqrt{-1} \Lambda_{\omega} F_{H_{\gamma u}}) e^{-(\gamma/2) u }.
\]
This can be estimated as
\[
\| ({\rm I a})_\gamma \|_{C^{0,a}_{\beta-2}} \leq C \| F_{H_{\gamma u}} \|_{C^{0,a}_{-2}} \| u
\|_{C^{0,a}_0} \| s \|_{C^{0,a}_\beta}
\]
Next, we consider
\[
({\rm Ib})_\gamma =  [\delta \exp]|_{\gamma u} (s/2)
\bigg[ {d \over d \gamma}   (\sqrt{-1} \Lambda_{\omega} F_{H_{\gamma u}}) \bigg]  e^{-(\gamma/2) u }.
\]
Since ${d \over d\gamma} F_{H_{\gamma u}} = \bar{\partial} \partial_{H_{\gamma
    u}} u$, we have
\[
\| ({\rm I b})_\gamma \|_{C^{0,a}_{\beta-2}} \leq C \|
\bar{\partial} \nabla_{H_{\gamma u}} u \|_{C^{0,a}_{-2}} \| s \|_{C^{0,a}_\beta}.
\]
The other term is
\[
({\rm Ic})_\gamma =  [\delta \exp]|_{\gamma u} (s/2)
(\sqrt{-1} \Lambda_{\omega} F_{H_{\gamma u}}) \bigg[ {d \over d
  \gamma}    e^{-(\gamma/2) u } \bigg]
\]
and can be estimated by
\[
\| ({\rm Ic})_\gamma \|_{C^{0,a}_{\beta-2}} \leq C \| F_{H_{\gamma
    u}} \|_{C^{0,a}_{-2}} \| u \|_{C^{0,a}_0} \| s \|_{C^{0,a}_\beta}.
\]
Altogether, using $\| u \|_{C^{2,a}_0} \leq 1$, the formula for the difference of connections $A_{H_{\gamma u}} - A_{H_0} = e^{-\gamma u} \nabla_{H_0} e^{\gamma u}$ which gives a formula for $F_{H_{\gamma u}}$, and $\| F_{H_0} \|_{C^{0,a}_{-2}} \leq C$, we have
\[
\| ({\rm I}) \|_{C^{0,a}_{\beta-2}} \leq C \| u \|_{C^{2,a}_0} \| s \|_{C^{0,a}_\beta}
\]
$\bullet$ Term $({\rm III})$. We write
\[
({\rm III}) = -\int_0^1 {d \over d \gamma} \bigg[ e^{(\gamma /2)u}[g^{j \bar{k}}
\partial_{\bar{k}} \nabla_j^{H_{\gamma u}} e^{-\gamma u} [\delta
 \exp]_{\gamma u} (s)]
e^{-(\gamma /2)u} \bigg] d \gamma.
\]
Using $\| u \|_{C_0^{2,a}} \leq 1$, from here we can derive the estimate
\[
\| ({\rm III}) \|_{C^{0,a}_{\beta-2}} \leq C \| u \|_{C^{2,a}_0} \| s \|_{C^{2,a}_\beta}.
  \]
  To do this, the covariant derivative $\nabla_{H_{\gamma u}}$ can be converted to $\nabla_{H_0}$ via $\nabla_{H_{\gamma u}} - \nabla_{H_0} = e^{-\gamma u} \nabla_{H_0} e^{\gamma u}$ and we can use the variational formula ${d \over d \gamma} \nabla_{H_{\gamma u}} = \nabla_{H_{\gamma u}} u$.
  \end{proof}
\medskip
\par By Proposition \ref{prop: contract} and the Banach fixed point theorem, there exists $\check{u} \in
C^{2,a}_\beta(W(X_t))$ with $\| \check{u} \|_{C^{2,a}_\beta} < |t|^{(2/3)
  \beta}$ such that $\mathcal{N}(\check{u})=\check{u}$, meaning that
$\Lambda_{\omega_{{\rm FLY},t}} F_{\check{u}} = 0$
where $F_u$ is the curvature of $H_t e^u$. This proves the existence of a pair solving
\[
    d \omega_{{\rm FLY},t}^2 = 0, \quad F_{\check{H}_t} \wedge
    \omega_{{\rm FLY},t}^2=0
  \]
on $X_t$. To complete the proof of the main theorem, we describe the behavior of $(g_{{\rm FLY},t},\check{H}_t)$ near the vanishing cycles.
\bigskip
\par {\it Proof of Theorem \ref{thm: mainTheorem}:} The local estimate (\ref{eq:fly-local}) follows from Lemma \ref{lem:gfly-gt} since $g_t=c_i
g_{co,t}$ in a fixed neighborhood of the vanishing cycles. The
Hermitian-Yang-Mills metric is given by $\check{H}_t = H_t
e^{\check{u}}$ where $\check{u}$ is the fixed-point solving
$\mathcal{N}(\check{u})=\check{u}$ and satisfying $\| \check{u} \|_{C ^{2,a}_\beta} \leq
|t|^{(2/3)|\beta|}$. Near the nodes, by the gluing construction (see (\ref{H_t-defn})) we have $\check{H}_t = d_i
g_{co,t} e^{\check{u}}$ in the region
\[
\mathcal{R}_{\lambda} = \{ |t| \leq \|z\|^2 \leq |t|^{\frac{3}{3+\lambda}}\},
\]
 which implies
\[
| \check{H}_t - d_i g_{co,t} |_{g_{co,t}} \leq Cr^{-|\beta|}
|t|^{(2/3)|\beta|} \leq C |t|^{|\beta|/3}
\]
since $r^3 \geq |t|$. We obtain similar estimates for
$\nabla^k \check{H}_t$ for $k=1$ and $k=2$. For the higher order
estimates, we write the equation
$g^{j \bar{k}} (F_{\check{H}})_{j \bar{k}}=0$ in holomorphic
cylindrical coordinates as
\[
r^2 (g_{{\rm FLY},t})^{j \bar{k}} \partial_j \partial_{\bar{k}}
\check{H}_t= r^2  (g_{{\rm FLY},t})^{j \bar{k}} \partial_j \check{H}_t
\check{H}_t^{-1} \partial_{\bar{k}} \check{H}_t.
\]
Note that since $g_t = r^2 O(I)$ (see Lemma \ref{lem: cylCoords}), by (\ref{eq:fly-local}) we also have $g_{{\rm FLY},t} =
r^2O(1)$ for small $t$. Schauder estimates and a bootstrap argument
imply
\[
|\partial^\ell \check{H}_{\bar{k} j}|_{g_{euc}} \leq C_\ell |t|^{|\beta|/3}.
\]
Converting $g_{euc}$ in holomorphic cylindrical coordinates to
$g_{co,t}$ gives the stated estimate with $\lambda = |\beta|/3$. $\qed$

\bigskip

\begin{rk}\label{rk: confRescale}
To simultaneously solve the Hermitian-Yang-Mills and the conformally balanced equations
\[
    d(\| \Omega_t \|_{\check{\omega}_t} \, \check{\omega}_t^2) = 0, \quad F_{\check{H}_t} \wedge \check{\omega}_{t}^2=0,
  \]
  we can set
  \[
\check{\omega}_t = \| \Omega_t \|^{-2}_ {\omega_{{\rm FLY},t}} \omega_{{\rm FLY},t}
\]
so that $\| \Omega_t \|_{\check{\omega}_t} \, \check{\omega}_t^2=
\omega_{{\rm FLY},t}^2$. Here $\Omega_t$ is a holomorphic volume form on $X_t$, whose existence through conifold transitions is proved in \cite{Fried1}, normalized by $\int_{X_t} \sqrt{-1} \Omega_t \wedge \overline{\Omega}_t = 1$. It is straightforward to show that this conformally balanced metric $\check{g}_{t}$ associated to
$\check{\omega}_t$ satisfies a decay estimate near the vanishing
cycles.  Namely, near any ODP $p_i \in \underline{X}$ there is a constant $c_i$ such that we have the estimates
\[
|\nabla^{k}_{g_{FLY,t}} (\check{g}_{t} - c_ig_{{\rm FLY},t})|_{g_{{\rm FLY},t}} \leq C_k r^{\frac{3}{2}-k},
\]
for all $k \in \mathbb{Z}_{\geq 0}$.  Note that, unlike Lemma \ref{lem:gfly-gt}, one can no longer expect decay in $|t|$ for fixed $r$.  This is due to the the fact that, near the ODP singularities, the holomorphic volume form $\Omega_{t}$ will not necessarily converge to a constant multiple of the natural equivariant holomorphic volume form on the conifold.  Nevertheless, these estimates still imply that there is a constant $d_i>0$ so that at a suitable scale the pair $(\check{g}_t, \check{H}_{t})$ converges smoothly to the pair $(g_{co,0},  d_{i}g_{co,0})$  as $|t|\rightarrow 0$.  In particular, at suitably small scales, the pair $(\check{g}_t, \check{H}_{t})$ converges to a solution of the Strominger system near the ODP singularities as $t\rightarrow 0$.
\end{rk}

\appendix

\section{The Fu-Li-Yau Gluing Construction}\label{sec: FLYapp}
In this appendix we will explain the gluing result of Fu-Li-Yau, which establishes Propositions~\ref{prop: FLY2.6} and~\ref{prop: FLY2.6-resolution}.  Before beginning, we recall the notation.  We are primarily interested in the small resolution of the conifold, given by $p: \mathcal{O}_{\mathbb{P}^1}(-1)^{\oplus 2} \rightarrow \mathbb{P}^1$.  Let $h_{FS}$ denote the Fubini-Study metric on $\mathcal{O}_{\mathbb{P}^1}(-1)$.  Let
\[
\pi:\mathcal{O}_{\mathbb{P}^1}(-1)^{\oplus 2}\rightarrow V_0
\]
be the map contracting $\mathbb{P}^1$. The pull-back radial function of the conical Calabi-Yau metric on $V_0$ is given by
\[
r^2 = \left(|u|^2_{h_{fs}} + |v|^2_{h_{fs}}\right)^{2/3}
\]
where $(u,v)$ are fiber coordinates on $\mathcal{O}_{\mathbb{P}^1}(-1)^{\oplus 2}$.  The holomorphic Reeb field on $V_0$ induces a holomorphic $\mathbb{C}^*$ action given by
\[
S_{\lambda}(x,u,v)= (x, \lambda^{3/2}u, \lambda^{3/2}v).
\]
where $x\in \mathbb{P}^1$.  Unless otherwise specified, we will consider the restriction to $\lambda \in \mathbb{R}_{>0} \subset \mathbb{C}^*$.  Note that we have
\[
S_{\lambda}^*r^2= \lambda^2r^2.
\]
We define a smooth K\"ahler metric on $ \mathcal{O}(-1)^{\oplus 2}$ by
\[
\omega_{sm} = \ddb r^3 + p^*\omega_{FS}.
\]
We will consider also the cone metric
\[
\omega_{co,0} =\frac{3}{2} \ddb r^2.
\]
It will be important for us to compare $\omega_{co,0}^2$ and $\omega_{sm}^2$.  

\begin{lem}\label{lem: wsmwcoComp}
There is a uniform constant $C>0$  such that following estimates hold
\begin{itemize}
\item[$(i)$]If $0<r <1$, then we have
\[
r^2C^{-1}\omega_{co,0}^2 \leq \omega_{sm}^2 \leq C r^{-1}\omega_{co,0}^2
\]
\item[$(ii)$] If $0<r<1$ then
\[
C^{-1}r^{2}\omega_{co,0}^2 \leq \omega_{sm} \wedge \omega_{co,0} \leq r^{-1}\omega_{co,0}^2
\]

\end{itemize}
\end{lem}
\begin{proof}
First, since $\omega_{co,0}$ and $\omega_{sm}$ define smooth K\"ahler metrics on the compact set  $\{1 \leq r \leq 2 \}$ we can fix a constant $A$ such that
\[
A^{-1}\omega_{co,0} \leq \omega_{sm} \leq A\omega_{co,0}.
\]
Consider
\[
S_{\lambda}:  \{1 \leq r \leq 2 \}\rightarrow  \{ \lambda \leq r \leq 2\lambda\}.
\]
From the homogeneity of $r$ we have 
\[
S_{\lambda}^*\omega_{co,0}^2 = \lambda^{4}\omega_{co,0}^2.
\]
On the other hand,
\[
S_{{\lambda}}^* \omega_{sm} = \lambda^{3}\ddb r^3 + p^*\omega_{FS},
\]
and so
\[
S_{{\lambda}}^* \omega_{sm}^2 = \lambda^{6}(\ddb r^3)^2 +2\lambda^{3}\ddb r^3\wedge p^*\omega_{FS}
\]
using that $p^*\omega_{FS}^2=0$.  If $\lambda \leq 1$ then we get
\[
\lambda^{6} \omega_{sm}^2 \leq S_{{\lambda}}^* \omega_{sm}^2  \leq \lambda^{3} \omega_{sm}^2 
 \]
 From the definition of $A$ we have
\[
 S_{{\lambda}}^* \omega_{sm}^2  \leq \lambda^{3} \omega_{sm}^2 \leq \lambda^{3}A^2 \omega_{co,0}^2 \leq \lambda^{-1}A^2S_{\lambda}^*\omega_{co,0}^2
 \]
 and similarly
 \[
 S_{{\lambda}}^* \omega_{sm}^2  \geq \lambda^{6} \omega_{sm}^2 \geq \lambda^{2}A^{-2}S_{\lambda}^*\omega_{co,0}^2.
 \]
 Since $\lambda \leq r \leq 2\lambda$, proves $(i)$.  The proof of $(ii)$ is similar.  

\end{proof}

Consider the four form
\[
\Omega:=\ddb (\chi(r^2)\ddb r^2)
\]
where $\chi(\cdot)$ is some smooth function to be determined.  We compute
\[
\Omega = \chi''(r^2)\sqrt{-1}\del r^2\wedge \dbar r^2 \wedge \ddb r^2+ \chi'(r^2)\omega_{co,0}^2.
\]
To understand the first term write
\[
\ddb r^2= 4\sqrt{-1}\del r \wedge \dbar r+r^2\ddb \log r^2.
\]
from which it follows that 
\[
(\ddb r^2)^2 = 8r^2\sqrt{-1}\del r \wedge \dbar r \wedge \ddb \log r^2 + (r^2\ddb \log r^2)^2
\]
and also
\[
\sqrt{-1}\del r \wedge \dbar r \wedge \ddb r^2 =r^2\sqrt{-1}\del r \wedge \dbar r \wedge \ddb \log r^2.
\]
Since $\ddb \log r^2 \geq 0$ we conclude

\begin{lem}\label{lem: OmegaRbnd}
The following estimate holds everywhere on $\mathcal{O}_{\mathbb{P}^1}(-1)^{\oplus 2}$
\[
\sqrt{-1}\del r^2\wedge \dbar r^2 \wedge \ddb r^2 \leq \frac{r^2}{2}(\ddb r^2)^2.
\]
In particular, whenever $\chi'' <0$ we have the lower bound
\[
\Omega \geq \left(\chi'(r) + \frac{r^2}{2}\chi''(r^2)\right) \omega_{co,0}^2
\]
\end{lem}

Fix $R \gg 1$.  Our goal is to find $R \gg 1$, and a constant $C_{R}>0$ such that we can glue
\[
\Omega_{R}:= C_{R}^2S_{R}^*\ddb(\chi(r^2)\ddb r^2)
\]
to the Calabi-Yau metric $\omega_{CY}^2$.  This will require carefully choosing $\chi$, and the constant $C_{R}$.  We are going to assume that
\[
\chi(s)= s
\]
for $s\in [0,4]$, so that $\Omega_{R}$ agrees with a rescaling of $\omega_{co,0}^2$ on $\{0\leq r \leq \frac{2}{R}\}$.  More precisely, we have
\[
C_{R}^2S_{R}^*\ddb(\chi(r^2)\ddb r^2) = C_{R}^2 R^{4} \omega_{co,0}^2 \qquad \{0<r<2R^{-1}\}.
\]
In the remainder of the appendix we will determine conditions on $\chi, C_{R}, R$ for this to be possible.

We now consider Calabi-Yau metric $\omega_{CY}$.  Consider the set $U=\{ r<4\} \subset X$.  Since $U$ is contractible onto $\mathbb{P}^1$, we can write
\[
\omega_{CY} = \lambda p^*\omega_{FS} + \del\beta + \dbar \overline{\beta}
\]
for some $(1,0)$ form $\beta$ on $U$, and some $\lambda >0$.  To simplify notation, let us assume $\lambda=1$. By solving the $\dbar$-equation we can write
\[
\omega_{CY} = p^*\omega+\ddb\phi
\]
where $\omega$ is a K\"ahler form on $\mathbb{P}^1$, with $[\omega]= [\omega_{FS}] \in H^{1,1}(\mathbb{P}^1,\mathbb{R})$ and $\phi:U\rightarrow \mathbb{R}$ a smooth function with $\phi|_{\mathbb{P}^1}=0$; see \cite[Lemma 2.4]{FLY}.  Let $h_1$ denote the degree $\frac{3}{2}$ part of $\phi$ under $S_{\lambda}$ (recall that $S_{\lambda}$ corresponds to scaling with weight $\frac{3}{2}$ along the fibers), so that $|\phi-h_{1}| \sim O(|u|^2+|v|^2)$, or in other words
\[
|\phi-h_1|\leq Cr^3.
\]

Let $\sigma(x)$ be a positive cut-off function with $0\leq \sigma(x) \leq 1$ and $\sigma(x)=0$ for $x\in[0,1]$ and $\sigma(x)=0$ for $x\geq 8$.  Define
\[
\begin{aligned}
\Psi_{R} &= \omega_{CY}^2 - \ddb \Gamma_{R}\\
\Gamma_{R}&=\ddb \left( \sigma(R^3r^3) \left((\phi-h_1)(2p^*\omega + \ddb( \phi + h_1)) + h_1\ddb h_1\right)\right)
\end{aligned}
\]
where, as before, $R \gg 1$ is a parameter to be determined.  From the definition of $\sigma$ we have
\[
\Psi_{R} = \omega_{CY}^2 \text{ if } \{r> 2R^{-1}\}
\]
We claim that
\[
\Psi_{R} =0  \text{ if } \{r< R^{-1}\}.
\]
To see this, note that if $x,y,w$ are commutative variables satisfying $yw=0$, then
\[
(x-y)(2w+(x+y)) = x(2w+x) + xy - 2wy-xy-y^2 = x(2w+x)-y^2.
\]
We apply this formula with $x =\ddb \phi, y= \ddb h_1, w= p^*\omega$ and product being the wedge product.  We only need to check that $\ddb h_1 \wedge p^*\omega_{FS} = \ddb h_1^2=0$, but this is clear since $\ddb h_1$ is linear along the fibers of $\mathcal{O}_{\mathbb{P}^1}(-1)^{\oplus 2}\rightarrow \mathbb{P}^1$. 

The main task is to find a lower bound for $\Psi_{R}$ in terms of $\omega_{co,0}^2$ in the transition region $\{ R^{-1}<r< 2R^{-1}\}$.  To do this we expand  
\[
\Psi_{R} = (1-\sigma)\omega_{CY}^2 - {\rm (I)} - {\rm (II)}- {\rm (III)}
\]
where
\[
\begin{aligned}
{\rm(I)} &= 2{\rm Re}\left((\sqrt{-1} \sigma'(R^{3}r^3)R^3\del r^3 \wedge \dbar(\phi-h_1)\wedge (2p^*\omega + \ddb( \phi + h_1))\right)\\
&+2{\rm Re}\left( \sigma'(R^{3}r^3) R^3\del r^3 \wedge \dbar h_1 \wedge \ddb h_1\right)\\
{\rm (II)} &= \sigma''(L^3r^3) L^6\sqrt{-1}\del r^3\wedge \dbar r^3 \wedge \left((\phi-h_1)(2p^*\omega + \ddb( \phi + h_1)) + h_1\ddb h_1\right)\\
{\rm (III)} &= \sigma'(L^3r^3) L^3\ddb r^3 \wedge \left((\phi-h_1)(2p^*\omega + \ddb( \phi + h_1)) + h_1\ddb h_1\right).
\end{aligned}
\]

Our goal is to estimate each term from below by $\omega_{co,0}^2$.  Each term will be treated differently, depending on whether it is homogeneous or not.  

\par $\bullet$ Term ${\rm (I)}$.  Observe that
\[
\del r^3 \wedge \dbar(\phi-h_1) \sim r^3\omega_{sm}
\]
To see this recall that, in coordinates we have
\[
r^3 = |u|_{h_{FS}}^2+|v|^2_{h_{FS}}
\]
so that, in coordinates where $\del h_{FS}=0$, we have
\[
\del r^3 = \bar{u}du + \bar{v}dv.
\]
On the other hand, since $h_1$ is linear along the fibers of $p: \mathcal{O}_{\mathbb{P}^1}(-1)^{\oplus 2}\rightarrow \mathbb{P}^1$ we have
\[
\dbar (\phi-h_1)= O(u,v)(d\bar{u} + d\bar{v}) 
\]
and so
\[
\del r^3 \wedge \dbar (\phi-h_1)\leq C r^3 \omega_{sm}.
\]
Thus, by Lemma~\ref{lem: wsmwcoComp}, the first term in ${\rm (I)}$ can be controlled by $r^3\omega_{sm}^2 \leq r^2\omega_{co,0}^2$.

To analyze the second term in ${\rm (I)}$ we observe that
\[
S_{\lambda}^*(\del r^3 \wedge \dbar h_1 \wedge \ddb h_1 )= \lambda^6\del r^3 \wedge \dbar h_1 \wedge \ddb h_1,
\]
which, by the homogeneity of $\omega_{co,0}$, implies
\[
\del r^3 \wedge \dbar h_1 \wedge \ddb h_1 \leq C r^2\omega_{co,0}^2.
\]
In total, we have
\[
{\rm (I)} \leq Cr^2\omega_{co,0}^2
\]

\par $\bullet$ Term ${\rm (II)}$.  Again, by homogeneity we have
 \[
 0 \leq \sqrt{-1}\del r^3\wedge \dbar r^3 \leq C r^2 \omega_{co,0},
 \]
 while the bound $|\phi-h_1| \leq C r^3$ yields a bound for the first term in ${\rm (II)}$ 
 \[
 \begin{aligned}
 \sqrt{-1}\del r^3\wedge \dbar r^3 \wedge \left((\phi-h_1)(2p^*\omega + \ddb( \phi + h_1)\right)&\leq r^2\omega_{co,0} \wedge r^3\omega_{sm}\\
 & \leq C r^4 \omega_{co,0}^2
 \end{aligned}
 \]
 where we used Lemma~\ref{lem: wsmwcoComp}, $(ii)$. The second term in ${\rm (II)}$ can be treated directly by scaling.  We have
\[
\sqrt{-1}\del r^3\wedge \dbar r^3 \wedge h_1 \ddb h_1 \leq C r^3\omega_{co,0}^2.
\]
In total, we have
 \[
 {\rm (II)} \leq  C r^{3} \omega_{co,0}^2
 \]

 \par $\bullet$ Term ${\rm (III)}$ can be treated similarly to term $({\rm II})$.  The first term can be estimated as
 \[
 \ddb r^3 \wedge(\phi-h_1)(2p^*\omega + \ddb( \phi + h_1))  \leq C r^3 \omega_{sm}^2 \leq C r^2\omega_{co,0}^2,
 \]
thanks to Lemma~\ref{lem: wsmwcoComp} $(i)$ again.   The homogeneous term is easily estimated as
 \[
 \ddb r^3 \wedge h_1\ddb h_1\leq C r^2\omega_{co,0}^2.
 \] 
 In summation, we have proved
 
 \begin{lem}\label{lem: PsiRlowBnd}
 There is a constant $C>0$ so that on the region $\{R^{-1}<r<2R^{-1}\}$ we have
 \[
 \Psi_{R} \geq -C R^3r^{2}\omega_{co,0}^2 \geq -CR\omega_{co,0}^2.
 \]
 \end{lem}
 
At this point we consider
 \[
 \Psi_{R} + C_0C_{R}^2 S_{R}^*\Omega.
 \]
In order for this form to be positive we need to show that $C_0, C_{R}$ can be chosen consistently.  To do this we consider the conditions for positivity in four different regions.  In the following $C$ will denote a uniform constant which can change from line to line, but depends only on the fixed background data and is, in particular, independent of $R, C_0, C_{R}$.
 \smallskip
 
 \par \noindent $\bullet$ {\bf Region $\{0 <r < R^{-1}\}$}.
 
\noindent In this region we have $\sigma =0$ and $S_{R}^*\chi(r^2)=R^{2}r^2$, so that
 \[
  \Psi_{R} + C_0C_{R}^2 S_{R}^*\Omega = \frac{4}{9}C_0C_{R}^2R^{4}\omega_{co,0}^2 >0
  \]
  so this region contributes no restriction.
 
\par \noindent $\bullet$ {\bf Region $\{R^{-1}<r<2R^{-1}\}$}.
 
\noindent Thanks to Lemma~\ref{lem: PsiRlowBnd}, and the fact that $\chi(s)=s$ for $s\in[0,4]$ we have
 \[
 \Psi_{R} \geq -CR\omega_{co,0}^2 \qquad C_0C_{R}^2 S_{R}^*\Omega = \frac{4}{9}C_{0}C_{R}^2R^{4}\omega_{co,0}^2.
 \]
Thus, in order to ensure positivity we need $C_0 >3C$, $C_{R}^2R^{4} = R$, so we must take $C_{R}^2=R^{-3}$.
\smallskip

\par \noindent $\bullet$ {\bf Region $\{ 2R^{-1}<r < 1\}$}.

 \noindent By definition we have
\[
\Psi_{R} = \omega_{CY}^2 \geq C \omega_{sm}^2.
\]
On the other hand whenever $\chi'' <0$, Lemma~\ref{lem: OmegaRbnd} gives the estimate
\[
C_0C_{R}^2 S_{R}^*\Omega \geq \frac{4}{9}C_0R(2\chi'(R^2r^2)+ (Rr)^2 \chi''(R^2r^2))\omega_{co,0}^2.
\]
 Now, from Lemma~\ref{lem: wsmwcoComp} we have $ \omega_{co,0}^2 \leq Cr^{-2}\omega_{sm}^2$ and so we can ensure positivity provided 
 \[
 \frac{4}{9}C_0Rr^{-2}(2\chi'(R^2r^2)+ (Rr)^2 \chi''(R^2r^2)) \geq -C.
 \]
 If we can constant $\chi$ so that $\chi' \geq 0$, then when $\chi''\geq 0$ we have $S_{R}^*\Omega \geq 0$ so no additional restriction is contributed from this case.

 \par \noindent $\bullet$ {\bf Region $\{ 1<r \}$}.

\noindent In this region we take $\chi(s)= {\rm const.}$, and so, by definition 
 \[
  \Psi_{R} + C_0C_{R}^2 S_{R}^*\Omega =\omega_{CY}^2
  \]
 
 In summary,  $\Psi_{R} + C_0C_{R}^2 S_{R}^*\Omega$ will be positive definite provided we can construct a cut-off function $\chi(s)$ (upon defining $s= R^2r^2$) with the following properties: for $0<\epsilon \ll 1$ given, we have
 \begin{itemize}
 \item $\chi(s)=s$ for $s\in[0,4]$, and $\chi'(s) \geq 0$,
  \item $\chi(s)$ is constant for $s\geq R^2$,
 \item $\chi(s)$ satisfies
 \[
 \frac{1}{s}\left(2\chi'(s) + s\chi''(s)\right) \geq -\frac{\epsilon}{R^3}
 \]
 for $s\in[4, R^2]$

 \end{itemize}
 
A cut-off function with these properties is constructed in \cite[Lemma 2.2]{FLY}, but for the readers convenience we give a slightly different proof here.
 
 \begin{lem}
 For $R\gg1$ sufficiently large there exists a smooth function $\chi(s)$ with the following properties
 \begin{itemize}
 \item[$(i)$] $\chi(s)=s$ for $s\in[0,4]$, and $\chi'(s) \geq 0$.
 \item[$(ii)$] $\chi(s)$ is constant for $s> R^2$,
 \item[$(iii)$] there is a uniform constant $C'$ such that
 \[
  \frac{1}{s}\left(2\chi'(s) + s\chi''(s)\right) \geq -\frac{C'}{R^4}
  \]
 
  \end{itemize}
  \end{lem}
  \begin{proof}
  Let $v=\chi'$.  Then we need to find $v$ such that $v(s)=1$ for $s\in[0,4]$, $v(s)=0$ for $s>R^2$, and 
  \[
  \frac{1}{s^2}\frac{d}{ds}(s^2v) \geq -\frac{C'}{R^4}.
  \]
  Consider 
  \[
  v(s) =\begin{cases} 1, & s\in [0,5]\\
   as^{-3}+bs^{-2}+cs^2+d, & s \in [5, R^2-1]\\
   0, & s \geq R^{2}-1.
   \end{cases}
   \]
   Demanding that $v(s)$ is $C^{1}$ gives a system of $4$ equations in the unknowns $a,b,c,d$, whose solutions are given by
   \[
   a= -250 +O(R^{-2}), \quad b= 75+O(R^{-2}), \quad c=75 R^{-8} +O(R^{-10}), \quad d=-150R^{-4} + O(R^{-6}).
   \]
   On the other hand, if $w(s)= C_{\alpha}s^{\alpha}$, then
   \[
    \frac{1}{s^2}\frac{d}{ds}(s^2w) = (\alpha+2)C_{\alpha}s^{\alpha-1} = \begin{cases} \geq 0 & \text{ if } \alpha > -2, \text{ and } C_{\alpha} \geq 0\\
     0 & \text{ if } \alpha =-2\\
     \geq 0 & \text{ if } \alpha < -2, \text{ and } C_{\alpha} \leq 0.
     \end{cases}
    \]
Since the terms corresponding to $a,b,c$ fall into one of these cases, it follows that
\[
  \frac{1}{s^2}\frac{d}{ds}(s^2v) \geq  -\frac{300}{R^{4}} +O(R^{-6}).
  \]
 It only remains to check that $v(s) \geq 0$ for $s\in[4, R^{2}]$.  To see this, observe that
 \[
 0= v'(\lambda R^{2}) \Longleftrightarrow \lambda^{5}-\lambda = O(R^{-2}).
 \]
Since the equation $\lambda^5-\lambda=0$ has only three real solutions $\lambda = \pm 1, 0$, it follows that for $R$ sufficiently large $v'(s)=0$ has only two real solutions on $[4, R^2]$.  Since these solutions are given by $s=4$ and $s= R^{2}-1$, it follows that $v'(s)\ne0$ in $(4,R^{2}-1)$, and hence $v$ does not have an interior minimum in $[4,R^2]$.  This immediately implies $v(s) \geq 0$.

Finally, let $\hat{v}$ be the result of convolving $v$ with a positive, symmetric mollifier with sufficiently small support.  Then $\hat{v}$ is smooth and has the same properties as $v$.  Integrating $\hat{v}$ yields $\chi$.

  \end{proof}

Now let us explain the extension of this construction to the metrics $\omega_{co,a}$ on the small resolution. For this, recall that
\[
\omega_{co, a} = \ddb f_{a}(\|z\|^2) + 4a^2\pi^*\omega_{FS}
\]
where $f_{a}(x) = a^2 f_{1}(\frac{x}{a^3})$ and $f_{1}$ satisfies
\[
(xf_{1}')^3+6(xf_1')^2=x^2.
\]
Rewriting this equation in terms of the variable $z$ defined by $x=z^{-3/2}$ and applying standard ODE techniques we obtain the following result, whose proof we leave to the reader
\begin{lem}\label{lem: wcoaUnif}
For $x\gg 1$, the function $f_1(x)$ has a convergent expansion
\[
f_1(x) = \frac{3}{2}x^{2/3} -2\log(x) + \sum_{n=0}^{\infty}c_nx^{-2n/3}.
\]
In particular, $f_{a}(x)\rightarrow \frac{3}{2}x^{2/3}$ smoothly and uniformly on any compact set as $a \rightarrow 0$.
\end{lem}

We now describe how to glue $\omega_{co,a}$ to the Calabi-Yau metric $\omega_{CY}$ to obtain a balanced metric.  We first recall that
\[
\Omega_{R} = C_0R^{-3}\ddb(\chi(R^{2}r^2) R^{2}\ddb r^2) = C_0R^{-1}\ddb(\chi(R^{2}r^2)\ddb r^2)
\]
On the other hand, we have
\[
\begin{aligned}
\omega_{co,a}^2 &= (\ddb f_a(\|z\|^2))^2 + 8a^2\ddb f_{a}(\|z\|^2)\wedge \pi^*\omega_{FS}\\
&= \ddb\left( f_{a}(\|z\|^2) \left(\ddb f_a(\|z\|^2) + 8a^2 \pi^*\omega_{FS}\right)\right)
\end{aligned}
\]
This suggests that we define
\[
\Omega_{R,a} = C_{0} \frac{2R^{-1}}{3} \ddb\left(\chi\left(\frac{2R^2}{3}f_{a}(\|z\|^2)\right)\left(\ddb f_a(\|z\|^2) + 8a^2 \pi^*\omega_{FS}\right)\right)
\]
\begin{lem}
For $0\leq a\ll 1$ sufficiently small there is an open set $U$ containing the $(-1,-1)$ rational curve such that
\[
\Omega_{R,a} = C_{0}R \omega_{co, a}^2
\]
\end{lem}
\begin{proof}
We only need to observe that the formula holds whenever
\[
\frac{2R^{2}}{3}f_{a}(\|z\|^2) \leq 4.
\]
Now since $f_{a}(\|z\|^2)$ converges uniformly to $\frac{3}{2}r^2$ on compact sets away from the $(-1,-1)$ rational curve, this inequality will hold for $a$ sufficiently small provided $r< 2R^{-1}$.
\end{proof}

Next, recall that the gluing of $\omega_{co,0}$ and $\omega_{CY}$ depending on only two estimates.

\par $\bullet$ The bounds
\[
 \Psi_{R} \geq -CR\omega_{co,0}^2, \qquad C_0R^{-3} S_{R}^*\Omega = \frac{4C_{0}}{9}R\omega_{co,0}^2
 \]
in the region $\{R^{-1}<r<2R^{-1}\}$.  Since $\Omega_{R,a}$ converges uniformly to $\omega_{co,0}^2$ on this region, the same bound holds with $\omega_{co,0}$ replaced by $\omega_{co,a}^2$, after possibly changing the constants.
\par $\bullet$ The bound
\[
\omega_{co,0}^2 \leq Cr^{-2}\omega_{sm}^2
\]
in the region $\{ 2R^{-1}<r < 1\}$.  Again, from the uniform convergence of $\omega_{co,a}$ to $\omega_{co,0}$ this bound holds, up to possibly increasing $C$ for $\omega_{co,a}^2$ as well, from the uniform convergence. 

It follows that the gluing procedure used to glue $\omega_{co,0}^2$ to $\omega_{CY}^2$ carries over in exactly the same way to glue $\omega_{co,a}^2$ to $\omega_{CY}^2$ for $0<a\ll 1$.  Furthermore, from Lemma~\ref{lem: wcoaUnif} we obtain the smooth, uniform convergence of $\omega_{a}$ to $\omega_{0}$ on compact sets away from the $(-1,-1)$ rational curves.

\bigskip
\bigskip

\end{document}